\documentclass[11pt,letterpaper]{amsart}
\usepackage{url,amsmath,amsthm, amssymb}
\usepackage[margin=1in]{geometry}
\usepackage[pdfencoding=auto, psdextra, colorlinks=true, linkcolor = blue, urlcolor=blue, citecolor=blue,breaklinks]{hyperref}
\usepackage{bookmark}
\usepackage[all]{mathtools}
\usepackage{makecell}
\usepackage{mathrsfs}
\usepackage{mathdots}
\usepackage{amsxtra,amscd}
\usepackage{amsrefs}
\usepackage{hyperref}
\usepackage{enumerate, enumitem}
\usepackage{arydshln}
\usepackage{multirow}
\usepackage{bbm} %This is for making nice matrices
\usepackage[all,arc]{xy}

 \usepackage{amssymb,amsrefs}
\usepackage{amsmath}
\usepackage{amscd,latexsym}
\usepackage{bookmark}
 \usepackage{amssymb,amsfonts,bm}
\usepackage[all,arc]{xy}
\usepackage{enumerate}
\usepackage{mathrsfs}
\usepackage{amscd}
\usepackage{tikz}
\usepackage{tikz-cd}
\usepackage{amsmath}
\usepackage{latexsym}
\usepackage{mathdots}
\usepackage{amsrefs}
\usepackage{graphicx}
\usepackage{mathtools}
\usepackage{leftidx}
\usepackage{tensor}
\usepackage{dsfont}
\usepackage{tikz-cd} 
\usepackage[new]{old-arrows}
\usepackage{xcolor}
\usepackage{bbm}
\usepackage{enumitem}
\usepackage{cleveref}

\allowdisplaybreaks

\theoremstyle{plain}
\newtheorem{theorem}{Theorem}[section]

\newtheorem{corollary}[theorem]{Corollary}
\newtheorem{induction hypothesis}[theorem]{Induction Hypothesis}

\newtheorem{lemma}[theorem]{Lemma}
\newtheorem{proposition}[theorem]{Proposition}
\numberwithin{equation}{section}

\theoremstyle{definition}
 
  \newtheorem*{data availability}{Data availability}
  \newtheorem*{conflicts of interest}{Conflicts of interest}

\theoremstyle{remark}
\newtheorem{remark}[theorem]{Remark}

\newtheorem*{acknowledgements}{Acknowledgments}

\makeatletter
\newcommand{\mytag}[2]{%
  \text{#1}%
  \@bsphack
  \begingroup
    \@onelevel@sanitize\@currentlabelname
    \edef\@currentlabelname{%
      \expandafter\strip@period\@currentlabelname\relax.\relax\@@@%
    }%
    \protected@write\@auxout{}{%
      \string\newlabel{#2}{%
        {#1}%
        {\thepage}%
        {\@currentlabelname}%
        {\@currentHref}{}%
      }%
    }%
  \endgroup
  \@esphack
}

\newcommand{\cComplex}{\mathbb{C}}

\newcommand{\zIntegers}{\mathbb{Z}}
\newcommand{\multiplicativegroup}[1]{#1^{\times}}

\newcommand{\Hom}{\mathrm{Hom}}
\newcommand{\sizeof}[1]{\left|#1\right|}
\newcommand{\abs}[1]{\left|#1\right|}
\newcommand{\fFiniteField}{\mathbb{F}}
\newcommand{\eFiniteField}{\mathbb{E}}
\newcommand{\fLocalField}{F}
\newcommand{\eLocalField}{E}

\newcommand{\bilinearPairing}[2]{\left\langle#1,#2\right\rangle}
\newcommand{\fieldCharacter}{\psi}

\newcommand{\whittaker}{\mathcal{W}}
\newcommand{\GL}{\operatorname{GL}}
\newcommand{\weyllong}{w}
\newcommand{\repBesselFunction}[1]{\mathcal{B}_{#1,\fieldCharacter}}
\newcommand{\grepBesselFunction}[2]{\mathcal{B}_{#1,#2}}

\newcommand{\contragredient}[1]{\check{#1}}
\newcommand{\fourierTransform}{\mathcal{F}_{\fieldCharacter}}
\newcommand{\ifourierTransform}{\mathcal{F}_{\fieldCharacter^{-1}}}
\newcommand{\gfourierTransform}[1]{\mathcal{F}_{#1}}
\newcommand{\Trace}{\operatorname{Tr}}

\newcommand{\qfFieldCharacter}{\fieldCharacter_{\eFiniteField \slash \fFiniteField}}
\newcommand{\qFieldCharacter}{\fieldCharacter_{\eLocalField \slash \fLocalField}}
\newcommand{\integersRing}{\mathfrak{o}}
\newcommand{\matrixRing}{\mathcal{M}}
\newcommand{\complexConjugate}[1]{\overline{#1}}
\newcommand{\Schwartz}{\mathcal{S}}

\newcommand{\nilpotentMatrices}{\mathcal{N}}
\newcommand{\multiplicativeMeasure}{d^{\times}}
\newcommand{\projection}{\operatorname{pr}}
\newcommand{\BFgammaFactor}[2]{\gamma(#1,\mathrm{BF}, #2)}
\newcommand{\BFpigammaFactor}{\BFgammaFactor{\pi}{\fieldCharacter}}
\newcommand{\BFDualpigammaFactor}{\BFgammaFactor{\contragredient{\pi}}{\fieldCharacter^{-1}}}
\newcommand{\BFlocalgammaFactor}[3]{\Gamma(#1, #2,\mathrm{BF}, #3)}
\newcommand{\BFLocalRhogammaFactor}{\BFlocalgammaFactor{s,t}{\rho}{\fieldCharacter_\fLocalField}}
\newcommand{\Volume}{\operatorname{Vol}}
\newcommand{\AsaiPiGammaFactor}{\gamma(\pi,\mathrm{As},\fieldCharacter)}
\newcommand{\AsaiDualPiGammaFactor}{\gamma(\contragredient{\pi},\mathrm{As},\fieldCharacter^{-1})}
\newcommand{\AsaiLocalRhoGammaFactor}{\Gamma(s,\rho,\mathrm{As},\fieldCharacter_\fLocalField)}

\newcommand{\TrivialRepresentation}{{\bf 1}}
\newcommand{\GaloisGroup}{\operatorname{Gal}}
\newcommand{\maximalCompactTimesCenter}{\mathcal{K}}
\newcommand{\realPart}{\operatorname{Re}}
\newcommand{\Support}{\operatorname{Supp}}
\newcommand{\liftedBesselFunction}[2]{\mathrm{B}_{#1,#2}}
\newcommand{\UnipotentSubgroup}{N}
\newcommand{\MirabolicSubgroup}{P}

\begin{document}

\title{Finite period vectors and Gauss sums}

\author{Yeongseong Jo}
\address{Department of Mathematics Education, Ewha Womans University, Seoul 03760, Republic of Korea}
\email{\href{mailto:jo.59@buckeyemail.osu.edu}{jo.59@buckeyemail.osu.edu};\href{mailto:yeongseong.jo@ewha.ac.kr}{yeongseong.jo@ewha.ac.kr}}

%\date{}

\begin{abstract}
We study four sums including the Jacquet--Piatetski-Shapiro--Shalika, Flicker, Bump--Friedberg, and Jacquet--Shalika sums associated to irreducible cuspidal representations
of general linear groups over finite fields. By computing explicitly, we relate Asai and Bump--Friedberg gamma factors over finite fields to those over nonarchimedean local fields through
level zero supercuspidal representation. Via Deligne--Kazhdan close field theory, we prove that exterior square and Bump--Friedberg gamma factors agree with corresponding Artin gamma factors of their associated tamely ramified representations through local Langlands correspondence. We also deduce product formul$\text{\ae}$ for Asai, Bump--Friedberg, and exterior square gamma factors in terms of Gauss sums. By combining these results, we examine Jacquet--Piatetski-Shapiro--Shalika, Flicker--Rallis, Jacquet--Shalika, and Friedberg--Jacquet periods and vectors and their connections to
Rankin-Selberg, Asai, exterior square, and Bump-Friedberg gamma factors, respectively.
\end{abstract}

\maketitle

\section{Introduction}
In a classical analytic number theory and related branches of mathematics, one of the main themes is to analyze a complex valued arithmetic function called a {\it Dirichlet character}.
One of its prominent property is what is known as the {\it functional equation} establishing the symmetry across the critical strip. Viewing the Euler's $\Gamma$-function as the $L$-factor in the archimedean context, the symmetric version of this global functional equation involves the epsilon function which can be presented as the product of the {\it classical Gauss sum} and the conductor. In the 1960's, the analytic paradigm for understanding Dirichlet characters shifted from real or complex analytic functions to the study of automorphic forms on ${\rm GL}_n$ and automorphic representations of ${\rm GL}_n$.
This naturally leads us to ask ourselves whether Gauss sums that appear in the global epsilon function have a representation theoretic interpretation.

\par
Since the representation of nonarchimedean local fields $F$ occurs as sub-quotients of representations inducing from tensor product of cuspidal representations, it is not so surprising for us to notice the imitation of such a interpretation for a pair of supercuspidal 
representations $\rho_1$ and $\rho_2$ of  ${\rm GL}_n(F)$ and ${\rm GL}_r(F)$. In particular, when $\rho=\rho_1$ and $\rho_2={\bf 1}_{F^{\times}}$ becomes the trivial representation of ${\rm GL}_1(F)$, the formula given in \cite{BF85,Bus87} defines
the {\it Godement--Jacquet gamma factor} $\Gamma(s,\rho,\psi_F)$ \cite{GJ72} in terms of {\it non-abelian Gauss sums}, where $\psi_F$ is  a fixed non-trivial additive character of $\fLocalField$. The identical Gauss sum emerges in {\it Tate's local gamma factor} for $n=1$, and in the seminal book of Bushnell and Henniart \cite{BH06}, albeit for $n=2$. In the twisted case, we find the explicit formula for {\it Rankin--Selberg gamma factor} $\Gamma(s,\rho_1 \times \rho_2,\psi_F)$ only adhering to
the conductor of the local constant \cite{BHK98}.

\par
Regarding many questions about representations of nonarchimedean local fields, oftentimes insights can be gained by inspecting the analogue question over a finite field  $\fFiniteField_q$.
 Before the Rankin--Selberg gamma factor $\Gamma(s,\rho_1 \times \rho_2,\psi_F)$ was established in the pioneering work of Jacquet--Piatetski-Shapiro--Shalika (cf. \cite[Theorem 1.2]{Nien19}), the parallel gamma factors
 $\gamma(\pi_1 \times \pi_2,\psi)$ associated to 
a pair of irreducible generic representations $\pi_1$ and $\pi_2$ of ${\rm GL}_n(\fFiniteField_q)$ and ${\rm GL}_r(\fFiniteField_q)$ had already been investigated in Piatetski-Shapiro's unpublished note \cite{PS76}. The finite gamma factor $\gamma(\pi_1 \times \pi_2,\psi)$, where now $\pi_2$ is a multiplicative character of $\fFiniteField_q^{\times}$, is revisited by Nien \cite{Nie17,Nien19} in the hope of resolving local converse theorem and distinction problems \cite{Nie19}, following the lead from nonarchimedean local fields situation. As a by-product, Nien does something more intriguing, namely that  $\gamma(\pi_1 \times \pi_2,\psi)$ is associated to the {\it abelian Gauss sum} \cite[Theorem 1.1]{Nie17}.

\par
In this article, we put the Asai, Bump--Friedberg, and exterior square setting on an equal footing with the Rankin--Selberg setting by expressing such finite gamma factors in terms of abelian Gauss sums.
We summarize our main results concerning irreducible cupsidal representation $\pi$ of ${\rm GL}_n(\fFiniteField_q)$, or ${\rm GL}_n(\fFiniteField_{q^2})$ if necessary, as follows;

\begin{itemize}[leftmargin=*]
\item In \Cref{Asai-Multiplicity}, the {\it Asai gamma factor} $\AsaiPiGammaFactor$ is defined as a proportionality of bilinear forms arising from the {\it Flicker sums} given by \eqref{Def-Flicker-Sum}.
We prove the prodcut formula  for $\AsaiPiGammaFactor$ in terms of abelian Gauss sums (\Cref{Asai-Gamma-Gauss}).
\item The {\it exterior square gamma factor} $\gamma(\pi,\wedge^2,\psi)$ is defined as a proportionality of bilinear forms arising from the {\it Jacquet--Shalika sums} given in \eqref{Jacquet-Shalika-even} and \eqref{Jacquet-Shalika-odd}. We prove the product formula for $\gamma(\pi,\wedge^2,\psi)$ in terms of abelian Gauss sums (\Cref{Exterior-Gamma-Gauss}).
\item In \Cref{BF-Multiplicity}, the {\it Bump--Friedberg gamma factor} $\BFpigammaFactor$ is defined as a proportionality of bilinear forms arising from the {\it Bump--Friedberg sums} given by \eqref{Def-BF-Even-Sum}
and \eqref{Def-BF-Odd-Sum}. This represents $\varepsilon_0(\varphi,\fieldCharacter)\varepsilon_0(\wedge^2 \circ \varphi,\fieldCharacter)$, the product of Deligne's arithmetic standard $\varepsilon_0$-factor and Deligne's arithmetic exterior square $\varepsilon_0$-factor  (\Cref{Deligne-factorazation}). We prove the product formula for $\BFpigammaFactor$ in terms of abelian Gauss sums (\Cref{Exterior-Gamma-Gauss}).
\end{itemize}
Returning to the motivating question over the finite field  $\fFiniteField_q$, Nien and Zhang \cite[Conjecture 2.2]{NZ18} subsequently propose the conjectural formula with regard to abelian Gauss sums
for Rankin--Selberg gamma factor $\gamma(\pi_1 \times \pi_2,\psi)$ for a pair of irreducible cuspidal representations $\pi_1$ and $\pi_2$ of ${\rm GL}_n(\fFiniteField_q)$ and ${\rm GL}_r(\fFiniteField_q)$ with different ranks $n \neq r$. A slightly modified formula is settled by Yang \cite[Theorem 0.1]{Yan19} and, independently, by Ye--Zelingher \cite[Corollary 4.5]{YZ21}. Afterwords, Zelingher generalizes the method to  $\gamma(\pi_1 \times \pi_2,\psi)$ with the same rank $n=r$ \cite[Theorem 2.18]{Ze22}.

\par
A key strategy of establishing the explicit formula boils down to computing gamma factors for level zero supercuspidal representations.
When it comes to an irreducible cuspidal representation $\pi$ (or a pair of irreducible cuspidal representations $\pi_1$ and $\pi_2$ on occasion), which does not possess suitable period vectors, we can further relate them with their counterparts for irreducible cuspidal representations over finite fields. The calculation of Rankin--Selberg gammas factors $\Gamma(s,\rho_1 \times \rho_2,\psi_F)$
for a pair of irreducible cuspidal representations $\rho_1$ and $\rho_2$ is attributed to Ye \cite{Ye19}, and later Ye and Zelingher  \cite{YZ20} performed exterior square gamma factors $\Gamma(s,\rho,\wedge^2,\fieldCharacter_\fLocalField)$. However the computation for the Asai gamma factor and the Bump--Friedberg gamma factor is newly explored in this paper.
Let us give a precise statement pertaining to level zero supercuspidal representations $\rho$ of ${\rm GL}_n(F)$, or ${\rm GL}_n(E)$ with $E$ an unramified quadratic extension over $F$ if necessary, here.

\begin{itemize}[leftmargin=*]
\item The {\it Asai gamma factor} $\AsaiLocalRhoGammaFactor $ satisfies the local functional equation given in \eqref{local-asai-func}. We prove in \Cref{levelzero-Asai} that this is equal to 
the rational function
\[
q^{n\left(s-1/2\right)}\omega^{-1}_{\rho}(\varpi) L(n(1-s),\omega^{-1}_{\rho}\restriction_{\multiplicativegroup{\fLocalField}})/L(ns,\omega_{\rho}\restriction_{\multiplicativegroup{\fLocalField}})
\]
 in $\mathbb{C}(q^{-s})$ if $n=2m+1$ and $\pi$ has a {\it Flicker--Rallis vector}, and to a complex number $\AsaiPiGammaFactor$ otherwise.
\item The {\it Bump--Friedberg factor} $\BFLocalRhogammaFactor$ satisfies the local functional equation given in \eqref{Local-BF-Func}. We prove in \Cref{levelzero-BF}
that  this equals to the rational function
\[
 \varepsilon(s,\rho,\fieldCharacter_\fLocalField)q^{m\left(2s-\frac{1}{2}\right)}\omega^{-1}_{\rho}(\varpi) L(m(1-2s),\omega^{-1}_{\rho})/L(2ms,\omega_{\rho})
 \]
 in $\mathbb{C}(q^{-2s})$ if $n=2m$ and $\pi$ has a {\it Friedberg--Jacquet vector}, and to a complex number $\BFpigammaFactor$ otherwise.
\end{itemize}

Another main ingredient toward the product formula is to associate analytic gamma factors over finite fields with the corresponding Deligne's arithmetic $\varepsilon_0$-factor. 
An instant benefit from such definition is that arithmetic $\varepsilon_0$-factors inherit multiplicative property which in turn make it feasible to express as product of Gauss sums \cite[Theorem 2.4]{YZ21}.
Part of reasons that Ye and Zelingher \cite[\S 5]{YZ21} primarily considered the product formula for Rankin--Selberg factor $\gamma(\pi_1 \times \pi_2,\psi)$ is 
that the matching between Jacquet--Shalika gamma factor $\Gamma(s,\rho(\varphi),\wedge^2,\fieldCharacter_\fLocalField)$ and Artin exterior square gamma factor $\Gamma(s,\wedge^2 ( \varphi),\psi_F)$
under the local Langlands correspondence was not available at that time. Another aim of this paper is to take up this issue and remove the constant ambiguity $``c_f"$ lingering in \cite[\S 1]{YZ21}.
We present an accurate statement regarding the identity in the following way:

\begin{itemize}[leftmargin=*]
\item Let $\varphi$ be a $n$-dimensional tamely ramified representation of $W_\fLocalField$ corresponding to the level zero supercuspidal representation $\rho(\varphi)$ of $\GL_n(\fLocalField)$ via local Langlands correspondence. We prove in \Cref{equal-zero-exterior} and \Cref{BF-JS equal} that
\begin{itemize}[leftmargin=*]
\item $ \Gamma(s,\rho(\varphi),\wedge^2,\fieldCharacter_\fLocalField)=\Gamma(s,\wedge^2( \varphi),\fieldCharacter_\fLocalField)$.
\item $\BFlocalgammaFactor{s,t}{\rho(\varphi)}{\fieldCharacter_\fLocalField} =\varepsilon(s+t+1/2,\varphi,\fieldCharacter_\fLocalField)\Gamma(2s,\wedge^2( \varphi),\fieldCharacter_\fLocalField)$.
\end{itemize}
\end{itemize}

In late 1980's, Bump and Friedberg predicted that $\BFlocalgammaFactor{s,t}{\rho(\varphi)}{\fieldCharacter_\fLocalField}$
is a product of the exterior square $\gamma$-factor and the standard $\gamma$-factor, based on the pattern in a spherical situation.
We partially confirm their conjecture \cite[Conjecture 4]{BF89} for level zero supercuspidal reprsentations. 
Our important tactic here is to utilize a globalization of level zero supercuspidal representations over local function fields equipped with a close field theory. 
In a series of versions of globalizing supercuspidal representations over number fields, one typically loses control of the local component of 
a cuspidal automorphic representation exactly at one place, an archimedean place. While the Langlands--Shahidi
theory over archimedean local fields has been fairly well navigated since the seminal work of Shahidi \cite{Sha85}, there has been little progress on the desired archimedean input for
Bump--Friedberg and Jacquet--Shalika integrals. However the globalization of level zero supercuspidal representations in positive characteristic gives rather good control at all places.
 Although one may sacrifice a few places, the necessary equalities of exterior square $\gamma$-factors for irreducible constituents of spherical representations at those bad places (Lemmata \ref{spherical-rep-JS}, \ref{spherical-rep-BF})
 do not appear to be insurmountable.

\par
Having a solid matching of $\gamma$-factors for level zero supercuspidal representations over positive characteristic (Theorems \ref{exter-positive}, \ref{main-BF-factor}) in hand, we then incorporate $\gamma$-factors arising from Langlands--Shahidi methods and integral representations with Deligne--Kazhdan theory over close (nonarchimedean local) fields. 
 Deligne proved that Artin local factors remain the same for parallel representations over close local fields via  Deligne isomorphisms \cite[Proposition 3.7.1]{Del84}. An analogous result has been studied by Ganapathy and Lomel\'{\i} \cite{Gan15,GL15} for Langlands--Shahidi local factors on analytic sides, though that time they only considered Kazhdan isomorphisms over sufficiently closed fields.
So far, there seems to have been no previous discourse in the literature about the very basic case of ``1-close fields". To be precise, we will show that local exterior square $\gamma$-factors for level zero supercuspidal representations via Jacquet-Shalika integrals are compatible with 
 the Kazhdan correspondence over $1$-close fields (Propositions \ref{Kaz-corres}, \ref{close-field-BF}). 
 In doing so, we are allowed to transport the identity of $\gamma$-factors over positive characteristic to characteristic zero. 

\par
The fact that the pole of local $L$-functions characterizes the existence of linear functionals for supercuspidal representations
has been used to great impetus in recent years on constructing multiplicative relations of $L$-factors \cite{Jo20,Jo24,Mat15,Mat17}. In particular, for a supercuspidal representation $\rho_1$ of ${\rm GL}_n(F)$ and an irreducible admissible generic representation $\rho_2$ of ${\rm GL}_n(F)$, the local Rankin--Selberg $L$-function $L(s,\rho_1 \times \rho_2)$ is an important tool to detect 
whether $\rho_1$ appears in the standard module that defies $\check{\rho}_2$. This comes down to the observation that a twist of $\rho_1$ occurs in the standard module for  $\check{\rho}_2$
if and only if local $L$-functions $L(s,\rho_1 \times \rho_2)$ has a pole of which the location determines that unramified twist \cite{CPS17}.
Lately, Soudry and Zelingher \cite[Theorem 1.3]{SZ23} suggest that the absolute value of the normalized finite Rankin--Selberg gamma factor $\gamma^{\star}(\pi_1 \times \pi_2,\fieldCharacter)$ might serve as an alternative for the order of the pole of local $L$-functions $L(s,\rho_1 \times \rho_2)$. These results parallel analogous results of Cogdell and Piatetski-Shapiro in the nonarchimedean local field setting  \cite{CPS17}. As long expected, for a pair of level zero supercuspidal representations $\rho_1$ and $\rho_2$ of  ${\rm GL}_n(F)$, we are able to show that the existence of poles of $L(s,\rho_1 \times \rho_2)$ indeed contributes to the absolute value of $\gamma^{\star}(\pi_1 \times \pi_2,\fieldCharacter)$ other than one, and vice versa.
More precisely, we list below the periods and vectors of interest provided in \Cref{Period-Vector-Integral}, which naturally arise from certain local $L$-functions and particular finite gamma factors.

\begin{itemize}[leftmargin=*]
\item The occurrence of {\it Jacquet--Piatetski-Shapiro--Shalika periods} and {\it vectors} is equivalent to having a pole of the Rankin--Selberg $L$-factor $L(s,\rho_1 \times \rho_2)$, 
or the absolute value of the {\it normalized} Rankin-Selberg gamma factor $\gamma^{\star}(\pi_1 \times \pi_2,\fieldCharacter)$ different to one.
\item The occurrence of {\it Flicker--Rallis periods} and {\it vectors}  is equivalent to having a pole of the Asai $L$-factor $L(s,\rho,{\rm As})$,
or the absolute value of the Asai gamma factor $\AsaiPiGammaFactor$ different to one.
\item The occurrence of {\it Jacquet--Shalika periods} and {\it vectors} is equivalent to have a pole of the exterior square $L$-factor $L(s,\rho,\wedge^2)$, or the absolute value of
the exterior square gamma factor $\gamma(\pi,\wedge^2,\psi)$ different to one.
\item The occurrence of {\it Friedberg--Jacquet periods} and {\it vectors} is equivalent to having a pole of the Bump--Friedberg $L$-factor $L(s,\rho,{\rm BF})$, or the absolute value of
the Bump--Friedberg gamma factor $\BFpigammaFactor$ different to one.
\end{itemize}

Via the theory of new forms for ${\GL}_n$, the parallel analogue results in the nonarchimedean setting is completed by the author \cite{Jo23}, and the method is further generalized to
archimedean setting by the joint work with Humphries \cite{HJ23}. 
We can further explorer the absolute value of gamma factors to irreducible generic representations $\pi$ of ${\rm GL}_n(\fFiniteField_q)$, not just limited to irreducible cuspidal ones.
Soudry and Zelingher recently worked it out for Rankin--Selberg gamma factor $\gamma(\pi_1 \times \pi_2,\fieldCharacter)$ over finite fields \cite[Theorem 1.3]{SZ23}.
Along the line of this philosophy, it is likely that so-called ``multiplicativity" of gamma over finite fields needs to be well-understood ahead of time. 
In the case of $p$-adic fields, it was the Langlands--Shahidi method that came to fruition first (cf. \cite{Lom16,Lom17}) and enabled Zelingher \cite{Ze22} to realize the goal of establishing multiplicativity via Shahidi gamma factors.
It turns out that long before the computation was investigated by Soudry, when he was a graduate student in 1979 \cite{SZ23}.
In a joint work of Soudry and Zelingher, the finite field analogue of Shahidi gamma factors for pairs $(\pi_1,\pi_2)$ is completed very recently \cite{SZ23}.
However it is still required to show that Shahidi gamma factors agrees with Rankin--Selberg gamma factor $\gamma(\pi_1 \times \pi_2,\fieldCharacter)$ to validate this robust argument. 
The author has carried out a similar matching question for the Asai gamma factor $\AsaiPiGammaFactor$, which will appear elsewhere. 
In conjunction with computing Rankin--Selberg sums for classical groups, the author currently pursues this topic in depth with Zelingher \cite{Ze22}. Although all these analysis seem to be doable, if somewhat involved, we think that our current formulation keeps our exposition a reasonable length and we plane to do so in near future.

\par
The structure of this article is the following. \Cref{RS factor} contains a brief review of the theory of Jacquet--Piatetski-Shapiro--Shalika sums, and Rankin--Selberg gamma factors. 
We continue to survey Deligne--Kazhdan close field theory and give its application to Rankin-Selberg gamma factors.
\Cref{AS factor}, and \Cref{BF factor} are devoted to presenting an analogue theory for Asai, and Bump--Friedberg gamma factors, orderly. We concern with exterior square gamma factor
in \Cref{Exterior factor}, where these results are all essentially analogue, although several of the results regarding close field theory are mostly addressed.
We discuss the relation between period vectors, integrals, absolute values of gamma factors, and poles of $L$-factors in \Cref{period}.

\section{The Rankin-Selberg Gamma Factor}
\label{RS factor}

We now detail the theory of Jacquet--Piatetski-Shapiro--Shalika sums as well as the relation between Jacquet--Piatetski-Shapiro--Shalika vectors and Rankin--Selberg $\gamma$-factors.
The results herein are all well-known. However, the normalization of Haar measure \cite{Ze22}, the choice of subspace of Schwartz--Bruhat functions, and the Fourier transformation  \cite{SZ23,Ye19,YZ20}
in the current literature are slightly different from this paper. We recall them as motivation for Sections \ref{AS factor} and \ref{BF factor},
in which we discuss analogous yet new results for Asai and exterior square $\gamma$-factors. Section \ref{RS factor} serves to
overview the breakdown of our computation that is repeated throughout the paper.

\subsection{The Jacquet--Piatetski-Shapiro--Shalika sum}

We let $N_n$ be the unipotent radical of the standard Borel subgroup $B_n$ of ${\rm GL}_n$ and $A_n$ the Levi subgroup of $B_n$, consisting of diagonal matrices in ${\rm GL}_n$.
We denote by $P_n$ the mirabolic subgroup of ${\rm GL}_n$, consisting of matrices in ${\rm GL}_n$ with last row equal to $(0,\dotsm,0,1)$.
We write $1_n$ to denote the $n \times n$ identity matrix.
Let $\fFiniteField_q$ be a finite field of $q=p^k$ elements with characteristic $p$.
Let $\fFiniteField=\fFiniteField_q$. We fix a non-trivial additive character $\fieldCharacter:=\fieldCharacter_\fFiniteField$ of $\fFiniteField$, and extend it to the character of $N_n(\fFiniteField)$
by setting $\psi(n)=\psi(n_{1,2}+\dotsm+n_{n-1,n})$ with $n \in N_n(\fFiniteField)$. Let 
$\Schwartz(\fFiniteField^n)=\{ \phi \,|\, \phi:  \fFiniteField^n \rightarrow \cComplex \}$
be the set of complex valued functions on $\fFiniteField^n$. 
Let $\{e_i \,|\, 1 \leq i \leq n \}$ be the standard row basis of $\fFiniteField^n$.  
We let $\bilinearPairing{x}{y} :=x \cdot {^ty}$ be the standard bilinear form on $\fFiniteField^n$.
The Fourier transform of $\phi \in \Schwartz(\fFiniteField^n)$ with respect to $\fieldCharacter$ is given by (cf. \citelist{\cite{Nie19}*{(2.3)} \cite{YZ20}*{\S 2.1.1}}
\[
\fourierTransform(\phi)(y)=q^{-\frac{n}{2}} \sum_{x \in \fFiniteField_q} \phi(x) \fieldCharacter(\bilinearPairing{x}{y}).
\]
The Fourier inversion formula takes the form 
\[
(\fourierTransform \circ \fourierTransform)(\phi)(x)=\phi(-x) \quad \text{and} \quad (\ifourierTransform \circ \fourierTransform)(\phi)(x)=\phi(x).
\]

\par
Given an irreducible cuspidal representation $\pi$, we fix a non-trivial $\GL_n(\fFiniteField)$-invariant unitary form $(\cdot,\cdot)$ on $V_{\pi} \times V_{\pi}$. Then there exists a non-trivial vector $v_0 \in V_{\pi}$ satisfying
$\pi(n)v_0=\psi(n)v_0$ for all $n \in N_n(\fFiniteField)$. Such a vector $v_0$ is called a {\it Whittaker vector}.
A {\it Whittaker function} of $\pi$ is a matrix coefficient of the form $W(g)=(\pi(g)v,v_0)$ for all $g \in \GL_n(\fFiniteField)$ and $v \in V_{\pi}$.
Whittaker functions satisfy $W(ng)=\psi(n)W(g)$ for every $n \in N_n(\fFiniteField)$.
The subspace generated by all Whittaker function $W(g)$ is unique \cite[Theorem 0.5]{Gel70}, and will be denoted by $\mathcal{W}(\pi,\psi)$ with $\GL_n(\fFiniteField)$ acting by right translations.
This space is called the {\it Whittaker model} of $\pi$. 
For an irreducible cuspidal representation $\pi$ of $\GL_n(\fFiniteField)$, we have its contragredient $\check{\pi}$ that is isomorphic to $\pi^{\iota}$,
where $\pi^{\iota}$ is the representation acting on the same underlying space $V_{\pi}$ of $\pi$ by $\pi^{\iota}(g)=\pi(^tg^{-1})$ for $g \in \GL_n(\fFiniteField)$.
Under $\check{\pi} \cong \pi^{\iota}$, we achieve an isomorphism of vector spaces $\whittaker(\pi,\fieldCharacter) \rightarrow \whittaker(\check{\pi},\fieldCharacter^{-1})$,
given by $W_{\pi}  \mapsto  \check{W}_{\pi}$, where
\begin{equation}
\label{dual-Whittaker-def}
 \check{W}_{\pi}(g)=W_{\pi}(\weyllong_n\, {^tg^{-1}}), \quad g \in \GL_n(\fFiniteField),
\end{equation}
and where 
$w_n=\begin{pmatrix}  \vspace{-2ex} && 1 \\ & \iddots& \vspace{-2ex} \\ 1 && \end{pmatrix}$
 is the longest Weyl element of $\GL_n(\fFiniteField)$. 
The {\it Bessel function} $\mathcal{B}_{\pi,\psi}$ of $\pi$ is the Whittaker function of the normalized Whittaker vector by setting
\[
  \mathcal{B}_{\pi,\psi}(g)=\frac{(\pi(g)v,v_0)}{(v_0,v_0)}=W(1_2)^{-1}W(g).
\]
Let $\omega_{\pi}$ denote the central character of $\pi$ on $\multiplicativegroup{\fFiniteField}$. Some of the elementary properties of the Bessel function $\mathcal{B}_{\pi,\psi}$ are (cf. \cite[Proposition 2.8]{Nie19}):
\begin{enumerate}[label=$(\mathrm{\arabic*})$]
\item $\mathcal{B}_{\pi,\psi}(n_1gn_2)=\psi(n_1)\psi(n_2)\mathcal{B}_{\pi,\psi}(g)$ for all $n_1,n_2 \in N_n(\fFiniteField)$.
\item  $\mathcal{B}_{\pi,\psi}(1_n)=1$ and $\mathcal{B}_{\pi,\psi}(a 1_n)=\omega_{\pi}(a)$ for all $a \in \multiplicativegroup{\fFiniteField}$.
\item $\mathcal{B}_{\pi,\psi}(g^{-1})=\overline{\mathcal{B}_{\pi,\psi}(g)}=\mathcal{B}_{\check{\pi},\psi^{-1}}(g)$ for all $g \in \GL_n(\fFiniteField)$.
\end{enumerate}

\par
Let $\pi_1$ and $\pi_2$ be irreducible cuspidal representations of $\GL_n(\fFiniteField)$. 
Let $\Schwartz_0(\fFiniteField^n)$ denote the set of $\cComplex$-valued functions on $\fFiniteField^n$ such that $\phi(0)=0$. For every $W_{\pi_1} \in \whittaker(\pi_1,\fieldCharacter)$ and $W_{\pi_2} \in \whittaker(\pi_2,\fieldCharacter^{-1})$,
and for any $\phi \in \Schwartz_0(\fFiniteField^n)$, there exists a complex number $\gamma^{\star}(\pi_1 \times \pi_2,\fieldCharacter)$ such that
\begin{equation}
\label{Finite-J-PS-S-Func}
\gamma^{\star}(\pi_1 \times \pi_2,\fieldCharacter) \sum_{g \in \UnipotentSubgroup_n(\fFiniteField) \backslash \GL_n(\fFiniteField) } W_{\pi_1}(g) W_{\pi_2}(g) \phi(e_ng)=
\sum_{g \in \UnipotentSubgroup_n(\fFiniteField) \backslash \GL_n(\fFiniteField) } W_{\pi_1}(g) W_{\pi_2}(g) \fourierTransform(\phi)(e_1 {^tg^{-1}}).
\end{equation}

\begin{remark}
As explained in \cite[p.23 footnote]{Nie19}, we normalize the Fourier transform and gamma factors, which are different from what is commonly adopted at least since Roddity-Gershon’s master thesis (cf. \citelist{\cite{Ye19}*{\S 2.2}\cite{SZ23}*{\S 2.3.1}}). In order to distinguish the normalized gamma factor from unnormalized one $\gamma(\pi_1 \times \pi_2,\fieldCharacter)$ in \cite[\S 2.2]{Ye19}, we add the superscript $\star$ to emphasize the normalization. In doing so, the absolute value of $ \gamma^{\star}(\pi_1 \times \pi_2,\fieldCharacter) $ becomes $1$, as shown in  \Cref{RS-Func}. 
\end{remark}

For $W_{\pi_1} \in \mathcal{W}(\pi_1,\psi)$ and $W_{\pi_2} \in \mathcal{W}(\pi_2,\psi^{-1})$, and $\phi \in \Schwartz(\fFiniteField^n)$, we define the {\it Jacquet--Piatetski-Shapiro--Shalika sum} $\Psi(W_{\pi_1},W_{\pi_2},\phi)$ by
\[
  \Psi(W_{\pi_1},W_{\pi_2},\phi):=\sum_{g \in \UnipotentSubgroup_n(\fFiniteField) \backslash \GL_n(\fFiniteField) } W_{\pi_1}(g) W_{\pi_2}(g) \phi(e_ng).
\]
In a similar manner, the  {\it dual Jacquet--Piatetski-Shapiro--Shalika sum} $\check{\Psi}(W_{\pi_1},W_{\pi_2},\phi)$ is defined by
\[
\check{\Psi}(W_{\pi_1},W_{\pi_2},\phi):=\sum_{g \in \UnipotentSubgroup_n(\fFiniteField) \backslash \GL_n(\fFiniteField) } \check{W}_{\pi_1}(g)\check{W}_{\pi_2}(g) \fourierTransform(\phi)(e_ng).
\]
A non-zero vector $v_1 \otimes v_2 \in V_{\pi_1} \otimes V_{\pi_2}$ is called a {\it Jacquet--Piatetski-Shapiro--Shalika vector}, if for every $g \in \GL_n(\fFiniteField)$,
we have $(\pi_1 \otimes \pi_2)(g)(v_1 \otimes v_2):= \pi_1(g) v_1 \otimes \pi_2(g)v_2 =v_1 \otimes v_2$. This condition is equivalent to $\pi_1 \cong \check{\pi}_2$.

\begin{remark} The definition of $\gamma^{\star}(\pi \times \contragredient{\pi},\fieldCharacter)$ taken in \cite{YZ20} differs slightly from herein by the inclusion of all Schwartz-Bruhat functions $\phi \in \Schwartz(\fFiniteField^n)$ in the Jacquet--Piatetski-Shapiro--Shalika sum. As stressed in \cite[Corollary 2.27]{YZ20}, the functional equation in \eqref{Finite-J-PS-S-Func} does not hold when $\pi_1 \times \pi_2$ has the Jacquet--Piatetski-Shapiro--Shalika vector. In a pioneering work \cite{PS76} (cf. \cite[\S 2.2]{Ye19}), Piatetski-Shapiro has already realized that it is advantageous to restrict the space $\Schwartz(\fFiniteField^n)$ to $\Schwartz_0(\fFiniteField^n)$. In order to define gamma factors over the finite field uniformly for every irreducible cuspidal representations, Schwartz-Bruhat functions $\phi$ are only taken over  
$\Schwartz_0(\fFiniteField^n)$ throughout the paper.
\end{remark}

We express $\gamma^{\star}(\pi_1 \times \pi_2,\fieldCharacter)$ in terms of the Bessel functions associated with $\pi_1$ and $\pi_2$.

\begin{theorem}\citelist{\cite{Ye19}*{Equation (16)}\cite{SZ23}*{Corollary 3.4}} 
\label{J-PS-S-Besel}
Let $\pi_1$ and $\pi_2$ be irreducible cuspidal representations of $\GL_n(\fFiniteField)$. Then
\begin{equation}
\label{J-PS-S-Bessel-sum}
 \gamma^{\star}(\pi_1 \times \pi_2,\fieldCharacter)= q^{-\frac{n}{2}} \sum_{g \in \UnipotentSubgroup_n(\fFiniteField)\backslash \GL_n(\fFiniteField)}   \grepBesselFunction{\pi_1}{\fieldCharacter}(g)   \grepBesselFunction{\pi_2}{\fieldCharacter^{-1}}(g)  \fieldCharacter(e_1 {^tg^{-1}}\;{^te_n}).
\end{equation}
In particular, we have 
$ \gamma^{\star}(\contragredient{\pi}_1 \times \contragredient{\pi}_2,\fieldCharacter^{-1})=\complexConjugate{ \gamma^{\star}(\pi_1 \times \pi_2,\fieldCharacter)}$.
\end{theorem}

We can precisely evaluate the sum \eqref{J-PS-S-Bessel-sum} in two different ways, when $\pi_1 \times \pi_2$ has the Jacquet--Piatetski-Shapiro--Shalika vector.
 The method of Ye \cite[Corollary 4.3]{Ye19} uses the system of linear equations and the theory of level zero supercuspidal representations, whereas
Soudry and Zelingher \cite[Theorem A.1]{SZ23} explicitly compute $\gamma^{\star}(\pi \times \contragredient{\pi},\fieldCharacter)$
within the context of the representation theory of groups over finite fields. We transport the former approach to the Bump--Friedberg setting in \Cref{BF-distinct-Gamma},
while the latter path is adapted to the Asai setting in \Cref{Asai-gamma-distinction} and the exterior square setting in \Cref{JS-distinct-Gamma}.

\begin{theorem}\citelist{\cite{SZ23}*{Theorem A.1}\cite{Ye19}*{Corollary 4.3}} 
\label{self-contra-comp}
Let $\pi$ be an irreducible cuspidal representation of $\GL_n(\fFiniteField)$.
Then we have
\[
   \gamma^{\star}(\pi \times \contragredient{\pi},\fieldCharacter)=-q^{-\frac{n}{2}}.
\]
\end{theorem}

We end this section by summarizing functional equations for $\gamma^{\star}(\pi_1 \times \pi_2,\fieldCharacter)$ over finite fields.
The result is a direct consequence of \Cref{J-PS-S-Besel} and \Cref{self-contra-comp}.

\begin{theorem}$(${\rm cf.} \cite[Corollary 2.2 and Proposition 2.5]{SZ23}$)$
\label{RS-Func}
 Let $\pi_1$ and $\pi_2$ be irreducible cuspidal representations of $\GL_n(\fFiniteField)$.
\begin{enumerate}[label=$(\mathrm{\arabic*})$]
\item If $\pi_1 \ncong \check{\pi}_2$,  we have $ \gamma^{\star}(\pi_1 \times \pi_2,\fieldCharacter) \gamma^{\star}(\contragredient{\pi}_1 \times \contragredient{\pi}_2,\fieldCharacter^{-1})=1$ and $\abs{\gamma^{\star}(\pi_1 \times \pi_2,\fieldCharacter)}=1$.
\item  If $\pi_1 \cong \check{\pi}_2$, we have $ \gamma^{\star}(\pi_1 \times \pi_2,\fieldCharacter) \gamma^{\star}(\contragredient{\pi}_1 \times \contragredient{\pi}_2,\fieldCharacter^{-1})=q^{-n}$ and $\abs{\gamma^{\star}(\pi_1 \times \pi_2,\fieldCharacter)}=q^{-\frac{n}{2}}$.
 \end{enumerate}
\end{theorem}

\subsection{The Jacquet--Piatetski-Shapiro--Shalika period and level zero supercuspidal representations}
Let $F$ be a non-archimedean local field with its residual finite field $\integersRing / \mathfrak{p} \cong \mathbb{F}_q$ of order $q=q_F$.
The base field $F$ is a finite extension of $\mathbb{Q}_p$ or $\mathbb{F}_p((t))$, called a {\it $p$-adic field} in characteristic 0,
or a {\it local function field} in characteristic $p > 0$. 
We write  $\integersRing:=\integersRing_F$ and $\mathfrak{p}:=\mathfrak{p}_F$ for the ring of its integers and the maximal ideal, respectively. 
We fix a generator $\varpi:=\varpi_F$ of $\mathfrak{p}$ and normalize the absolute value $|\cdot|$ of $F$ so that
$|\varpi|=q^{-1}$.
Let $\fieldCharacter_\fLocalField$ be a fixed non-trivial additive character of $\fLocalField$ such that $\fieldCharacter_\fLocalField$ is trivial on $\mathfrak{p}$ and nontrivial on $\integersRing$.
The self-dual Haar measure for $\fieldCharacter_\fLocalField$ \cite[\S 23]{BH06} then satisfies
\[
 \int_{\integersRing} dx=q^{\frac{1}{2}}.
\]
For the purpose of calculation, it will be convenient to choose the Haar measure $\multiplicativeMeasure x$ on $\multiplicativegroup{\fLocalField}$ such that
\[
  \int_{\multiplicativegroup{\integersRing}} \multiplicativeMeasure x=1.
\]
We denote by $\projection : \integersRing \rightarrow \integersRing/\mathfrak{p} \cong \fFiniteField$ the quotient map.
We define $\fieldCharacter(\projection(k))=\fieldCharacter_\fLocalField(k)$ for $k \in \integersRing$. 
Let $\mathcal{S}(F^n)$ be the space of locally constant and compactly supported functions $\Phi : F^n \rightarrow \mathbb{C}$.
We denote its Fourier transform
\[
  \gfourierTransform{\fieldCharacter_\fLocalField}(\Phi)(y)=\int_{\fLocalField^n} \Phi(x) \fieldCharacter_\fLocalField(\bilinearPairing{x}{y}) \, dx,
\]
for $\Phi \in \Schwartz(\fLocalField^n)$.
The Fourier inversion formulas are given by
\[
(\gfourierTransform{\fieldCharacter_\fLocalField} \circ \gfourierTransform{\fieldCharacter_\fLocalField})(\Phi)(x)=\Phi(-x) \quad \text{and} \quad (\gfourierTransform{\fieldCharacter_\fLocalField^{-1}} \circ \gfourierTransform{\fieldCharacter_\fLocalField})(\Phi)(x)=\Phi(x).
\] 
For $\phi \in \Schwartz(\fFiniteField^n)$, we define a lift $\Phi_{\circ} \in \Schwartz(\fLocalField^n)$ of $\phi$ by
\[
  \Phi_{\circ}(x)=
  \begin{cases}
  \phi(\projection(x)), & \text{if $x \in \integersRing^n$,}\\
  0, & \text{otherwise.}\\
  \end{cases}
\]
Then $\gfourierTransform{\fieldCharacter_\fLocalField}(\Phi_{\circ})$ is a lift of $\fourierTransform(\phi)$ \cite[Proposition 3.6]{YZ20} in the sense of
\[
 \gfourierTransform{\fieldCharacter_\fLocalField}(\Phi_{\circ})(x)=
  \begin{cases}
 \fourierTransform( \phi)(\projection(x)), & \text{if $x \in \integersRing^n$,}\\
  0, & \text{otherwise.}\\
  \end{cases}
\]

\par
 We let $K_n=\GL_n(\mathfrak{o})$ be the standard maximal compact subgroup of $\GL_n(F)$. A {\it level zero supercuspidal representation} of  $\GL_{n}(\fLocalField)$ is given by
\[
 \rho \cong {\rm c\text{-}Ind}^{{\rm GL}_n(F)}_{F^{\times}{\rm GL}_n(\mathfrak{o})} \widetilde{\mu},
\]
where a representation $\mu:=\widetilde{\mu}\restriction_{{\rm GL}_n(\mathfrak{o})}$ is inflated from an irreducible cuspidal representation $\pi$ of $\GL_n(\fFiniteField)$
and the central character $\omega_{\pi}$ of $\pi$ satisfies $\widetilde{\mu}\restriction_{\mathfrak{o}^{\times}=F^{\times} \cap {\rm GL}_n(\mathfrak{o})}(a)=\omega_{\pi}(\projection(a))$.
We let $\mathcal{A}_0({\rm GL}_n(F))$ be the set of isomorphism classes of level zero supercuspidal representations of ${\rm GL}_n(F)$ and let $\mathcal{A}_0({\rm GL}_n(\fFiniteField))$
denote the set of equivalence classes of irreducible cuspidal representations of ${\rm GL}_n(\fFiniteField)$. Then \cite[\S 3.3]{Yan19} and \cite[Theorem 3.5]{YZ20}  give rise to a bijection
\begin{equation}
\label{level-zero-corresp}
\begin{split}
 \mathcal{A}_0({\rm GL}_n(F)) &\longleftrightarrow \mathbb{C}^{\times} \times \mathcal{A}_0({\rm GL}_n(\fFiniteField)) \\
\rho & \longleftrightarrow (\omega_{\rho}(\varpi),\pi).
\end{split}
\end{equation}
The contragredient representation $\check{\rho}$ of ${\rm GL}_n(F)$ is again a level zero supercuspidal representation
constructed from an irreducible cuspidal representation $\check{\pi}$ of ${\rm GL}_n(\fFiniteField)$ \cite[Lemma 2.1]{Ye19}.
If $W_{\rho} \in \whittaker(\rho,\psi_F)$, then $\check{W}_{\rho}(g):=W_{\rho}(\weyllong_n\, {^tg^{-1}}) \in \whittaker(\check{\rho},\psi^{-1}_F)$.
We denote by $\langle \cdot, \cdot \rangle$ the pairing $V_{\contragredient{\rho}} \times V_{\rho}$ and $V_{\contragredient{\pi}} \times V_{\pi}$ given by the evaluation.
Let $\lambda \in \Hom_{\UnipotentSubgroup_n(\fFiniteField)}(\pi,\fieldCharacter)$ be a non-zero Whittaker functional of $\pi$.
We define the linear functional $\lambda_{\circ} : V_{\rho} \rightarrow \cComplex$ by
\[
  \langle \lambda_{\circ}, f \rangle:=\int_{\UnipotentSubgroup_n(\fLocalField) \cap K_n  \backslash \UnipotentSubgroup_n(\fLocalField)}\langle \lambda, f(u)  \rangle \fieldCharacter^{-1}(u) \,\multiplicativeMeasure u,
\]
where $f \in V_{\rho}$. We view $f$ as a function $f : \GL_n(\fLocalField) \rightarrow V_{\pi}$. Then $\lambda_{\circ} \in  \Hom_{\UnipotentSubgroup_n(\fLocalField)}(\rho,\fieldCharacter_\fLocalField)$ is a non-zero Whittaker function of $\GL_n(\fLocalField)$,
which is a lift of $\lambda$ \cite[Proposition 3.7]{YZ20}. Let $W_{\pi} \in \whittaker(\pi,\fieldCharacter)$ and $v_{W_{\pi}} \in V_{\pi}$ a unique vector such that $W_{\pi}(g)=\langle \lambda, \pi(g)v_{W_{\pi}} \rangle$ for every $g \in \GL_n(\fFiniteField)$. Let $\maximalCompactTimesCenter:=\multiplicativegroup{\fLocalField}K_n$. We define $f_{W_{\pi}}$ to be
\[
f_{W_{\pi}}(g)=
 \begin{cases}
 \omega_{\rho}(a) \pi(\projection(k))v_{W_{\pi}}, & \text{if $g=ak \in \maximalCompactTimesCenter$ with $a \in \multiplicativegroup{\fLocalField}, k \in K_n$}  \\
 0, & \text{otherwise}.
 \end{cases}
\]
We define $W^{\circ}_{\rho} \in \whittaker(\rho,\fieldCharacter_\fLocalField)$ by $W^{\circ}_{\rho}(g)=\langle \lambda_{\circ}, \rho(g) f_{W_{\pi}} \rangle$ for $g \in \GL_n(\fLocalField)$.
Then the support of $W^{\circ}_{\rho}$ is contained in $\UnipotentSubgroup_n(\fLocalField)\multiplicativegroup{\fLocalField}K_n$. The two Whittaker functions $W^{\circ}_{\rho}$ and $W_{\pi}$ are related by 
\[
 W^{\circ}_{\rho}(g)=\omega_{\rho}(a) W_{\pi}(\projection(k)),
\]
for any $g=ak \in \maximalCompactTimesCenter$ with $a \in \multiplicativegroup{\fLocalField}$ and $k \in K_n$ \cite[Proposition 3.9]{YZ20}. In particular, if $W_{\pi}$ is a Bessel function $\repBesselFunction{\pi}$ over finite fields $\fFiniteField$,
the lift $W^{\circ}_{\rho}$ is nothing but the {\it Paskunas--Stevens partial Bessel function} $\liftedBesselFunction{\rho}{\fieldCharacter_\fLocalField}$ in  \cite[\S 3.4]{NZ18} and \cite[Theorem 5.8]{PSt08}.

\par
Let $G$ be a group and $L$ a subgroup of $G$. A representation $\rho$ of $G$ is called $(L,\xi)$-{\it distinguished} if 
\begin{equation}
\label{distinction}
  {\rm Hom}_L(\rho,\xi) \neq 0.
\end{equation}
If $\xi={\bf 1}_L$ is a trivial character of $L$, we simply say that $\rho$ is $L$-{\it distinguished}. In particular, let $G=\GL_{n}(\fLocalField) \times \GL_{n}(\fLocalField)$ and $L=\GL_{n}(\fLocalField)$ embedded in $G$ diagonally. We also say that  $\rho_1 \times \rho_2$ has a {\it Jacquet--Piatetski-Shapiro--Shalika period} if the representation $\rho_1 \times \rho_2$ of $G$ is
$(\GL_{n}(\fLocalField),{\bf 1}_{\GL_{n}(\fLocalField)})$-distinguished, and it is equivalent to the condition that $\rho_1 \cong \check{\rho}_2$. Because of this property, these distinguished representations appears naturally in the theory of Rankin-Selberg $L$-functions that we describe in a moment.

\par
Let $\rho_1$ and $\rho_2$ be level zero supercuspidal representations of $\GL_n(F)$ constructed from irreducible cuspidal representations $\pi_1$ and $\pi_2$ of  $\GL_n(\fFiniteField)$ with attached Whittaker models, $\mathcal{W}(\rho_1,\psi_F)$ and $\mathcal{W}(\rho_2,\psi^{-1}_F)$, respectively.  We take each pair of Whittaker functions $W_{\rho_1} \in \mathcal{W}(\rho_1,\psi_F)$, $W_{\rho_2} \in \mathcal{W}(\rho_2,\psi^{-1}_F)$, and a Schwartz-Bruhat function  $\Phi \in \Schwartz(\fLocalField^n)$,
and form the {\it Jacquet--Piatetski-Shapiro--Shalika integral} defined by 
\[
\Psi(s,W_{\rho_1},W_{\rho_2},\Phi)=\int_{\UnipotentSubgroup_n(F) \backslash \GL_n(F)} W_{\rho_1}(g) W_{\rho_2}(g) \Phi(e_ng)|\det g|^s \,dg.
\]
The integral converges absolutely for ${\rm Re}(s)$ sufficiently large, and extend meromorphically to the entire complex plane. Moreover there exists a rational function $ \Gamma(s,\rho_1 \times \rho_2,\fieldCharacter_\fLocalField) \in \mathbb{C}(q^{-s})$ satisfying the functional equation \cite[Proposition 4.2]{Mat15}:
\[
\Psi(1-s,\check{W}_{\rho_1},\check{W}_{\rho_2},\gfourierTransform{\fieldCharacter_\fLocalField}(\Phi))=  \Gamma(s,\rho_1 \times \rho_2,\fieldCharacter_\fLocalField)  \Psi(s,W_{\rho_1},W_{\rho_2},\Phi).
\]
It is worthwhile to emphasize that the gamma factor $\Gamma(s,\rho_1 \times \rho_2,\fieldCharacter_\fLocalField)$ defined above differs by a sign from the traditional one defined by Jacquet--Piatetski-Shapiro--Shalika  in \cite[\S 7.1]{PSt08}. The {\it local Rankin-Selberg $L$-function} $L(s,\rho_1 \times \rho_2)$ is the generator of the $\mathbb{C}[q^{\pm s}]$-fractional ideal of the Jacquet--Piatetski-Shapiro--Shalika integrals
$ \Psi(s,W_{\rho_1},W_{\rho_2},\Phi)$ with $W_{\rho_1} \in \mathcal{W}(\rho_1,\psi_F)$, $W_{\rho_2} \in \mathcal{W}(\rho_2,\psi^{-1}_F)$, and $\Phi \in \Schwartz(\fLocalField^n)$
normalized to be of the form $P(q^{-s})^{-1}$ for some $P(X) \in \mathbb{C}[X]$ with $P(0)=1$.

\par
We take pairs of Whittaker functions $W_{\rho_1}=W^{\circ}_{\rho_1}$ and $W_{\rho_2}=W^{\circ}_{\rho_2}$, which are lifts of corresponding pairs of Whittaker functions $W_{\pi_1}$ and $W_{\pi_2}$ 
 over finite fields, and we insert certain test functions $\Phi_{\circ}$, the lift of $\phi$, for Schwartz–Bruhat functions $\Phi$. With lifting datum $(W^{\circ}_{\rho_1},W^{\circ}_{\rho_2},\Phi_{\circ})$, 
 Jacquet--Piatetski-Shapiro--Shalika integrals can be reduced to Jacquet--Piatetski-Shapiro--Shalika sums, we obtain so-called {\it modified functional equations} just as in \cite[\S 3.3]{YZ20}.
\begin{multline*}
 \check{\Psi}(W_{\pi_1},W_{\pi_2},\phi)+q^{-n(1-s)}(\omega_{\rho_1}\omega_{\rho_2})^{-1}(\varpi)\fourierTransform(\phi)(0)L(n(1-s),(\omega_{\rho_1}\omega_{\rho_2})^{-1})\Psi(W_{\pi},\TrivialRepresentation_{\fFiniteField^n})\\
 = \Gamma(s,\rho_1 \times \rho_2,\fieldCharacter_\fLocalField)(\Psi(W_{\pi_1},W_{\pi_2},\phi)+q^{-ns}\omega_{\rho_1}\omega_{\rho_2}(\varpi)\phi(0)L(ns,\omega_{\rho_1}\omega_{\rho_2})\Psi(W_{\pi_1},W_{\pi_2},\TrivialRepresentation_{\fFiniteField^n})).
\end{multline*}
As a result of the modified function equation, we recover the following main theorem of \cite{Ye19}.

\begin{theorem}
\label{levelzero-RS}
Let $\rho_1$ and $\rho_2$ be level zero supercuspidal representations of $\GL_{n}(\fLocalField)$.
\begin{enumerate}[label=$(\mathrm{\arabic*})$]
\item \cite[Theorem 4.1]{Ye19} If $\pi_1 \ncong \check{\pi}_2$, we have
\[
 \Gamma(s,\rho_1 \times \rho_2,\fieldCharacter_\fLocalField)=\gamma^{\star}(\pi_1 \times \pi_2,\fieldCharacter).
\]
\item \cite[\S4 Right after Corollary 4.3]{Ye19} If $\pi_1 \cong \check{\pi}_2$, we have
\[
 \Gamma(s,\rho_1 \times \rho_2,\fieldCharacter_\fLocalField)=q^{n\left(s-\frac{1}{2}\right)}(\omega_{\rho_1}\omega_{\rho_2})^{-1}(\varpi) \frac{L(n(1-s),(\omega_{\rho_1}\omega_{\rho_2})^{-1})}{L(ns,\omega_{\rho_1}\omega_{\rho_2})}.
\]
\end{enumerate}
\end{theorem}

\subsection{The Rankin-Selberg epsilon factor and the Gauss sum}
Let $\overline{\fFiniteField}$ be an algebraic closure of $\fFiniteField$ and $\widehat{\fFiniteField}^{\times}_{q^n}$ the multiplicative group of $\multiplicativegroup{\fFiniteField_{q^n}}$.
A multiplicative character $\alpha \in \widehat{\fFiniteField}^{\times}_{q^d}$ is called {\it regular} 
if $\{ \alpha, \alpha^q, \dotsm, \alpha^{q^{n-1}} \}$ is of size $n$. Two characters $\alpha$ and $\beta$ are called {\it equivalent} if $\alpha=\beta^{q^d}$ for some integer $d$. In next paragraph, we see that 
this amounts to saying that $\alpha$ and $\beta$ are in the same Frobenius orbit.
Let $\mathcal{R}_n(\fFiniteField_q)$ denote the set of equivalence classes of regular characters of $\fFiniteField^{\times}_{q^n}$. Green's parameterization \citelist{\cite{Yan19}*{3.2}\cite{Nie17}*{3.1}}
gives a bijection 
\[
\begin{split}
 \mathcal{A}_0({\rm GL}_n(\fFiniteField)  &\longleftrightarrow \mathcal{R}_n(\fFiniteField_q) \\
\pi & \longleftrightarrow \alpha.
\end{split}
\]
 For each $d\, |\, n$, we have a norm map ${\rm Nr}_{n:d}: \multiplicativegroup{\fFiniteField_{q^n}} \rightarrow \multiplicativegroup{\fFiniteField_{q^d}}$, which induces a dual (embedding) map $\widehat{\rm Nr}_{n:d} : \multiplicativegroup{\widehat{\fFiniteField}_{q^d}} \rightarrow  \multiplicativegroup{\widehat{\fFiniteField}_{q^n}}$
by assigning $\beta \in \multiplicativegroup{\widehat{\fFiniteField}_{q^d}}$ to $\beta \circ {\rm Nr}_{n:d} \in \multiplicativegroup{\fFiniteField_{q^n}}$. In this way, $({\multiplicativegroup{\widehat{\fFiniteField}_{q^n}})_{n \in \mathbb{N}}}$ with the embedding map $(\widehat{\rm Nr}_{n:d})_{d\, |\, n}$ forms a direct system. We denote by $\Omega:=\varinjlim  \multiplicativegroup{\fFiniteField_{q^n}}$ its direct limit.
\par

Let $W_F$ be a Weil group of $F$, $I_F$ the inertia subgroup, and $P_F$ the wild inertia subgroup. Then $W_F \cong I_F \rtimes \langle {\rm Fr} \rangle$,
where  $ {\rm Fr}  \in {\rm Gal}(\overline{\fFiniteField} / \fFiniteField)$ is the geometric Frobenius automorphism given by ${\rm Fr}(x^q)=x$ for every $x \in \overline{\fFiniteField}$.
The Frobenius map ${\rm Fr}$ acts on $\Omega$ via ${\rm Fr} \cdot \beta=\beta^q$. We identify $\multiplicativegroup{\fFiniteField_{q^n}}$ with the subgroup $\Omega_n:=\{\beta \in \Omega \,|\, {\rm Fr}^n \cdot \beta=\beta \}$ of $\Omega$. A Galois orbit is a set of the form $\mathcal{O}=\mathcal{O}(\beta):=\{ {\rm Fr}^i \cdot \beta \, |\, i \in \mathbb{Z} \}$ for $\beta \in \Omega$.
Given a Galois orbit $\mathcal{O}$, we define its {\it degree} ${\rm deg}(\mathcal{O})$ to be the cardinality of $\mathcal{O}$. Then for $\beta \in \mathcal{O}$, we have $\beta \in \Gamma_{{\rm deg}(\mathcal{O})}$. We denote by ${\rm Fr}\backslash \Omega$ the set of Galois orbits.
 Let $\fieldCharacter_n:=\fieldCharacter \circ {\Trace}_{\fFiniteField_{q^n} \slash \fFiniteField}$. We define the {\it Gauss sum} $\tau(\alpha,\fieldCharacter_n)$ by
\[
  \tau(\alpha,\fieldCharacter_n):=-\sum_{x \in \multiplicativegroup{\fFiniteField}_{q^n}} \alpha^{-1}(x) \fieldCharacter_n(x).
\]

\par
Let $\varphi : W_F \rightarrow {\rm GL}(V)$ be a $n$-dimensional Frobenius semisimple representation of the Weil group $W_F$. The representation $\varphi$ is said to be {\it unramified} (resp. {\it tamely ramified}),  
if  $\ker \varphi$ contains $I_F$ (resp. $P_F$). We let $r$ be an operation on the Frobenius semisimple representation of $W_F$ that preserves tame ramification. 
In particular, we take $r$ to be an identity operation, $id$, a tensor product, $\otimes$, a twisted tensor product, ${\rm As}$ (known as an {\it Asai} representation), and an exterior square, $\wedge^2$.
In a spirit of Deligne \cite[\S 4 and \S 5]{Del73}, $\varepsilon(s,r(\varphi),\fieldCharacter_\fLocalField)$ and $\varepsilon_0(r(\varphi),\fieldCharacter_\fLocalField)$ are related by
\begin{equation}
\label{rel-epsilon-epsilon0}
 \varepsilon(s,r(\varphi),\fieldCharacter_\fLocalField)=\varepsilon_0(r(\varphi),\fieldCharacter_\fLocalField)\det\left(-{\rm Fr},r(V)^{I_F}\right)^{-1} q^{\left(\dim r(V)^{I_F}\right)s}.
\end{equation}

Following the literature in \cite{Gan15,Yan19}, we set $\mathcal{G}^t({\rm GL}_n(F))$ to be the isomorphism classes of tamely ramified representations of $W_F$
of degree $n$. Since the local Langlands reciprocity map preserves the conductor and the depth of the representation \cite[Theorem 7.3]{Gan15},
the correspondence induces a natural bijective map ${\rm LLC} : \mathcal{A}_0({\rm GL}_n(F)) \rightarrow \mathcal{G}^t({\rm GL}_n(F))$ \cite[Appendix A]{SZ08} (cf. \cite[\S 7]{Gan15}).
Two Weil representations are called $I_F$-{\it equivalent} if their restrictions to $I_F$ are equivalent, and we write  $\mathcal{G}^t_{I}({\rm GL}_n(F))$ for the set of $I_F$-equivalence classes of 
tamely ramified $n$-dimensional representations of $W_F$. Long before the local Langlands correspondence is established, Macdonald \cite[Theorem 3.1]{YZ21} has already obtained
a canonical bijection $\mathcal{M}$ from $\mathcal{A}_0({\rm GL}_n(\fFiniteField)$ to  $\mathcal{G}^t_{I}({\rm GL}_n(F))$. Hence we get a diagram
\[
\xymatrixcolsep{3.5pc}
\xymatrixrowsep{1.5pc}
\xymatrix{
&\mathcal{A}_0({\rm GL}_n(F)) \ar[d]_{p_1} \ar[r]^{\rm LLC}_{\cong} & \mathcal{G}^t({\rm GL}_n(F)) \ar[d]^{p_2} \\
&\mathcal{A}_0({\rm GL}_n(\fFiniteField)) \ar[r]^{\mathcal{M}}_{\cong} &\mathcal{G}_I^t({\rm GL}_n(F))}
\]
where $p_1$ is a projection map induced from \eqref{level-zero-corresp}, and $p_2$ is a canonical projection map sending a representation to its $I_F$-equivalence class.
In particular, it is a consequence of \cite[Appendix A]{SZ08} that the above diagram is commutative. Composing the Macdonald correspondence $\mathcal{M}$ with the Green's parameterization then yields a bijection between $\mathcal{G}_I^t({\rm GL}_n(F))$ and $\mathcal{R}_n(\fFiniteField_q)$, which by abuse of teminology we again refer to as Green's parametrization.

\begin{theorem}\cite[Theorem 4.2]{YZ21}
\label{Rankin-Selberg-Gauss}
 Let $\varphi_1$ and $\varphi_2$ be $n$-dimensional tamely ramified representations of $W_F$. Then
\[
 \varepsilon_0(\varphi_1 \otimes \varphi_2,\fieldCharacter_\fLocalField)=(-1)^{n}q^{-\frac{n^2}{2}} \prod_{i=0}^{n-1} \tau(\alpha\beta^{q^{i}},\psi_n),
\]
where $\alpha$ and $\beta$ are regular characters of $\fFiniteField^{\times}_{q^n}$ corresponding to $\varphi_1$ and $\varphi_2$, respectively, via Green's parametrization.
\end{theorem}

Rankin-Selberg $\gamma$-factors and tensor product $\varepsilon_0$-factors over finite fields are compatible with the Macdonald correspondence.

\begin{proposition}\cite[Theorem 2.18]{Ze22} 
\label{TensorProduct-GaussProductFormula}
 Let $\pi_1(\varphi_1) $ and $\pi_2(\varphi_2)$ be irreducible cuspidal representations of $\GL_n(\fFiniteField)$ associated to $n$-dimensional tamely ramified representations $\varphi_1$ and $\varphi_2$ of $W_F$ via  Macdonald correspondence. Then we have
\[
 \gamma^{\star}(\pi_1(\varphi_1) \times \pi_2(\varphi_2),\fieldCharacter)=\omega^{n-1} _{\pi_2}(-1) \varepsilon_0(\varphi_1 \otimes \varphi_2,\fieldCharacter_\fLocalField).
\]
\end{proposition}

As a corollary of \Cref{Rankin-Selberg-Gauss} and \Cref{TensorProduct-GaussProductFormula}, we gain a product formula for $\gamma^{\star}(\pi_1 \times \pi_2,\fieldCharacter)$ with regard to Gauss sums.

\begin{theorem}[Gauss sum] Let $\pi_1$ and $\pi_2$ be irreducible cuspidal representations of $\GL_n(\fFiniteField)$. 
We let $\alpha$ and $\beta$ are regular characters of $\fFiniteField^{\times}_{q^n}$ corresponding to $\pi_1$ and $\pi_2$, respectively, via Green's parametrization.
Then we have
\[
 \gamma^{\star}(\pi_1 \times \pi_2,\fieldCharacter)=\omega^{n-1}_{\pi_2}(-1) \cdot (-1)^{n}   q^{-\frac{n^2}{2}} \prod_{i=0}^{n-1} \tau(\alpha\beta^{q^{i}},\psi_n).
 \]
\end{theorem}

\subsection{Deligne--Kazhdan close field theory}

We turn our attention to Deligne-Kazhdan close local field theory. Two non-archimedean local fields $F$ and $F'$ are {\it m-close} if $\mathfrak{o}_F/\mathfrak{p}_F^m \cong \mathfrak{o}_{F'}/\mathfrak{p}_{F'}^m$. For example, the fields $\mathbb{F}_p((t))$ and $\mathbb{Q}_p(p^{1/m})$ are $m$-close. We follow the elaboration about Deligne's theory in \cite[\S 2.1]{Gan15}
and \cite[\S 6.3]{GL15}.
If $F$ and $F'$ is $1$-close, Deligne (cf. \cite[\S 2.1]{Gan15}) gave a bijection:
\begin{multline}
\label{Del-corr}
\{ \text{Isomorphism classes of Frobenius semisimple representations $\varphi$ of $W_F$ trivial on $P_F$} \}\\
 \xleftrightarrow{ \rm Del}
 \{ \text{Isomorphism classes of Frobenius semisimple representations $\varphi'$ of $W_{F'}$ trivial on $P_{F'}$} \}.
\end{multline}
Elements $\varphi$ and $\varphi'$ are nothing but tamely ramified representations. The triplet $(F,\varphi,\psi_F)$ is said to be {\rm Del}-{\it associated to} $(F',\varphi',\psi'_{F'})$ if
 \begin{enumerate}[label=$(\mathrm{\alph*})$]
\item\label{Del-1} $F$ and $F'$ are $1$-close;
\item $\varphi'= {\rm Del}(\varphi)$;
\item an character $\psi'_{F'}$ of $F'$ satisfies ${\rm cond}(\psi'_{F'})=\mathfrak{p}_{F'}$ and the character induced by $\psi'_{F'}$ on $\mathfrak{o}_{F'} / \mathfrak{p}_{F'}$ 
coincides with that induced by $\psi_F$ on $\mathfrak{o}_F / \mathfrak{p}_F$ under the isomorphism implicit in \ref{Del-1}.
 \end{enumerate}

The analogous isomorphism of Deligne on the analytic side over close local fields has been studied by  Kazhdan \cite[\S 2.3]{Gan15}. 
We provide a revamped version of the Kazhdan isomorphism \cite[\S 6.2]{GL15}, which can be directly verified from \cite[Theorem 3.5]{YZ20}:
\begin{equation}
\label{Kaz-corr}
  \left\{ \begin{split} &\text{\quad\quad\; Level zero supercuspidal} \\ &\text{ representations $(\rho,V_{\rho})$ of ${\rm GL}_n(F)$} \end{split} \right\} 
  \xleftrightarrow{ \rm Kaz}   \left\{ \begin{split} &\quad\quad\quad \text{Level zero supercuspidal} \\ &\text{ representations $(\rho',V'_{\rho'})$ of ${\rm GL}_n(F')$} \end{split} \right\},
\end{equation}
where $\rho \cong {\rm c\text{-}Ind}^{{\rm GL}_n(F)}_{F^{\times}{\rm GL}_n(\mathfrak{o}_F)} \widetilde{\mu}$ and $\rho' \cong {\rm c\text{-}Ind}^{{\rm GL}_n(F')}_{{F'}^{\times}{\rm GL}_n(\mathfrak{o}_{F'})} \widetilde{\mu}'$ under the isomorphism $``\rm Kaz"$ satisfy
 \begin{enumerate}[label=$(\mathrm{\roman*})$]
  \item $\omega_{\rho}(\varpi_F)=\omega_{\rho'}(\varpi_{F'})$;
\item $\mu:=\widetilde{\mu}\restriction_{{\rm GL}_n(\mathfrak{o}_F)}$ and $\mu':=\widetilde{\mu}'\restriction_{{\rm GL}_n(\mathfrak{o}_{F'})}$
are an inflation of a common irreducible cuspidal representation $\tau$ via the canonical projections: 
\[
\xymatrixcolsep{4pc}
 \xymatrix{
 ({\rm GL}_n(\mathfrak{o}_F),\mu) \ar[r]^{{\rm   mod}\,\, \mathfrak{p}_F} & ({\rm GL}_n(\mathbb{F}_q),\pi) &\ar[l]_{\quad {\rm  mod}\,\, \mathfrak{p}_{F'}}
( {\rm GL}_n(\mathfrak{o}_{F'}),\mu').
 }
\]
 \end{enumerate}

 We say that the triplet $(F,\rho,\psi_F)$ is {\rm Kaz}-{\it associated to} $(F',\rho',\psi'_{F'})$ if
 \begin{enumerate}[label=$(\mathrm{\alph*})$]
\item\label{Kaz-1}  $F$ and $F'$ are $1$-close;
\item $\rho'= {\rm Kaz}(\rho)$;
\item an character $\psi'_{F'}$ of $F'$ satisfies ${\rm cond}(\psi'_{F'})=\mathfrak{p}_{F'}$ and the character induced by $\psi'_{F'}$ on $\mathfrak{o}_{F'} / \mathfrak{p}_{F'}$ 
coincides with that induced by $\psi_F$ on $\mathfrak{o}_F / \mathfrak{p}_F$ under the isomorphism implicit in \ref{Kaz-1}.
 \end{enumerate}
 
 Let $\pi_1$ be an irreducible cuspidal representation of ${\rm GL}_n(\mathbb{F})$ and $\pi_2$ an irreducible cuspidal representation of ${\rm GL}_r(\mathbb{F})$ 
 with $n > r$.  Then there exists a complex number $\gamma(\pi_1 \times \pi_2,\psi) \in \mathbb{C}$ such that
 \[
  \gamma(\pi_1 \times \pi_2,\psi) \sum_{g \in {N}_r(\mathbb{F}) \backslash {\rm GL}_r(\mathbb{F})} W_{\pi_1} \begin{pmatrix} g & 0 \\ 0 & 1_{n-r} \end{pmatrix} W_{\pi_2}(g)
  =\sum_{g \in {N}_r(\mathbb{F}) \backslash {\rm GL}_r(\mathbb{F})} W_{\pi_1} \begin{pmatrix} 0 & 1_{n-r} \\ g & 0 \end{pmatrix} W_{\pi_2}(g),
 \]
 for all $W_{\pi_1} \in \mathcal{W}(\pi_1,\psi)$ and $W_{\pi_2} \in \mathcal{W}(\pi_2,\psi^{-1})$ \cite[Theorem 3.4]{NZ18}. 
 Let $\rho_1$ be a level zero supercuspidal representation of ${\rm GL}_n(F)$ associated to $\pi_1$ and $\rho_2$ a level zero supercuspidal representation of ${\rm GL}_r(F)$
 associated to $\pi_2$ , with $n > r$. Let $\Gamma(s,\rho_1 \times \rho_2,\psi_F)$ denote the Rankin-Selberg $\gamma$-factor defined by Jacquet, Piatetski-Shapiro, and Shalika (cf. \cite[Theorem 3.8]{NZ18}).

 \begin{lemma}
 \label{level-zero-RS}
 For $(F,\rho_i,\psi_F)$ that is {\rm Kaz}-{\it associated to} $(F',\rho'_i,\psi'_{F'})$ with $i=1,2$, we have
 \[
  \Gamma(s,\rho_1 \times \rho_2,\psi_F)= \Gamma(s,\rho'_1 \times \rho'_2,\psi'_{F'}).
 \]
\end{lemma}

\begin{proof}
With aid of \cite[Theorem 3.11]{NZ18}, we can relate gamma factors for a pair of level zero supercuspidal representations with
those for the corresponding cuspidal representations over finite fields: 
\[
 \omega^{n-1}_{\rho_2}(-1)\Gamma(s,\rho_1 \times \rho_2,\psi_F)={\rm Vol}(\mathfrak{p}_F)^{r(n-r-1)} \gamma(\pi_1 \times \pi_2,\psi).
\]
Since we have normalized Haar measures on $F$ so that the volume of $\mathfrak{o}_F$ is $q^{1/2}$ (and similarly for $\mathfrak{o}_{F'}$), we have
${\rm Vol}(\mathfrak{p}_F)={\rm Vol}(\mathfrak{p}_{F'})$ and the result follows.
\end{proof}

The assignment $``{\rm LLC}"$ is now reconciled with the Deligne-Kazhdan theory \eqref{Del-corr} and \eqref{Kaz-corr}.

\begin{proposition} 
\label{commu-LLC}
We assume that non-archimedean local fields $F$ and $F'$ are $1$-close. Then the following diagram commutes:
\[
\xymatrixcolsep{3.5pc}
\xymatrixrowsep{2pc}
\xymatrix{
&\mathcal{A}_0({\rm GL}_n(F)) \ar[d]^{\cong}_{\rm Kaz} \ar[r]^{\rm LLC}_{\cong} & \mathcal{G}^t({\rm GL}_n(F)) \ar[d]^{\rm Del}_{\cong} \\
&\mathcal{A}_0({\rm GL}_n(F')) \ar[r]^{\rm LLC}_{\cong} &\mathcal{G}^t({\rm GL}_n(F'))}
\]
\end{proposition}

\begin{proof}
We will prove this theorem by induction on $n$. When  $n=1$, the Deligne-Kazhdan philosophy is compatible with local class field theory \cite[Property (i) of \S 2.1]{Gan15}.
Now we assume that \Cref{commu-LLC} holds for $1 \leq d \leq n-1$. Let $\rho_1 \in \mathcal{A}_0({\rm GL}_n(F))$ and $\sigma \in \mathcal{A}_0({\rm GL}_d(F))$. Let $\varphi_{\rho_1}$ and $\varphi_{\sigma}$
denote the local Langlands parameter attached to $\rho_1$ and $\sigma$, respectively. We put $\rho'_1={\rm Kaz}(\rho_1)$ and $\sigma'={\rm Kaz}(\sigma)$.
Writing $\rho_2={\rm LLC}^{-1} \circ {\rm Del}^{-1}(\varphi_{\rho'_1})$, the corresponding local Langlands parameter $\varphi_{\rho_2}$ is Del-associated to $\varphi_{\rho'_1}$.
In view of \cite[$({\rm d})$ of Theorem 7.1]{Gan15} along with \cite[Property (i) of \S 2.1]{Gan15} again, $\rho_1$ and $\rho_2$ share the same central character $\omega_{\rho_1}=\omega_{\rho_2}$.
By induction hypothesis, we have 
\[
\sigma={\rm LLC}^{-1} \circ {\rm Del}^{-1}(\varphi_{\sigma'}).
\]
This leads us to a chain of identities:
\begin{multline*}
  \Gamma(s,\rho_1 \times \sigma,\psi_F)= \Gamma(s,\rho'_1 \times \sigma',\psi'_{F'})=\Gamma(s,\varphi_{\rho'_1}\otimes\varphi_{\sigma'},\psi'_{F'})\\
  =\Gamma(s,\varphi_{\rho_2}\otimes {\rm Del}^{-1}(\varphi_{\sigma'}),\psi_F) 
  =\Gamma(s,\rho_2 \times {\rm LLC}^{-1} \circ {\rm Del}^{-1}(\varphi_{\sigma'}),\psi_F)
  =\Gamma(s,\rho_2 \times \sigma,\psi_F)
\end{multline*}
for all $\sigma \in \mathcal{A}_0({\rm GL}_d(F))$ and $1 \leq d \leq n-1$. Here, the second and fourth equalities are a part of local Langlands correspondence \cite[$({\rm b})$ of Theorem 7.1]{Gan15},
the third equality follows from  \cite[Property (iii) of \S 2.1]{Gan15} due to Deligne, and the first equality is clear from Lemma \ref{level-zero-RS}.
Then by the local converse theorem for  level zero supercuspidal representations \cite[Theorem 5.3]{Ye19}, we conclude that $\rho_1 \cong \rho_2$ from which the desired commutative diagram follows.
\end{proof}

\section{The Asai Gamma Factor}
\label{AS factor}

\subsection{The Flicker sum}
Let $\eFiniteField=\fFiniteField_{q^2}$. 
We fix a non-trivial additive character $\qfFieldCharacter$ of $\eFiniteField$ such that
$\qfFieldCharacter\restriction_{\fFiniteField}=\textbf{1}_{\fFiniteField}$.
It is worth pointing out that $\qfFieldCharacter$ can be constructed starting from a non-trivial additive character $\fieldCharacter$ (cf.\citelist{\cite{Nie19}*{\S 1} \cite{AM20}*{\S 2} }).
We define the character of $\qfFieldCharacter$ to be $\qfFieldCharacter(x)=\fieldCharacter({\Trace}_{\eFiniteField \slash \fFiniteField}(\Delta x))$, where $\Delta \in \multiplicativegroup{\eFiniteField}$
is of trace zero. Let $c: x \mapsto \complexConjugate{x}$ be the nontrivial Galois element in $\GaloisGroup(\eFiniteField \slash \fFiniteField)$.
Let $\pi$ be an irreducible cuspidal representation of $\GL_n(\eFiniteField)$ with its associated Whittaker model $\whittaker(\pi,\qfFieldCharacter)$. For  $W_{\pi} \in \whittaker(\pi,\qfFieldCharacter)$ and $\phi \in \Schwartz(\fFiniteField^n)$, we define the {\it Flicker sum} 
\begin{equation}
\label{Def-Flicker-Sum}
  I(W_{\pi},\phi):=\sum_{g \in \UnipotentSubgroup_n(\fFiniteField) \backslash \GL_n(\fFiniteField) } W_{\pi}(g)  \phi(e_ng).
\end{equation}
Similarly, we define the {\it dual Flicker sum} 
\[
  \check{I}(W_{\pi},\phi)
  :=\sum_{g \in \UnipotentSubgroup_n(\fFiniteField) \backslash \GL_n(\fFiniteField) } \check{W}_{\pi}(g) \fourierTransform(\phi)(e_ng).
\]

\begin{lemma}
\label{Dual-Flicker-variant}
Let $\pi$ be an irreducible cuspidal representation of $\GL_{n}(\fFiniteField)$. Then we have
 \[
   \check{I}(W_{\pi},\phi)=\sum_{g \in \UnipotentSubgroup_n(\fFiniteField) \backslash \GL_n(\fFiniteField) }  W_{\pi}(g)   \fourierTransform(\phi)(e_1 {^tg^{-1}}).
 \]
\end{lemma}

\begin{proof}
We insert the definition \eqref{dual-Whittaker-def}. Performing the change of variables $g \mapsto \weyllong_n\, {^tg^{-1}} \weyllong_n$ and then $g \mapsto g \weyllong_n$
yields the result.
\end{proof}

We now aim to show that the space of $\GL_{n}(\fFiniteField)$-equivariant bilinear forms
\[
 L : \whittaker(\pi,\qfFieldCharacter) \times \Schwartz_0(\fFiniteField^n) \rightarrow \mathbb{C}
\]
is at most one-dimensional. Since Flicker sums $I(W_{\pi},\phi)$ and dual Flicker sums $\check{I}(W_{\pi},\phi)$ define non-zero elements of this space (see Theorem \ref{flicker-one-side-thm}),
there is a proportionality constant called the {\it Asai gamma factor} $ \AsaiPiGammaFactor$
of an irreducible cuspidal representation $\pi$ of ${\GL}_n(\eFiniteField)$. The Asai gamma factor $ \AsaiPiGammaFactor$ satisfies
the following functional equation 
\[
 \AsaiPiGammaFactor I(W_{\pi},\phi)=\check{I}(W_{\pi},\phi)
 \]
for any  $W_{\pi} \in \whittaker(\pi,\qfFieldCharacter)$ and $\phi \in \Schwartz_0(\fFiniteField^n)$.

\begin{theorem}
\label{Asai-Multiplicity}
Let $\pi$ be an irreducible cuspidal representation of $\GL_{n}(\fFiniteField)$. For every $W_{\pi} \in \whittaker(\pi,\qfFieldCharacter)$ and for any $\phi \in \Schwartz_0(\fFiniteField^n)$,
 there exists a complex number $\AsaiPiGammaFactor$ satisfying
\[
  \AsaiPiGammaFactor \sum_{g \in \UnipotentSubgroup_n(\fFiniteField) \backslash \GL_n(\fFiniteField) } W_{\pi}(g)  \phi(e_ng)
  =\sum_{g \in \UnipotentSubgroup_n(\fFiniteField) \backslash \GL_n(\fFiniteField) } W_{\pi}(g)  \fourierTransform(\phi)(e_1 {^tg^{-1}}).
\]
\end{theorem}

\begin{proof}
It can be verified from \Cref{Dual-Flicker-variant} that $L_1 : (W_{\pi},\phi) \mapsto   I(W_{\pi},\phi)$ and $L_2 : (W_{\pi},\phi) \mapsto   \contragredient{I}(W_{\pi},\phi)$ correspond to
elements of ${\rm Hom}_{\GL_n(\fFiniteField)}(\pi \otimes \Schwartz_0(\fFiniteField^n),\TrivialRepresentation_{\GL_n(\fFiniteField)})$. 
It is then enough to show that such forms are unique up to scalars, that is to say,  $\dim {\rm Hom}_{\GL_n(\fFiniteField)}(\pi \otimes \Schwartz_0(\fFiniteField^n),\TrivialRepresentation_{\GL_n(\fFiniteField)}) \leq 1$.
We identify $\MirabolicSubgroup_n(\fFiniteField) \backslash \GL_n(\fFiniteField) $
with $\fFiniteField^n - \{ 0 \}$, and then employ the Frobenius reciprocity law to find isomorphisms
\[
\begin{split}
 {\rm Hom}_{\GL_n(\fFiniteField)}(\pi \otimes \Schwartz_0(\fFiniteField^n),\TrivialRepresentation_{\GL_n(\fFiniteField)}) &\cong  {\rm Hom}_{\GL_n(\fFiniteField)}(\pi\restriction_{\GL_n(\fFiniteField)} \otimes {\rm Ind}^{{\GL_n(\fFiniteField)}}_{\MirabolicSubgroup_n(\fFiniteField)}(\TrivialRepresentation),\TrivialRepresentation_{\GL_n(\fFiniteField)}) \\
  &\cong  {\rm Hom}_{\MirabolicSubgroup_n(\fFiniteField)}(\pi \restriction_{\MirabolicSubgroup_n(\fFiniteField)},\TrivialRepresentation_{\MirabolicSubgroup_n(\fFiniteField)}).
\end{split}
\]
The proof of at most one dimension of the space $ {\rm Hom}_{\MirabolicSubgroup_n(\fFiniteField)}(\pi \restriction_{\MirabolicSubgroup_n(\fFiniteField)},\TrivialRepresentation_{\MirabolicSubgroup_n(\fFiniteField)})$ 
is then parallel to that for nonarchimedean local fields \cite[Theorem 1.1]{AKT04}, relying on the theory of Bernstein and Zelevinsky's derivatives for finite fields established in \cite[\S 4]{GGP12}.
\end{proof}

In the course of the proof of proceeding theorem, we get the following multiplicity one result as a byproduct, which is used repeatedly in the proof of \Cref{Asai-gamma-distinction} and \Cref{FR-sum-characterization}.

\begin{corollary} [Multiplicity one result]
\label{mirabolic-Flicker-mulone}
Let $\pi$ be an irreducible cuspidal representation of $\GL_{n}(\fFiniteField)$. Then we have
\[
\dim {\rm Hom}_{\MirabolicSubgroup_n(\fFiniteField)}(\pi \restriction_{\MirabolicSubgroup_n(\fFiniteField)},\TrivialRepresentation_{\MirabolicSubgroup_n(\fFiniteField)}) \leq 1.
\]
\end{corollary}

We express $ \AsaiPiGammaFactor$ in terms of the Bessel functions associated with $\pi$.

\begin{theorem} 
\label{flicker-one-side-thm}
Let $\pi$ be an irreducible cuspidal representation of $\GL_n(\eFiniteField)$. Then we have
\begin{equation}
\label{flicker-one-side}
 \AsaiPiGammaFactor=q^{-\frac{n}{2}} \sum_{g \in \UnipotentSubgroup_n(\fFiniteField) \backslash \GL_n(\fFiniteField) } \grepBesselFunction{\pi}{\qfFieldCharacter}(g)
 \fieldCharacter(e_1 {^tg^{-1}}\;{^te_n}).
\end{equation}
In particular, we have $\AsaiDualPiGammaFactor=\complexConjugate{\AsaiPiGammaFactor}$. 
\end{theorem}

\begin{proof}
We take $W_{\pi}=\grepBesselFunction{\pi}{\qfFieldCharacter}$ and $\phi$ to be an indicator function $\delta_{e_n}$ on $e_n$. It can be seen from \cite[Lemma 2.7]{Nie19}
that $ I(\grepBesselFunction{\pi}{\qfFieldCharacter},\delta_{e_n})=1$ and $\fourierTransform(\delta_{e_n})(y)=q^{-\frac{n}{2}}\fieldCharacter(e_n {^ty})$ from which \eqref{flicker-one-side} shall follow. We now take the complex conjugate to reach
\[
\complexConjugate{ \AsaiPiGammaFactor}
 =q^{-\frac{n}{2}} \sum_{g \in \UnipotentSubgroup_n(\fFiniteField) \backslash \GL_n(\fFiniteField) } \grepBesselFunction{\contragredient{\pi}}{\qfFieldCharacter^{-1}}(g) \fieldCharacter^{-1}(e_1 {^tg^{-1}}\;{^te_n})
 =\AsaiDualPiGammaFactor.
 \qedhere
\]
\end{proof}

The following general lemma plays a crucial role to evaluate the sums of Bessel functions against additive characters for later purpose.

\begin{lemma}
\label{key-distinction}
\cite[Lemma A.2]{SZ23} Let $G$ be a finite group and $L$ a subgroup of $G$. Suppose that $L$ is a semidirect of the form $L=Z \rtimes \GL_n(\fFiniteField)$. Let $\Xi : L \rightarrow \multiplicativegroup{\cComplex}$ be a character which is trivial on $\GL_n(\fFiniteField)$.
Let $\Pi$ be an irreducible representation of $G$ satisfying
\begin{enumerate}[label=$(\mathrm{\roman*})$]
\item $\dim \Hom_L(\Pi\restriction_L,\Xi)=1$.
\item $\dim \Hom_{Z \rtimes \MirabolicSubgroup_n(\fFiniteField)}(\Pi\restriction_{Z \rtimes \MirabolicSubgroup_n(\fFiniteField)},\Xi)=1$.
\item There exists a linear functional $\Lambda \in \Hom_{\UnipotentSubgroup_n(\fFiniteField)}(\Pi\restriction_{\UnipotentSubgroup_n(\fFiniteField)},\TrivialRepresentation_{\UnipotentSubgroup_n(\fFiniteField)})$ and a vector $v_0 \in V_{\Pi}$ such that
\[
  \sum_{p \in \UnipotentSubgroup_n(\fFiniteField) \backslash \MirabolicSubgroup_n(\fFiniteField)  } \sum_{z \in Z} \Lambda(\Pi(zp)v_0) \Xi^{-1}(z)=1.
\]
Then we have
\[
 \sum_{g \in \UnipotentSubgroup_n(\fFiniteField) \backslash \GL_n(\fFiniteField)} \sum_{z \in Z} \Lambda(\Pi(zg) v_0) \Xi^{-1}(z)  \fieldCharacter^{-1}(e_n g {\;^te_1})=-1.
\]
\end{enumerate}
\end{lemma}

A non-zero vector $v \in V_{\pi}$ is said to be a {\it Flicker--Rallis vector} if  $\pi(g)v=v$ for all $g \in {\rm GL}_n(\fFiniteField)$.
Using  \cite[Proposition 5.1]{Pra20} in combination with  \cite[Corollary 2.4]{Pra20}, it is noteworthy that $\pi$ does not have the Flicker--Rallis vector whenever $n=2m$ is even. (Refer to \cite[Theorem 3.9]{Nie19} for $n=2$).

\begin{proposition}
\label{Asai-gamma-distinction}
 Let $\pi$ be an irreducible cuspidal representation of $\GL_{n}(\eFiniteField)$. Suppose that $n=2m+1$ and $\pi$ admits a Flicker--Rallis vector.
Then we have
\[
 \AsaiPiGammaFactor= \AsaiDualPiGammaFactor=-q^{-\frac{n}{2}}.
\]
\end{proposition}

\begin{proof}
Thanks to  \Cref{mirabolic-Flicker-mulone}, we apply \Cref{key-distinction} to $V_{\Pi}=\whittaker(\pi,\qfFieldCharacter)$, $G=\GL_{n}(\eFiniteField)$, $L=\GL_{n}(\fFiniteField)$, $Z=\{ 1_{n} \}$, and $\Xi=\TrivialRepresentation_L$ a trivial character. We define a non-trivial  linear functional 
$\Lambda \in \Hom_{\UnipotentSubgroup_n(\fFiniteField)}(\pi\restriction_{\UnipotentSubgroup_n(\fFiniteField)},\TrivialRepresentation_{\UnipotentSubgroup_n(\fFiniteField)})$
on $\whittaker(\pi,\qfFieldCharacter)$ by
$  \Lambda(W_{\pi})=W_{\pi}(1_{n})$.
By choosing $W_{\pi}= \grepBesselFunction{\pi}{\qfFieldCharacter}$, it is clear from \cite[Lemma 2.7]{Nie19} that
\[
  \sum_{p \in  \UnipotentSubgroup_n(\fFiniteField) \backslash \MirabolicSubgroup_n(\fFiniteField)  } \Lambda(\pi(p)\grepBesselFunction{\pi}{\qfFieldCharacter})= \sum_{p \in \UnipotentSubgroup_n(\fFiniteField) \backslash \MirabolicSubgroup_n(\fFiniteField) }\grepBesselFunction{\pi}{\qfFieldCharacter}(p)=1.
\]
With aid of \Cref{flicker-one-side-thm} coupled with \Cref{key-distinction} again, and then making the change of variables $g \mapsto g^{-1}$, we find that
\begin{multline*}
  \AsaiDualPiGammaFactor=q^{-\frac{n}{2}} \sum_{g \in \UnipotentSubgroup_n(\fFiniteField) \backslash \GL_n(\fFiniteField) } \grepBesselFunction{\contragredient{\pi}}{\qfFieldCharacter^{-1}}(g)
 \fieldCharacter^{-1}(e_1 {^tg^{-1}}\;{^te_n})\\
 = q^{-\frac{n}{2}}\sum_{g \in  \UnipotentSubgroup_n(\fFiniteField) \backslash \GL_n(\fFiniteField)  } \Lambda(\pi(g)\grepBesselFunction{\pi}{\qfFieldCharacter})\fieldCharacter^{-1}(e_n g {^te_1})
 =-q^{-\frac{n}{2}}. 
\end{multline*}
All that remains is to take the complex conjugate to conclude from \Cref{flicker-one-side-thm} that
\[
  \AsaiPiGammaFactor=\complexConjugate{ \AsaiDualPiGammaFactor}=\complexConjugate{-q^{-\frac{2m+1}{2}}}=-q^{-\frac{2m+1}{2}}. \qedhere
\]
\end{proof}

We end this section with functional equations for $\AsaiPiGammaFactor$ of an irreducible cuspidal representation $\pi$ of ${\GL}_n(\eFiniteField)$.

\begin{theorem}[Functional equation] 
\label{Asai-Func}
Let $\pi$ be an irreducible cuspidal representation of $\GL_n(\eFiniteField)$.
 \begin{enumerate}[label=$(\mathrm{\arabic*})$]
\item
If $\pi$ does not admit a Flicker--Rallis vector, then we have
\[ 
\AsaiPiGammaFactor   \AsaiDualPiGammaFactor =1
\quad \text{and} \quad \abs{\AsaiPiGammaFactor}=1.
  \]
\item If $n=2m+1$ and $\pi$ admits a Flicker--Rallis vector, then we have
 \[
 \AsaiPiGammaFactor   \AsaiDualPiGammaFactor =q^{-n} \quad \text{and} \quad
 \abs{\AsaiPiGammaFactor}=q^{-\frac{n}{2}}. 
 \]
 \end{enumerate}
\end{theorem}

\begin{proof}
Appealing to \Cref{Dual-Flicker-variant}, when $\pi$ does not admit a Flicker--Rallis vector, the functional equation is a direct consequence of the double-duality
\[
 \check{I}( \check{W}_{\pi},\ifourierTransform(\phi))=I(W_{\pi},\phi).
\]
just like \cite[Proposition 2.12]{YZ20}. The rest of assertions can be verified from \Cref{flicker-one-side-thm} along with \Cref{Asai-gamma-distinction}.
\end{proof}

\subsection{The Flicker--Rallis period and level zero supercuspidal representations}

We set out to investigate the existence of Flicker--Rallis vectors which characterizes the non-vanishing sum.

\begin{lemma}  
\label{FR-sum-characterization}
Let $\pi$ be an irreducible cuspidal representation of $\GL_{n}(\eFiniteField)$ with $n=2m+1$ odd. Then $\pi$ admits a Flicker--Rallis vector if and only if there exists 
$W_{\pi}  \in \whittaker(\pi,\qfFieldCharacter)$ such that
\[
 \sum_{g \in \UnipotentSubgroup_{n}(\fFiniteField) \backslash \GL_{n}(\fFiniteField) } W_{\pi}(g) \neq 0.
\]
\end{lemma}

\begin{proof}
We assume that $\pi$ has a Flicker--Rallis vector. We endow $\whittaker(\pi,\qfFieldCharacter)$ with
an inner product $(\cdot,\cdot)$ in which $\pi$ is unitary. We define $W_{\rm FR} \in \whittaker(\pi,\qfFieldCharacter)$ by
\[
 W_{\rm FR}(g)=\frac{1}{\sizeof{  \GL_{n}(\fFiniteField)}} \sum_{p \in \MirabolicSubgroup_n(\fFiniteField) } \grepBesselFunction{\pi}{\qfFieldCharacter}(gp).
\]
for $g \in \GL_{n}(\eFiniteField)$. Benefited from the average, we find that
$  W_{\rm FR}(gh)=  W_{\rm FR}(g)$ for all $h \in \MirabolicSubgroup_n(\fFiniteField)$. Using the containment $\Hom_{\GL_{n}(\fFiniteField)}(\pi\restriction_{\GL_{n}(\fFiniteField)},\TrivialRepresentation)
\subseteq \Hom_{ \MirabolicSubgroup_n(\fFiniteField) }(\pi\restriction_{ \MirabolicSubgroup_n(\fFiniteField) },\TrivialRepresentation)$,
we deduce the equality  $\Hom_{\GL_{n}(\fFiniteField)}(\pi\restriction_{\GL_{n}(\fFiniteField)},\TrivialRepresentation)
=\Hom_{ \MirabolicSubgroup_n(\fFiniteField) }(\pi\restriction_{ \MirabolicSubgroup_n(\fFiniteField) },\TrivialRepresentation)$ by the one-dimensionality of both spaces, \Cref{mirabolic-Flicker-mulone}. 
In the same fashion, $W_{\rm FR}$ produces an element $T_{W_{\rm FR}} \in \Hom_{\GL_{n}(\fFiniteField)}(\pi\restriction_{\GL_{n}(\fFiniteField)},\TrivialRepresentation)$
stated by $T_{W_{\rm FR}}(W')=(W',W_{\rm FR})$ for $W' \in \whittaker(\pi,\qfFieldCharacter)$, from which it follows that $W_{\rm FR}$ is a Flicker--Rallis vector.
Furthermore, a non-trivialness of the given summation can be verified, because \cite[Lemma 2.7]{Nie19} yields
\[
 \sum_{g \in \UnipotentSubgroup_{n}(\fFiniteField) \backslash \GL_{n}(\fFiniteField) }W_{\rm FR}(g)=\frac{1}{\sizeof{\UnipotentSubgroup_{n}(\fFiniteField)} }
 \sum_{p \in \MirabolicSubgroup_n(\fFiniteField) } \grepBesselFunction{\pi}{\qfFieldCharacter}(p)=1.
\]

\par
Conversely, we assume that there exists $W_{\pi}  \in \whittaker(\pi,\qfFieldCharacter)$ such that
\[
 \sum_{g \in \UnipotentSubgroup_{n}(\fFiniteField) \backslash \GL_{n}(\fFiniteField) } W_{\pi}(g) \neq 0.
\]
We define $W^{\sharp}_{\rm FR}   \in \whittaker(\pi,\qfFieldCharacter)$ by
\[
W^{\sharp}_{\rm FR}(h)=\frac{1}{\sizeof{\UnipotentSubgroup_{n}(\fFiniteField)} }\sum_{g \in  \GL_{n}(\fFiniteField) }  W_{\pi}(hg).
\]
for $h \in  \GL_{n}(\eFiniteField)$. Combining
\[
 W^{\sharp}_{\rm FR}(1_n)=\sum_{g \in \UnipotentSubgroup_{n}(\fFiniteField) \backslash \GL_{n}(\fFiniteField) } W_{\pi}(g) \neq 0.
\]
along with the quasi-invariance property that $W^{\sharp}_{\rm FR}(hh')=W^{\sharp}_{\rm FR}(h)$ for all $h' \in  \GL_{n}(\fFiniteField)$, $ W^{\sharp}_{\rm FR}$ is indeed a
Flicker--Rallis vector that we seek for.
\end{proof}

Let $\eLocalField$ be a quadratic unramified extension of nonarchimedean local fields $\fLocalField$.
Let $\qFieldCharacter$ be a fixed non-trivial character of $E$ that is trivial on $F$. Then $\qFieldCharacter$ will be of the form $\qFieldCharacter(x)=\fieldCharacter_\fLocalField({\Trace}_{\eLocalField \slash \fLocalField}(\delta x))$, where $\delta \in \multiplicativegroup{\eLocalField}$ is an element of trace zero. According to \cite[Remark 6]{AM17}, $\delta$ is in fact a unit in $\multiplicativegroup{\integersRing}_{\eLocalField}$  of trace zero. For the purpose of relating $\qFieldCharacter$ to $\qfFieldCharacter$, we take $\delta$ to be $\projection^{-1}(\Delta)$, so that
\[
\qfFieldCharacter(\projection(k))=\qFieldCharacter(k)
\]
for $k \in \integersRing_E$. 
Let $\rho$ be level zero supercuspidal representations of $\GL_n(E)$ constructed from irreducible cuspidal representations $\pi$ of  $\GL_n(\eFiniteField)$ with its attached Whittaker models $\whittaker(\rho,\psi_{E \slash F})$.  
We take a Whittaker function $W_{\rho} \in \mathcal{W}(\rho,\psi_{E \slash F})$ and a Schwartz--Bruhat function  $\Phi \in \Schwartz(\fLocalField^n)$,
and form the {\it Flicker integral} defined by
\[
 I(s,W_{\rho},\Phi)=\int_{\UnipotentSubgroup_n(F) \backslash \GL_n(F)} W_{\rho}(g) \Phi(e_ng) |\det g|^s \, dg.
\]
The integral converges absolutely for ${\rm Re}(s)$ sufficiently large, and extend meromorphically to the entire complex plane. 
Furthermore there exists a rational function $\AsaiLocalRhoGammaFactor \in \mathbb{C}(q^{-s})$ satisfying the functional equation \cite[\S 8]{AKMSS21}:
\begin{equation}
\label{local-asai-func}
  I(1-s,\check{W}_{\rho},\gfourierTransform{\fieldCharacter_\fLocalField}(\Phi))=\AsaiLocalRhoGammaFactor I(s,W_{\rho},\Phi)
\end{equation}
As before, interested reader may notice that the gamma factor $\AsaiLocalRhoGammaFactor$ defined above differs by a sign from the conventional one defined by Flicker in \cite[Theorem 1.1]{Mat10}.
The {\it local Asai $L$-function} $L(s,\rho, {\rm As})$ is the generator of  the $\mathbb{C}[q^{\pm s}]$-fractional ideal of $\mathbb{C}(q^{-s})$ generated by the family of Flicker integrals $I(s,W_{\rho},\Phi)$
with $W_{\rho} \in \mathcal{W}(\rho,\psi_{E \slash F})$ and $\Phi \in \Schwartz(\fLocalField^n)$, which is normalized to be of the form $P(q^{-s})^{-1}$ for some $P(X) \in \mathbb{C}[X]$ with $P(0)=1$.

\begin{proposition}[The modified functional equation]
\label{modified-Flicker}
Let $\pi$ be an irreducible cuspidal representation of $\GL_{n}(\eFiniteField)$. Then for every $W_{\pi} \in \whittaker(\pi,\qfFieldCharacter)$, $\phi \in \Schwartz(\fFiniteField^n)$, and $s \in \cComplex$,
there exists $\AsaiLocalRhoGammaFactor$ such that
\begin{multline*}
 \check{I}(W_{\pi},\phi)+q^{-n(1-s)}\omega^{-1}_{\rho}(\varpi)\fourierTransform(\phi)(0)L(n(1-s),\omega^{-1}_{\rho}\restriction_{\multiplicativegroup{\fLocalField}})I(W_{\pi},\TrivialRepresentation_{\fFiniteField^n})\\
 =\AsaiLocalRhoGammaFactor(I(W_{\pi},\phi)+q^{-ns}\omega_{\rho}(\varpi)\phi(0)L(ns,\omega_{\rho}\restriction_{\multiplicativegroup{\fLocalField}})I(W_{\pi},\TrivialRepresentation_{\fFiniteField^n})).
\end{multline*}
\end{proposition}

\begin{proof}
Since $\Support (W^{\circ}_{\rho})\subseteq \UnipotentSubgroup_n(\eLocalField)\multiplicativegroup{\eLocalField}\GL_n(\integersRing_\eLocalField)=\amalg_{l \in \zIntegers} \varpi^{l}_\eLocalField \UnipotentSubgroup_n(\eLocalField)\GL_n(\integersRing_\eLocalField)$, for $\realPart(s) \gg 0$, 
our integral can be decomposed as an infinite series
\[
I(s,W^{\circ}_{\rho},\Phi_{\circ})
=\sum_{l \in \zIntegers} q^{-nls}\int_{\multiplicativegroup{\integersRing}}\omega_{\rho}(x\varpi^l)\int_{\UnipotentSubgroup_n(\fLocalField)\cap K_n \backslash K_n} W^{\circ}_{\rho}(k)\Phi_{\circ}(e_nkx\varpi^{l})\,dk\multiplicativeMeasure x.
\]
With $\Phi_{\circ}$ being a lift of $\phi$, $\Phi_{\circ}(e_nkx\varpi^{l})=0$ for $l < 0$, whereas $\Phi_{\circ}(e_nkx\varpi^{l})=\phi(0)$ for $l > 0$.
When $l=0$, we make the change of variables $k \mapsto kx^{-1}$ to obtain
\begin{multline*}
I(s,W^{\circ}_{\rho},\Phi_{\circ})
=\sum_{l=1}^{\infty} q^{-nls} \omega_{\rho}(\varpi^l)\phi(0)\int_{\multiplicativegroup{\integersRing}} \omega_{\rho}\,(x) \multiplicativeMeasure x \int_{\UnipotentSubgroup_n(\fLocalField)\cap K_n \backslash K_n} W^{\circ}_{\rho}(k) \, dk\\
+ \int_{\UnipotentSubgroup_n(\fLocalField)\cap K_n \backslash K_n} W^{\circ}_{\rho}(k) \Phi_{\circ}(e_nk)\, dk.
\end{multline*}
Just as in the proof of \cite[Theorem 3.1]{NZ18}, we express the integrals as the sum
\begin{multline*}
I(s,W^{\circ}_{\rho},\Phi_{\circ})
=\Volume(\UnipotentSubgroup_n(\integersRing)(1_n+\matrixRing_n(\mathfrak{p})))\\
\times\left(\sum_{l=1}^{\infty} q^{-nls} \omega_{\rho}(\varpi^l)\phi(0)\cdot\int_{\multiplicativegroup{\integersRing}} \omega_{\rho}\,(x) \multiplicativeMeasure x\cdot I(W_{\pi},\TrivialRepresentation_{\fFiniteField^n})+I(W_{\pi},\phi)\right).
\end{multline*}
The first sum becomes $q^{-ns}\omega_{\rho}(\varpi)\phi(0)L(ns,\omega_{\rho}\restriction_{\multiplicativegroup{\fLocalField}})I(W_{\pi},\TrivialRepresentation_{\fFiniteField^n})$ if $\omega_{\rho}$ is unramified, while the first term vanishes  if $\omega_{\rho}$ is ramified, but it is still equal to $q^{-ns}\omega_{\rho}(\varpi)\phi(0)L(ns,\omega_{\rho}\restriction_{\multiplicativegroup{\fLocalField}})I(W_{\pi},\TrivialRepresentation_{\fFiniteField^n})$, since $\omega_{\pi}$ is non-trivial so that $\pi$ dose not possess a non-zero Flicker--Rallis vector. This is equivalent to saying that $I(W_{\pi},\TrivialRepresentation_{\fFiniteField^n})=0$ in virtue of \Cref{FR-sum-characterization}. Combining all together, we find
\[
 I(s,W^{\circ}_{\rho},\Phi_{\circ})
=\Volume(\UnipotentSubgroup_n(\integersRing)(1_n+\matrixRing_n(\mathfrak{p})))
(I(W_{\pi},\phi)+q^{-ns}\omega_{\rho}(\varpi)\phi(0)L(ns,\omega_{\rho}\restriction_{\multiplicativegroup{\fLocalField}})I(W_{\pi},\TrivialRepresentation_{\fFiniteField^n})).
\]
\par
Regarding the dual side, we analogously follow the proof of \Cref{Dual-Flicker-variant} to write it as
\[
I(1-s,\check{W}^{\circ}_{\rho},\gfourierTransform{\fieldCharacter_\fLocalField}(\Phi_{\circ}))=\int_{\UnipotentSubgroup_n(\fLocalField) \backslash \GL_n(\fLocalField) }  W^{\circ}_{\rho}(g)   \gfourierTransform{\fieldCharacter_\fLocalField}(\Phi_{\circ})(e_1 {^tg^{-1}})\abs{\det g}^{s-1} \,dg.
\]
We iterate the process for $I(1-s,\check{W}^{\circ}_{\rho},\gfourierTransform{\fieldCharacter_\fLocalField}(\Phi_{\circ}))$ in order to produce
\begin{multline*}
 I(1-s,\check{W}^{\circ}_{\rho},\gfourierTransform{\fieldCharacter_\fLocalField}(\Phi_{\circ}))
 =\Volume(\UnipotentSubgroup_m(\integersRing)(1_m+\matrixRing_m(\mathfrak{p})))\\
  \times(\check{I}(W_{\pi},\phi)+q^{-n(1-s)}\omega^{-1}_{\rho}(\varpi)\fourierTransform(\phi)(0)L(n(1-s),\omega^{-1}_{\rho}\restriction_{\multiplicativegroup{\fLocalField}})I(W_{\pi},\TrivialRepresentation_{\fFiniteField^n})).
\end{multline*}
It remains to use the functional equation \eqref{local-asai-func} as well as to cross out the common volume term.
\end{proof}

Let $G={\rm GL}_n(E)$ and $L={\rm GL}_n(F)$ in \eqref{distinction}. Analogously to the finite field case, we say that a representation $\rho$ of ${\rm GL}_n(E)$ admits a {\it Flicker--Rallis period} if
$\rho$ is ${\rm GL}_n(F)$-distinguished. According to  \cite[Theorem 6.4]{LMS22}, which is based on \cite[Theorems 1.1]{HM02}, a level zero supercuspidal representation $\rho$ does not have the Flicker--Rallis period as long as $n=2m$ is even. The keen reader may notice that this property is completely analogously to the finite field case, but it is not so surprising to see it in \Cref{Period-Vector-Integral} that both conditions are equivalent. 

\begin{theorem}
\label{levelzero-Asai}
Let $\rho$ be a level zero supercuspidal representation of $\GL_{n}(\fLocalField)$.
 \begin{enumerate}[label=$(\mathrm{\arabic*})$]
 \item If $\pi$ does not admit a Flicker--Rallis vector, then we have
 \[
  \AsaiLocalRhoGammaFactor= \AsaiPiGammaFactor.
 \]
\item If $n=2m+1$ and $\pi$ admits a Flicker--Rallis vector, then we have
\[
  \AsaiLocalRhoGammaFactor=q^{n\left(s-\frac{1}{2}\right)}\omega^{-1}_{\rho}(\varpi) \frac{L(n(1-s),\omega^{-1}_{\rho}\restriction_{\multiplicativegroup{\fLocalField}})}{L(ns,\omega_{\rho}\restriction_{\multiplicativegroup{\fLocalField}})}.
\]
\end{enumerate}
\end{theorem}

\begin{proof}
We begin with the case when $\pi$ does not admit a Flicker--Rallis vector. We apply \Cref{FR-sum-characterization} to see that
$I(W_{\pi},\TrivialRepresentation_{\fFiniteField^n})=0$ for any $W_{\pi}  \in \whittaker(\pi,\qfFieldCharacter)$.
Therefore, \Cref{modified-Flicker} boils down to the equality $\check{I}(W_{\pi},\phi)=\AsaiLocalRhoGammaFactor I(W_{\pi},\phi)$ for any $W_{\pi}  \in \whittaker(\pi,\qfFieldCharacter)$ and $\phi \in  \Schwartz_0(\fFiniteField^n)$. This relation tells us that $\AsaiLocalRhoGammaFactor$ is exactly $\AsaiPiGammaFactor$.

\par
In what follows, we focus on the case when $n=2m+1$ and $\pi$ admits a Flicker--Rallis vector. We take $\phi$ to be $\TrivialRepresentation_{\fFiniteField^n}$ an indicator function on $\TrivialRepresentation_{\fFiniteField}^n$.
The relation $\fourierTransform(\TrivialRepresentation_{\fFiniteField^n})=q^{\frac{n}{2}}\delta_{0}$ implies that $\check{I}(W_{\pi},\TrivialRepresentation_{\fFiniteField^n})=0$. 
In addition, \Cref{FR-sum-characterization} allows us to choose $W_{\pi} \in \whittaker(\pi,\qfFieldCharacter)$ such that $I(W_{\pi},\TrivialRepresentation_{\fFiniteField^n})=1$,
and consequently we obtain from \Cref{modified-Flicker} that
\begin{multline*}
q^{-n(1-s)+\frac{n}{2}}\omega^{-1}_{\rho}(\varpi)L(n(1-s),\omega^{-1}_{\rho}\restriction_{\multiplicativegroup{\fLocalField}})
 =\AsaiLocalRhoGammaFactor(1+q^{-ns}\omega_{\rho}(\varpi)\phi(0)L(ns,\omega_{\rho}\restriction_{\multiplicativegroup{\fLocalField}}))\\
 =\AsaiLocalRhoGammaFactor L(ns,\omega_{\rho}).
\end{multline*}
We are left with solving it for $\AsaiLocalRhoGammaFactor$.
\end{proof}

When $E=F \times F$ and $\rho \cong \rho_1 \times \rho_2$ is a representation of $\GL_n(\fLocalField) \times \GL_n(\fLocalField)$, \Cref{levelzero-Asai} coincides with \Cref{levelzero-RS}.

\subsection{The Asai epsilon factor and the Gauss sum}
\label{sec-Asai-Gauss}
Let $V$ be a $n$-dimensional vector space over $\eFiniteField$. We consider the semi-direct product
\[
({\rm GL}_n(\mathbb{C}) \times {\rm GL}_n(\mathbb{C})) \rtimes {\rm Gal}(\eFiniteField / \fFiniteField),
\]
where the non-trivial Galois element $c$ in $\GaloisGroup(\eFiniteField \slash \fFiniteField)$ acts on ${\rm GL}_n(\mathbb{C}) \times {\rm GL}_n(\mathbb{C})$ by
$(g_1,g_2) \rtimes c:=(g_2,g_1)$. This is the Langlands dual group of ${\rm Res}_{\eFiniteField \slash \fFiniteField}(V / \eFiniteField)$.
Let $s$ be an element of $W_F$ which generates the quotient group $W_F / W_E \cong {\rm Gal}(E/F)$.
Let $\varphi : W_E \rightarrow {\rm GL}_n(\mathbb{C})$ be an $n$-dimensional representation of the Weil group $W_E$. 
We obtain the {\it Asai representation}, $``\rm As"$,
\[
  {\rm As}  ({\varphi}) :  W_F \rightarrow ({\rm GL}_n(\mathbb{C}) \times {\rm GL}_n(\mathbb{C})) \rtimes {\rm Gal}(\eFiniteField / \fFiniteField)
\]
by setting 
\[
 {\rm As}  ({\varphi})(\tau)=(\varphi(\tau),\varphi(s\tau s^{-1})) \in {\rm GL}_n(\mathbb{C}) \times {\rm GL}_n(\mathbb{C})
\]
 for $\tau \in W_E$, and
\[
 {\rm As}  ({\varphi})(s)=(1_n,\varphi(s^2)) \rtimes c \in  ({\rm GL}_n(\mathbb{C}) \times {\rm GL}_n(\mathbb{C})) \rtimes {\rm Gal}(\eFiniteField / \fFiniteField)  \setminus  ({\rm GL}_n(\mathbb{C}) \times {\rm GL}_n(\mathbb{C})).
\]
For $v_1,v_2 \in \mathbb{C}^n$, we have $ {\rm As}  ({\varphi})(\tau)(v_1\otimes v_2)=\varphi(\tau)v_1\otimes\varphi(s\tau s^{-1})v_2$ and
$ {\rm As}  ({\varphi})(s)(v_1\otimes v_2)=\varphi(s^2)v_2\otimes v_1$. For $x$ a real number, let $\lfloor x \rfloor$ be the greatest integer less than or equal to $x$.

\begin{theorem}[E. Zelingher]
\label{Asai-Gauss}
 Let $\varphi$ be $n$-dimensional tamely ramified representations of $W_E$. Let $\alpha$ be a regular character of $\multiplicativegroup{\widehat{\fFiniteField}_{q^{2n}}}$
 corresponding to $\varphi$ via Green's parametrization and $m=\lfloor \frac{n-1}{2} \rfloor$.
 Then we have
\[
  \varepsilon_0( {\rm As}  ({\varphi}),\fieldCharacter_\fLocalField)=(-1)^n q^{-\frac{n^2}{2}}  \tau(\alpha^{1+q^{2m+1}},\psi_d) \prod_{i=0}^{m-1}  \tau(\alpha^{1+q^{2i+1}},\psi_{2n}),
\]
where $d=n$ if $n$ is odd, and $d=2n$ if $n$ is even.
\end{theorem}

\begin{proof}
It is worthwhile to point out that the local class field theory gives
\[
  I_E / P_E \cong \varprojlim  \multiplicativegroup{\fFiniteField_{q^n}},
\]
where the transition maps are given by the norm maps $({\rm Nr}_{n:d})_{d\, |\, n}$ as seen before. Henceforth we may consider $\alpha$ as a character of $I_E / P_E$.
With respect to a suitably chosen basis (cf. \cite[Theorem 2.4]{YZ21}), we have 
\[
 \varphi \restriction_{I_E}(i_E)={\rm diag}(\alpha(i_E),\alpha^{q^2}(i_E),\dotsm,\alpha^{q^{2n-2}}(i_E)) \in {\rm GL}_n(\mathbb{C})
\]
for $i_E \in I_E$, which induces that
\[
 {\rm As}  ({\varphi}) \restriction_{I_F}(i_E)=\left( \begin{pmatrix} \alpha(i_E) &&&\\ &\alpha^{q^2}(i_E)&& \\ &&\ddots& \\ &&& \alpha^{q^{2n-2}}(i_E) \end{pmatrix},
 \begin{pmatrix} \alpha^{q}(i_E) &&&\\ &\alpha^{q^3}(i_E)&& \\ &&\ddots& \\ &&& \alpha^{q^{2n-1}}(i_E) \end{pmatrix} 
 \right).
\]
The element belongs to ${\rm GL}_n(\mathbb{C}) \times {\rm GL}_n(\mathbb{C})$, because $i_E \in I_E$ (so there is no Frobenius). The reason why we have a second matrix
is that ${\rm Fr}_F \cdot x \cdot {\rm Fr}_F^{-1} \equiv x^{q}  \pmod{P_F}$ for $x \in I_F$. This means that $\alpha({\rm Fr}_F \cdot i_E \cdot {\rm Fr}_F^{-1})=\alpha^q(i_E)$.
We index the standard basis of $\mathbb{C}^n$ by $e_i$, where $i=0,1,\dotsm,n-1$.
Putting all together, we obtain
\[
 {\rm As}  ({\varphi}) \restriction_{I_F}(i_E)(e_i \otimes e_j)=\alpha^{q^{2i}+q^{2j+1}}(i_E)(e_i \otimes e_j).
\]
The matrix representing $ {\rm As} ( \varphi)  \restriction_{I_F}$ is indexed by $(i,j)$, where $0 \leq i,j \leq n-1$. Furthermore, the eigenvalue $\alpha^{q^{2i}+q^{2j+1}}$
lies in the Galois orbit of $\alpha^{1+q^{2(j-i)+1}}$ as $( \alpha^{1+q^{2(j-i)+1}})^{q^{2i}}=\alpha^{q^{2i}+q^{2j+1}}$.
Therefore, the Galois orbits indexed by integers $0 \leq d \leq \lfloor \frac{n-1}{2}\rfloor$ are given by
\[
\mathcal{O}(\alpha^{1+q^{2d+1}})=\{ (\alpha^{1+q^{2d+1}})^{q^k} \,|\, 0 \leq k \leq 2n-1 \},
\]
whereas $\mathcal{O}(\alpha^{1+q^{2m+1}})$ looks a bit different for $n=2m+1$ odd:
\[
 \mathcal{O}(\alpha^{1+q^{2m+1}})=\{ (\alpha^{1+q^{2m+1}})^{q^k} \,|\, 0 \leq k \leq n-1 \}.
\]
We emphasize that $\mathcal{O}(\alpha^{1+q^{2d+1}})$'s should be thought of as (multi-)sets possibly with duplicated elements, so they are not exactly Galois orbits.
Each $\mathcal{O}(\alpha^{1+q^{2d+1}})$ consists of multiple copies of the same Galois orbit. Then \cite[Theorem 2.4]{YZ21}
gives us the desired formula for the $\varepsilon_0$-factor.
\end{proof}

Let $ \lambda_{\eLocalField \slash \fLocalField}(\fieldCharacter_\fLocalField)$ be the {\it Langlands constant}
defined in \cite[(30.4)]{BH06}. Appealing to \cite[Proposition 34.3-(1)]{BH06}, the Langlands constant $ \lambda_{\eLocalField \slash \fLocalField}(\fieldCharacter_\fLocalField)$ is given by $ \lambda_{\eLocalField \slash \fLocalField}(\fieldCharacter_\fLocalField)=-1$. We apologize for the double usage of $``\lambda"$, but we hope that the reader can separate the meaning from the context.

\begin{proposition}
\label{Artin-Asai-Flicker}
 Let $\pi(\varphi)$ be an irreducible cuspidal representations of $\GL_n(\eFiniteField)$ associated to $n$-dimensional tamely representation $\varphi$ of $W_E$ via Macdonald correspondence. Then we have
 \[
  \gamma(\pi(\varphi),{\rm As},\fieldCharacter)=\omega^{n-1}_{\pi}(\Delta) \lambda_{\eLocalField \slash \fLocalField}(\fieldCharacter_\fLocalField)^{-\frac{n(n-1)}{2}}   \varepsilon_0( {\rm As}  ({\varphi}),\fieldCharacter_\fLocalField).
 \]
 \end{proposition}

\begin{proof} We divide it into two cases. We assume that $\pi$ does not admits a Flicker-Rallis vector. 
We use \Cref{levelzero-Asai} in conjunction with \cite[Theorem 1.3]{AKMSS21} and \cite[Corollary 2.7]{YZ21}  in order to see that
\begin{multline*}
   \gamma(\pi(\varphi),{\rm As},\fieldCharacter)=\Gamma(s,\rho(\varphi),{\rm As},\fieldCharacter_\fLocalField)\\
   =\omega^{n-1}_{\rho}(\delta) \lambda_{\eLocalField \slash \fLocalField}(\fieldCharacter_\fLocalField)^{-\frac{n(n-1)}{2}} \varepsilon(s,{\rm As}  ({\varphi}),\fieldCharacter_\fLocalField)
   =\omega^{n-1}_{\pi}(\Delta) \lambda_{\eLocalField \slash \fLocalField}(\fieldCharacter_\fLocalField)^{-\frac{n(n-1)}{2}}   \varepsilon_0( {\rm As}  ({\varphi}),\fieldCharacter_\fLocalField).
\end{multline*}

\par
It remains to deal with the case when $n=2m+1$ and $\pi$ admits a Flicker--Rallis vector. Since $\Delta$ is an element of trace zero, $\Delta^2=-\Delta\overline{\Delta}$ belongs to 
$\multiplicativegroup{\fFiniteField}$. The central character restricted to $\multiplicativegroup{\fFiniteField}$, $\omega_{\pi}\restriction_{\multiplicativegroup{\fFiniteField}}=\alpha\restriction_{\multiplicativegroup{\fFiniteField}}$, becomes trivial so that $\alpha^{1+q^{2m+1}}=\TrivialRepresentation$. Using $\omega^{n-1}_{\pi}(\Delta)=\omega_{\pi}^m(\Delta^2)=1$,
this reduces the problem to confirm that
  \[
  \gamma(\pi(\varphi),{\rm As},\fieldCharacter)=(-1)^m \varepsilon_0( {\rm As}  ({\varphi}),\fieldCharacter_\fLocalField). 
 \]
 Now, $\left(\alpha^{1+q^{2i+1}}\right)^{1+q^{2m+1}}=\TrivialRepresentation$ and $\alpha^{1+q^{2i+1}}$ is not trivial for $0 \leq i \leq m-1$. \cite[Proposition 2.6]{YZ21} gives us that
 $\tau(\alpha^{1+q^{2i+1}},\psi_{2(2m+1)})=-q^{2m+1} \alpha(x^{1+q^{2i+1}})$ for which $x \in \multiplicativegroup{\fFiniteField}$. Thus $x^{1+q^{2i+1}} \in  \multiplicativegroup{\fFiniteField}$,
 so $\alpha(x^{1+q^{2i+1}})=1$. Collecting all these together, and then using  the fact that $\tau(\TrivialRepresentation,\psi_{2m+1})=1$, \Cref{Asai-Gauss} tells us that
\begin{multline*}
  \varepsilon_0( {\rm As}  ({\varphi}),\fieldCharacter_\fLocalField)=(-1)^{2m+1} q^{-\frac{(2m+1)^2}{2}}  \tau(\alpha^{1+q^{2m+1}},\psi_{2m+1}) \prod_{i=0}^{m-1}  \tau(\alpha^{1+q^{2i+1}},\psi_{2(2m+1)})\\
  =(-1)^{2m+1} q^{-\frac{(2m+1)^2}{2}}(-q^{2m+1})^m \tau(\TrivialRepresentation,\psi_{2m+1})=(-1)^{m-1}q^{-\frac{2m+1}{2}},
\end{multline*}
which agrees with $(-1)^m \gamma(\pi(\varphi),{\rm As},\fieldCharacter)$  in \Cref{Asai-gamma-distinction}, as required.
\end{proof}

We are in a position to state a main product formula for $ \gamma(\pi,{\rm As},\fieldCharacter)$ with regard to Gauss sums.

\begin{theorem}[Gauss sum]  
\label{Asai-Gamma-Gauss}
Let $\pi$ be an irreducible cuspidal representations of $\GL_n(\eFiniteField)$.  We let $\alpha \in  \multiplicativegroup{\widehat{\fFiniteField}_{q^{2n}}}$ be a regular character 
corresponding to $\pi$ via Green's parametrization 
and $m=\lfloor \frac{n-1}{2} \rfloor$. Then we have
\[
 \gamma(\pi,{\rm As},\fieldCharacter)=\omega^{n-1}_{\pi}(\Delta) \cdot (-1)^{-\frac{n(n+1)}{2}} q^{-\frac{n^2}{2}}  \tau(\alpha^{1+q^{2m+1}},\psi_d) \prod_{i=0}^{m-1}  \tau(\alpha^{1+q^{2i+1}},\psi_{2n}),
\]
where $d=n$ if $n$ is odd, and $d=2n$ if $n$ is even.
\end{theorem}

\section{The Exterior Square Gamma Factor}
\label{Exterior factor}

\subsection{Jacquet--Shalika sums and periods}
We let $\mathcal{M}_{n}$ be $n \times n$ matrices, $\mathcal{N}_n$ the subspace of upper triangular matrices of $\mathcal{M}_n$.
We let $\sigma_{2m}$ be a permutation matrix given by
\[
\sigma_{2m}=\begin{pmatrix} 1 & 2 & \dotsm & m & | & m+1 & m+2 & \dotsm & 2m \\ 
                                                1 & 3 & \dotsm & 2m-1 & | &  2 & 4 & \dotsm &2m \\ 
                                                \end{pmatrix}
\]
and let $\sigma_{2m+1}$ denote
\[
 \sigma_{2m+1}=\begin{pmatrix} 1 & 2 & \dotsm & m & | & m+1 & m+2 & \dotsm & 2m & 2m+1 \\ 
                                                1 & 3 & \dotsm & 2m-1 & | &  2 & 4 & \dotsm &2m & 2m+1 \\ 
                                                \end{pmatrix}.
\]
Let $\pi$ be an irreducible cuspidal representation of $\GL_n(\fFiniteField)$ with its associated Whittaker model $\whittaker(\pi,\psi)$. 
For all $W_{\pi} \in \mathcal{W}(\pi,\psi)$ and $\phi \in  \Schwartz_0(\fFiniteField^m)$,
there exists a complex number $ \gamma(\pi,\wedge^2,\psi)   \in \mathbb{C}^{\times}$ such that
\begin{multline}
\label{Jacquet-Shalika-even}
 \gamma(\pi,\wedge^2,\psi) \sum_{g \in \UnipotentSubgroup_m(\fFiniteField) \backslash \GL_m(\fFiniteField) } \sum_{X \in \mathcal{N}_m(\fFiniteField) \backslash \mathcal{M}_{m}(\fFiniteField)} W_{\pi} \left( \sigma_{2m} \begin{pmatrix} 1_m & X \\ & 1_m  \end{pmatrix}
\begin{pmatrix} g & \\ & g \end{pmatrix} \right)  \psi^{-1}(\Trace X) \phi(e_mg)  \\
 =\sum_{g \in \UnipotentSubgroup_m(\fFiniteField) \backslash \GL_m(\fFiniteField) } \sum_{X \in \mathcal{N}_m(\fFiniteField) \backslash \mathcal{M}_{m}(\fFiniteField)} W_{\pi} \left( \sigma_{2m} \begin{pmatrix} 1_m & X \\ & 1_m  \end{pmatrix} \begin{pmatrix} g & \\ & g \end{pmatrix} \right)  \psi^{-1}( \Trace X)  \fourierTransform(\phi)(e_1\prescript{t}{}{g}^{-1})
\end{multline}
in the even case $n=2m$
\begin{multline}
\label{Jacquet-Shalika-odd}
 \gamma(\pi,\wedge^2,\psi)  
 \sum_{g  } \sum_{X } \sum_{Z}
 W_{\pi} \left( \sigma_{2m+1} \begin{pmatrix} 1_m &X& \\ &1_m& \\ &&1 \end{pmatrix} \begin{pmatrix} g && \\ &g& \\ &&1 \end{pmatrix} 
 \begin{pmatrix} 1_m && \\ &1_m& \\ &Z&1 \end{pmatrix} \right) \psi^{-1}(\Trace X)  \phi(Z) \\
=\sum_{g}  \sum_{X } 
 \sum_{Z} 
 W_{\pi} \left( \begin{pmatrix} & 1 \\ 1_{2m} & \end{pmatrix} \sigma_{2m+1} \begin{pmatrix} 1_m &X& \\ &1_m& \\ &&1 \end{pmatrix} \begin{pmatrix} g && \\ &g& \\ &&1 \end{pmatrix} \begin{pmatrix} 1_m &&-^tZ \\ &1_m& \\ &&1 \end{pmatrix} \right) \\  \psi^{-1}(\Trace X)   \fourierTransform(\phi)(Z)
\end{multline}
in the odd case $n=2m+1$, where the summation domain of $g$, $X$, and $Z$ are taken over $\UnipotentSubgroup_m(\fFiniteField) \backslash \GL_m(\fFiniteField)$,
$\mathcal{N}_m(\fFiniteField) \backslash \mathcal{M}_{m}(\fFiniteField)$, and $\fFiniteField^m$, respectively. We express $\gamma(\pi,\wedge^2,\fieldCharacter)$ in terms of the Bessel functions associated with $\pi$.

\begin{proposition} \cite[Remark 2.20]{YZ20} 
\label{conj-exterior}
Let $\pi$ be an irreducible cuspidal representation of $\GL_{2m}(\fFiniteField)$. Then we have
\begin{multline*}
 \gamma(\pi,\wedge^2,\fieldCharacter)=
 q^{-\frac{m}{2}}\sum_{g \in \UnipotentSubgroup_m(\fFiniteField) \backslash \GL_m(\fFiniteField)} \sum_{X \in \nilpotentMatrices_m(\fFiniteField) \backslash \matrixRing_m(\fFiniteField)} 
  \repBesselFunction{\pi}
  \left( \sigma_{2m} \begin{pmatrix} 1_m & X \\ & 1_m \end{pmatrix} \begin{pmatrix} g & \\ & g \end{pmatrix}\sigma^{-1}_{2m} \right) \\
  \times \fieldCharacter^{-1}({\Trace}\, X)    \fieldCharacter(e_1 {^tg^{-1}}\;{^te_m}).
\end{multline*}
In particular, we have $\gamma(\contragredient{\pi},\wedge^2,\fieldCharacter^{-1})=\complexConjugate{ \gamma(\pi,\wedge^2,\fieldCharacter) }$.
\end{proposition}

We define a Shalika subgroup $S_{2m}$ of ${\rm GL}_{2m}$ by
\[
  S_{2m}=\left\{ \begin{pmatrix} 1_m & X \\ & 1_m \end{pmatrix} \begin{pmatrix} g & \\ & g \end{pmatrix} \,\middle|\, X \in \mathcal{M}_m,\; g \in {\rm GL}_m \right\}.
\]
Let $\Theta$ be a Shalika character of $S_{2m}$ given by
\[
  \Theta \left(  \begin{pmatrix} 1_m & X \\ & 1_m \end{pmatrix} \begin{pmatrix} g & \\ & g \end{pmatrix} \right)=\psi(\Trace X).
\]
A non-zero vector $v \in V_{\pi}$ is called a {\it Jacquet--Shalika vector} if $\pi(s)v=\Theta(s)v$ for every $s \in S_{2m}(\fFiniteField)$.
Over a nonarchimedean local field $F$, we let $G={\rm GL}_{2n}(F)$ and $H=S_{2m}(F)$ in \eqref{distinction}. We say that a representation $\rho$ of $G$ admits a {\it Jacquet--Shalika period}
if $\rho$ is $(S_{2m}(F),\Theta)$-distinguished.  

\begin{proposition}
\label{JS-distinct-Gamma}
Let $\pi$ be an irreducible cuspidal representation of $\GL_{m}(\fFiniteField)$. Suppose that $n=2m$ and $\pi$ admits a Jacquet--Shalika vector.
Then we have
\[
   \gamma(\pi,\wedge^2,\fieldCharacter) =\gamma(\contragredient{\pi},\wedge^2,\fieldCharacter^{-1})=-q^{-\frac{m}{2}}.
\]
\end{proposition}

\begin{proof}
We embed the Shalika subgroup $S_{2m}(\fFiniteField)$ to $G=\GL_{2m}(\fFiniteField)$ via the conjugation by $\sigma_{2m}$.
We define a linear functional $\Lambda \in \Hom_{\UnipotentSubgroup_m(\fFiniteField)}(\pi\restriction_{\UnipotentSubgroup_m(\fFiniteField)},\TrivialRepresentation_{\UnipotentSubgroup_m(\fFiniteField)})$ on 
$\whittaker(\pi,\fieldCharacter)$ by 
\[
\Lambda(W_{\pi})=\frac{1}{\sizeof{\nilpotentMatrices_m(\fFiniteField)}}W_{\pi}(1_{2m}).
\]
We choose $W_{\pi}$ to be $\repBesselFunction{\pi}$. Upon using \Cref{key-distinction} with $G=\GL_{2m}(\fFiniteField)$, $L=S_{2m}(\fFiniteField)$ a Shalika subgroup, and $\Xi=\Theta$ a Shalika character, we find from \cite[Lemma 2.7]{Nie19} that
\begin{multline*}
  \sum_{p \in  \UnipotentSubgroup_m(\fFiniteField) \backslash \MirabolicSubgroup_m(\fFiniteField)  } \sum_{X \in \matrixRing_m(\fFiniteField)} \Lambda\left(\pi\left(\sigma_{2m} \begin{pmatrix} p & \\ & p \end{pmatrix} 
  \begin{pmatrix} 1_m & X \\ & 1_m \end{pmatrix}
  \sigma^{-1}_{2m}\right)\repBesselFunction{\pi}\right)\fieldCharacter^{-1}({\Trace}\, X) \\
  =\sum_{p \in  \UnipotentSubgroup_m(\fFiniteField) \backslash \MirabolicSubgroup_m(\fFiniteField)  } \sum_{X \in \nilpotentMatrices_m(\fFiniteField) \backslash \matrixRing_m(\fFiniteField)}  \repBesselFunction{\pi}\left(\sigma_{2m} \begin{pmatrix} p & \\ & p \end{pmatrix} 
  \begin{pmatrix} 1_m & X \\ & 1_m \end{pmatrix}
  \sigma^{-1}_{2m} \right)\fieldCharacter^{-1}({\Trace}\, X) =1,
\end{multline*}
which gives rise to
\[
 \sum_{g \in  \UnipotentSubgroup_m(\fFiniteField) \backslash \GL_m(\fFiniteField)  } \sum_{X \in \matrixRing_m(\fFiniteField)} \Lambda\left(\pi\left(\sigma_{2m} \begin{pmatrix} g & \\ & g \end{pmatrix} 
  \begin{pmatrix} 1_m & X \\ & 1_m \end{pmatrix}
  \sigma^{-1}_{2m}\right)\repBesselFunction{\pi}\right) \fieldCharacter^{-1}({\Trace}\, X)\fieldCharacter^{-1}(e_m g \;{^te_1}) 
    =-1.
\]
  After multiplying both sides by $q^{-\frac{m}{2}}$, and then making the change of variables $g \mapsto g^{-1}$ and $X \mapsto -X$, we arrive at the identity
\begin{multline*}
 \gamma(\contragredient{\pi},\wedge^2,\fieldCharacter^{-1})
 =q^{-\frac{m}{2}}  \sum_{g \in  \UnipotentSubgroup_m(\fFiniteField) \backslash \GL_m(\fFiniteField)  } \sum_{X \in \nilpotentMatrices_m(\fFiniteField) \backslash \matrixRing_m(\fFiniteField)} \grepBesselFunction{\contragredient{\pi}}{\fieldCharacter^{-1}} \left(\sigma_{2m} \begin{pmatrix} g & \\ & g \end{pmatrix} 
  \begin{pmatrix} 1_m & X \\ & 1_m \end{pmatrix}
  \sigma^{-1}_{2m} \right) \\
  \fieldCharacter({\Trace}\, X)    \fieldCharacter^{-1}(e_1 {^tg^{-1}}\;{^te_m})
     =-q^{-\frac{m}{2}}.
   \end{multline*}
All that remains is to take the complex conjugate. In this way, we conclude from \Cref{conj-exterior} that $ \gamma(\pi,\wedge^2,\fieldCharacter) =\complexConjugate{ \gamma(\contragredient{\pi},\wedge^2,\fieldCharacter^{-1})}=-q^{-\frac{m}{2}}$, as desired.
 \end{proof}
 
 We now spell out functional equations for $\gamma(\pi,\wedge^2,\fieldCharacter)$ of an irreducible cuspidal representation $\pi$ of ${\GL}_n(\fFiniteField)$.
 
 \begin{theorem}[Functional equation]
 \label{Exterior-Func}
  Let $\pi$ be an irreducible cuspidal representation of $\GL_{n}(\fFiniteField)$.
 \begin{enumerate}[label=$(\mathrm{\arabic*})$]
\item\label{Exterior-Func1}
If $\pi$ does not admits a Jacquet--Shalika vector, then we have
\[ 
\gamma(\pi,\wedge^2,\fieldCharacter)   \gamma(\contragredient{\pi},\wedge^2,\fieldCharacter^{-1}) =1
\quad \text{and} \quad \abs{\gamma(\pi,\wedge^2,\fieldCharacter)}=1.
  \]
\item\label{Exterior-Func2} If $n=2m$ and $\pi$ admits a Jacquet--Shalika vector, then we have
 \[
 \gamma(\pi,\wedge^2,\fieldCharacter)   \gamma(\contragredient{\pi},\wedge^2,\fieldCharacter^{-1}) =q^{-m} \quad \text{and} \quad
 \abs{\gamma(\pi,\wedge^2,\fieldCharacter)}=q^{-\frac{m}{2}}. 
 \]
 \end{enumerate}
\end{theorem}

\begin{proof}
Part \ref{Exterior-Func1} has been done in \cite[Remark 2.20]{YZ20}, and Part \ref{Exterior-Func2} is straightforward from \Cref{JS-distinct-Gamma}.
\end{proof}

\subsection{Jacquet--Shalika integrals and close field theory}
Let $\rho$ be level zero supercuspidal representations of $\GL_n(F)$ with its attached Whittaker models $\whittaker(\rho,\psi_F)$.  
Let $\fieldCharacter_\fLocalField^{\flat}$ be a non-trivial additive character of $F$ of level zero, that is trivial on $\mathfrak{o}$ but not on $\mathfrak{p}^{-1}$.
For each $W_{\rho} \in \mathcal{W}(\rho,\fieldCharacter_\fLocalField)$ and $\Phi \in \mathcal{S}(F^m)$, we define Jacquet--Shalika integrals $J(s,W_{\rho},\Phi)$ by
\begin{multline*}
  \int_{N_m(F) \backslash {\rm GL}_m(F)}  \int_{\mathcal{N}_m(F) \backslash \mathcal{M}_{m}(F)} \int_{F^m} W_{\rho} \left( \sigma_{2m+1} \begin{pmatrix} 1_m &X& \\ &1_m& \\ &&1 \end{pmatrix} \begin{pmatrix} g && \\ &g& \\ &&1 \end{pmatrix} \begin{pmatrix} 1_m && \\ &1_m& \\ &z&1 \end{pmatrix} \right) \\
  \psi_\fLocalField^{-1}(\Trace X)  \Phi(z) |\det g|^{s-1}\, dz dX dg
\end{multline*}
in the odd case $n=2m+1$ and 
\[
 \int_{N_m(F) \backslash {\rm GL}_m(F)} \int_{\mathcal{N}_m(F) \backslash \mathcal{M}_{m}(F)} W_{\rho} \left( \sigma_{2m} \begin{pmatrix} 1_m & X \\ & 1_m  \end{pmatrix} \begin{pmatrix} g & \\ & g \end{pmatrix} \right)  \psi_\fLocalField^{-1}(\Trace X) \Phi(e_mg) |\det g|^s \, dX dg
\]
in the even case $n=2m$. These integrals converge absolutely for ${\rm Re}(s)$ sufficiently large, and it defines a rational function in $\mathbb{C}(q^{-s})$.
The exterior square gamma factor $\Gamma(s,\rho,\wedge^2,\fieldCharacter_\fLocalField)$ defined as a proportionality. The exterior square $\gamma$-factor is a rational function in $\mathbb{C}(q^{-s})$
satisfying 
\begin{equation}
\label{JS-local-functional-equ}
J(1-s,\check{\rho}(\tau_n)\check{W}_{\rho},\gfourierTransform{\fieldCharacter_\fLocalField}(\Phi))= \Gamma(s,\rho,\wedge^2,\fieldCharacter_\fLocalField)J(s,W_{\rho},\Phi),
\end{equation}
where $\tau_n$ is the matrix $\displaystyle \begin{pmatrix} & 1_m \\ 1_m & \end{pmatrix}$ if $n=2m$, and the matrix $\displaystyle \begin{pmatrix} & 1_m & \\ 1_m && \\ &&1 \end{pmatrix}$
 if $n=2m+1$. The {\it local exterior square $L$-function} $L(s,\rho,\wedge^2)$ is the generator of the $\mathbb{C}[q^{\pm s}]$-fractional ideal of Jacquet--Shalika integrals $J(s,W_{\rho},\Phi)$
 with $W_{\rho} \in \mathcal{W}(\rho,\fieldCharacter_\fLocalField)$ and $\Phi \in \mathcal{S}(F^m)$ normalized to be of the form $P(q^{-s})$ for some $P(X) \in \mathbb{C}[X]$ satisfying $P(0)=1$.

A principal series representation of the form $\Sigma={\rm Ind}_{B_n(F)}^{{\rm GL}_n(F)}(\mu_1 \otimes \dotsm \otimes \mu_n )$ is said to be {\it spherical} if it has a $K_n$-fixed vector.
It is worthwhile to mention that  $\Sigma$ is a full induced representation from the Borel subgroup $B_n(F)$ of unramified characters  $\mu_i$ of $F^{\times}$.
Here unramified means that each $\mu_i$ is invariant under  the maximal compact subgroup $\mathfrak{o}^{\times}$ of $F^{\times}$.
Such a spherical representation must have a one-dimensional space of Whittaker functionals $\Lambda \in \Hom_{\UnipotentSubgroup_n(\fLocalField)}(\Sigma\restriction_{\UnipotentSubgroup_n(\fLocalField)},\fieldCharacter_\fLocalField^{\flat})$. The map $v \mapsto \Lambda(\Sigma(\cdot)\cdot v)$ a priori need not to be injective, so that the Whittaker model $\mathcal{W}(\Sigma,\fieldCharacter_\fLocalField^{\flat})$
consisting of Whittaker functions on ${\rm GL}_n(F)$ of the form $W_{\Sigma}(g):=\Lambda(\Sigma(g)\cdot v)$ may
only be a model of a quotient of $\Sigma$. However, if $\rho$ is an irreducible generic representation of ${\rm GL}_n(F)$, then $\rho$ is isomorphic to its unique Whittaker model $\mathcal{W}(\rho,\psi_{\fLocalField})$,
 which is the image of $V_{\rho}$ under the map $v \mapsto  \Lambda(\rho(\cdot)\cdot v)$. According to \cite[Theorem 2.2]{Jo24}, the local functional equation \eqref{JS-local-functional-equ} can be extended to
 $\rho$ and $\Sigma$, which is sufficient for applications therein.
%This $K_n$-fixed vector, which is unique up to scalar multiplication, 
%is called the {\it spherical vector} of $\pi$.

\begin{lemma}
\label{spherical-rep-JS}
Let $F$ be a local function field. Let $\rho$ be an irreducible generic subquotient of a spherical representation ${\rm Ind}_{B_n(F)}^{{\rm GL}_n(F)}(\mu_1 \otimes \dotsm \otimes \mu_n )$. Then we have
\[
 \Gamma(s,\rho,\wedge^2,\fieldCharacter^{\flat}_\fLocalField)=\prod_{1\leq j < k \leq n} \Gamma(s,\mu_j \times \mu_k,\fieldCharacter^{\flat}_\fLocalField).
\]
\end{lemma}

\begin{proof}
Let us set $\Sigma={\rm Ind}_{B_n(F)}^{{\rm GL}_n(F)}(\mu_1 \otimes \dotsm \otimes \mu_n )$. 
Let $V_{\rho}$ and $V_{\Sigma}$ denote their underlying space of $\rho$ and $\Sigma$, respectively.
By the uniqueness of Whittaker functionals, 
a non-zero Whittaker functional on $V_{\rho}$ induces a non-zero Whittaker functional on $V_{\Sigma}$. As this representation has a unique Whittaker functional,
this must be it and we conclude that $ \Gamma(s,\rho,\wedge^2,\fieldCharacter^{\flat}_\fLocalField)= \Gamma(s,\Sigma,\wedge^2,\fieldCharacter^{\flat}_\fLocalField)$.
For such a spherical representation $\Sigma$, the subspace of spherical vectors must be one-dimensional and we normalized the spherical Whittaker function $W_{\Sigma}^{\flat}$ in the Whittaker model $\mathcal{W}(\Sigma,,\fieldCharacter^{\flat}_\fLocalField)$
so that $W_{\Sigma}^{\flat}(1_n)=1$. Upon taking $W_{\Sigma}^{\flat} \in \mathcal{W}(\Sigma,\fieldCharacter^{\flat}_\fLocalField)$ and $\Phi^{\flat} \in \mathcal{S}(F^{m})$ of a characteristic function on $\mathfrak{o}^m$ and using \cite[\S 2]{JS88}, we have the identity
\begin{multline*}
  \prod_{i=1}^n L(1-s,\mu^{-1}_i,\wedge^2)\prod_{1\leq j < k \leq n} L(1-s,\mu^{-1}_j \times \mu^{-1}_k)
  =J(1-s,\check{\Sigma}(\tau_n)\check{W}^{\flat}_{\Sigma},\gfourierTransform{\fieldCharacter^{\flat}_\fLocalField}(\Phi^{\flat}))\\
  =\Gamma(s,\Sigma,\wedge^2,\fieldCharacter^{\flat}_\fLocalField)J(s,W_{\Sigma}^{\flat},\Phi^{\flat})=\Gamma(s,\Sigma,\wedge^2,\fieldCharacter^{\flat}_\fLocalField) \prod_{i=1}^n L(s,\mu_i,\wedge^2)\prod_{1\leq j < k \leq n} L(s,\mu_j \times \mu_k)
\end{multline*}
from which the result we seek for follows.
\end{proof}

We let $k$ denote a global function field with field of constant $\mathbb{F}_q$ and ring of ad\`{e}les $\mathbb{A}_k$.
One of the most powerful tool for proving local factors is  the standard globalization due to Lomel\'{i} \cite[Lemma 3.1 and Remark 3.2]{Lom17} (cf. \cite[Lemma 9.28]{AKMSS21}).

\begin{theorem} [Lomel\'{i}]
\label{Lomeli-Globalization}
Let $\rho$ be a level zero unitary supercuspidal representation of ${\rm GL}_n(F)$ over a local function field $F$.
There is a global field $k$ with a set of three places $S=\{ v_0,v_1,v_{\infty}\}$ such that $k_{v_0} \cong F$.
There exists an irreducible cuspidal automorphic representation $\Pi=\otimes'_v \Pi_v$ of ${\rm GL}_n(\mathbb{A}_k)$ satisfying
the following properties:
 \begin{enumerate}[label=$(\mathrm{\roman*})$]
\item $\Pi_{v_0} \cong \rho$;
\item $\Pi_v$ is an irreducible unramified principal series representation at every $v \notin S$;
\item $\Pi_{v_1}$ and $\Pi_{v_{\infty}}$ are irreducible quotients of unramified principal series representations.
\item If $\rho$ is generic, then $\Pi$ is globally generic.
\end{enumerate}
\end{theorem}

The aforementioned globalization is required to prove a purely local statement, namely, that 
local exterior square $\gamma$-factors $\Gamma(s,\rho,\wedge^2,\fieldCharacter_\fLocalField)$ via Rankin-Selberg methods due to Jacquet and Shalika \cite{JS88}
agrees with those $\Gamma_{\rm LS}(s,\rho,\wedge^2,\fieldCharacter_\fLocalField)$ via Langlands-Shahidi methods \cite{HL11} in positive characteristic at hand.
We eventually generalize the equality to all characteristic, notably, zero in \Cref{equal-zero-exterior}.

\begin{theorem}
\label{exter-positive}
Let $\rho$ be a level zero supercuspidal representation of ${\rm GL}_n(F)$ over a local function field $F$. Then we have
\[
  \Gamma(s,\rho,\wedge^2,\fieldCharacter_\fLocalField)= \Gamma_{\rm LS}(s,\rho,\wedge^2,\fieldCharacter_\fLocalField).
\]
\end{theorem}

\begin{proof}
Twisting by an unramified character does not affect the conclusion, so there is no harm to assume that $\pi$ is unitary (cf. \cite[\S 6.6-(vii)]{Lom16}).
Applying Theorem \ref{Lomeli-Globalization} to the level zero supercuspidal representation, there are a global field $k$ with three places $v_0, v_1,$ and $v_{\infty}$
such that $k_{v_0} \cong F$, and an irreducible unitary cuspidal automorphic representation $\Pi$ of ${\rm GL}_n(\mathbb{A}_k)$ with the required properties in \Cref{Lomeli-Globalization}. 
We take a non-trivial additive character $\Psi$ of $\mathbb{A}_k \slash k$, and assume, as we may, that $\Psi_{v_0}=\psi_F$.
The global functional equation via the Langlands-Shahidi method can be read from \cite[\S4-(vi)]{HL11} as
\begin{equation}
\label{LS-global-exterior}
 L^{S}(s,\Pi,\wedge^2)=\Gamma_{\rm LS}(s,\Pi_{v_0},\wedge^2,\Psi_{v_0}) \prod_{v \in S-\{ v_0\}} \Gamma_{\rm LS}(s,\Pi_v,\wedge^2,\Psi_v)  L^{S}(1-s,\check{\Pi},\wedge^2).
\end{equation}
Since for $v \notin S$ we know that $\Pi_v$ and $\Psi_v$ are unramified so that $\varepsilon(s,\Pi_v,\wedge^2,\Psi_v) \equiv 1$,
\cite[Theorem 3.3]{Jo24} is rephrased as
\begin{equation}
\label{RS-global-exterior}
 L^{S}(s,\Pi,\wedge^2)=\Gamma(s,\Pi_{v_0},\wedge^2,\Psi_{v_0}) \prod_{v \in S-\{ v_0\}} \Gamma(s,\Pi_v,\wedge^2,\Psi_v) L^{S}(1-s,\check{\Pi},\wedge^2).
\end{equation}
Applying Lemma \ref{spherical-rep-JS} gives us $\Gamma_{\rm LS}(s,\Pi_v,\wedge^2,\Psi_v)=\Gamma(s,\Pi_v,\wedge^2,\Psi_v)$ for $v \in S-\{ v_0\}$.
In order to do so, the result that we seek for is immediate, once we divide \eqref{LS-global-exterior} by \eqref{RS-global-exterior}.
\end{proof}

Let $\Gamma(s,\wedge^2 ( \varphi),\fieldCharacter_\fLocalField)$ denote the Artin exterior square $\gamma$-factor and $\Gamma(s, \varphi,\fieldCharacter_\fLocalField)$  the Artin standard $\gamma$-factor. We then have a following result.

\begin{proposition}  \cite[Proposition 6.2]{GL15}
\label{Deligne-corres}
For $(F,\varphi,\fieldCharacter_\fLocalField)$ that is {\rm Del}-{\it associated to} $(F',\varphi',\fieldCharacter'_{\fLocalField'})$, we have
\[
  \Gamma(s,\wedge^2 ( \varphi),\psi_F)=\Gamma(s,\wedge^2 ( \varphi'),\fieldCharacter_{\fLocalField'}).
\]
\end{proposition}

Let $\rho(\varphi)$ be a level zero supercuspidal representation of ${\rm GL}_n(F)$ obtained from a tamely ramified Weil representation $\varphi$ of $W_F$ of degree $n$ 
via the local Langlands correspondence (LLC). 
The identity 
\begin{equation}
\label{equal-artin}
\Gamma_{\rm LS}(s,\rho(\varphi),\wedge^2,\fieldCharacter_\fLocalField)=\Gamma(s,\wedge^2( \varphi),\fieldCharacter_\fLocalField)
\end{equation}
relating analytic $\gamma$-factors $\Gamma_{\rm LS}(s,\rho,\wedge,\fieldCharacter_\fLocalField)$ with corresponding Artin factors $\Gamma(s,\wedge^2 (\varphi),\fieldCharacter_\fLocalField)$
 has been established for non-archimedean local fields $F$ of characteristic $0$ in \cite{CST17} 
and positive characteristic in \cite{HL11}.

\begin{proposition} 
\label{Kaz-corres}
For $(F,\rho,\fieldCharacter_\fLocalField)$ that is {\rm Kaz}-{\it associated to} $(F',\rho',\fieldCharacter_{\fLocalField'})$, we have
\[
    \Gamma(s,\rho,\wedge^2,\fieldCharacter_\fLocalField)=\Gamma(s,\rho',\wedge^2,\fieldCharacter_{\fLocalField'}) \quad \text{and} \quad
\Gamma_{\rm LS}(s,\rho,\wedge^2,\fieldCharacter_\fLocalField)=\Gamma_{\rm LS}(s,\rho',\wedge^2,\fieldCharacter_{\fLocalField'}).
\]
\end{proposition}

\begin{proof}
We consider the first equality. Proposition 3.23 in \cite{YZ20} yields the equivalent condition that
${\rm Hom}_{S_{2m}(F)}(\rho \otimes|\det(\cdot)|^{s/2},\Theta) \neq 0$ for some $s \in \mathbb{C}$ if and only if ${\rm Hom}_{S_{2m}(F')}(\rho' \otimes|\det(\cdot)|^{s'/2},\Theta) \neq 0$ for some $s' \in \mathbb{C}$.
If this is the case, we use \cite[Theorem 3.17]{YZ20} in conjunction with the fact that $\omega_{\rho}(\varpi_F)=\omega_{\rho'}(\varpi_{F'})$ to prove
\begin{equation}
\label{even-exterior-JS}
 \Gamma(s,\rho,\wedge^2,\fieldCharacter_\fLocalField)=\dfrac{q^{m(s-1/2)}}{\omega_{\rho}(\varpi_F)}\cdot \dfrac{L(m(1-s),\omega^{-1}_{\rho})}{L(ms,\omega_{\rho})}
 =\dfrac{q^{m(s-1/2)}}{\omega_{\rho'}(\varpi_{F'})}\cdot \dfrac{L(m(1-s),\omega^{-1}_{\rho'})}{L(ms,\omega_{\rho'})}
 =\Gamma(s,\rho',\wedge^2,\fieldCharacter_{\fLocalField'}).
 \end{equation}
Otherwise, owing to \cite[Theorem 3.16]{YZ20}, we are guided to
\begin{equation}
\label{odd-exterior-JS}
\Gamma(s,\rho,\wedge^2,\fieldCharacter_\fLocalField)= \gamma(\pi,\wedge^2,\psi)=\Gamma(s,\rho',\wedge^2,\fieldCharacter_{\fLocalField'}).
\end{equation}

\par
Next, we deal with the second equality. For $(F,\varphi,\fieldCharacter_\fLocalField)$ that is {\rm Del}-associated to $(F',\varphi',\fieldCharacter_{\fLocalField'})$, a similar notation $\rho'(\varphi')$ applies to $\varphi'$.
We know from Proposition \ref{commu-LLC} that $\rho(\varphi)$ is Kaz-associated to $\rho'(\varphi')$, at which point Proposition \ref{Deligne-corres} together with \eqref{equal-artin} completes the proof.
\end{proof}

It is time to bring Deligne-Kazhdan close field theory and Theorem \ref{exter-positive} back together for good use.

\begin{theorem}
\label{equal-zero-exterior}
Let $\varphi$ be a $n$-dimensional tamely ramified Weil representation of $W_\fLocalField$ corresponding to the level zero supercuspidal representation $\rho(\varphi)$ of $\GL_n(\fLocalField)$ via the Macdonald correspondence.
Then we have
\[
  \Gamma(s,\rho(\varphi),\wedge^2,\fieldCharacter_\fLocalField)=\Gamma(s,\wedge^2( \varphi),\fieldCharacter_\fLocalField).
\]
\end{theorem}

\begin{proof}
Given a local field $F'$ of characteristic $p$ and an integer $m \geq 1$, there exists a local field $F$ of characteristic $0$ such that $F'$ is $m$-close to $F$ \cite[p.1123]{GL15}. The converse also holds for $m=1$. Specifically, for a field $F$ of characteristic $0$, its residue field $\mathfrak{o}_F/\mathfrak{p}_F$ is isomorphic to $\mathbb{F}_{q}$ with $q=p^k$ for some prime $p$ and integer $k \geq 1$. Then we take $F'$ to be $\mathbb{F}_{q}((t))$ of characteristic $p$. Now for a $p$-adic field $F$, \Cref{exter-positive} and \Cref{Kaz-corres} allow us to deduce $\Gamma(s,\rho,\wedge^2,\fieldCharacter_\fLocalField)= \Gamma_{\rm LS}(s,\rho,\wedge^2,\fieldCharacter_\fLocalField)$. The desired equality then simply follows from \eqref{equal-artin} and \Cref{exter-positive}.
\end{proof}

\section{The Bump--Friedberg Gamma Factor}
\label{BF factor}

\subsection{The Bump--Friedberg sum}

We define the embedding $J: \GL_m \times \GL_m \rightarrow \GL_n$ by
\[
 J(g,g^{\prime})_{k,l}=
  \begin{cases}
  g_{i,j} &  \text{if $k=2i-1$, $l=2j-1$, } \\
  g^{\prime}_{i,j} &  \text{if $k=2i$, $l=2j$,} \\
 0 &  \text{otherwise,} \\
  \end{cases}
\]
for $n=2m$ even
and  $J: \GL_{m+1} \times \GL_m \rightarrow \GL_n$ by
\[
 J(g,g^{\prime})_{k,l}=
  \begin{cases}
  g_{i,j} &  \text{if $k=2i-1$, $l=2j-1$, } \\
  g^{\prime}_{i,j} &  \text{if $k=2i$, $l=2j$,} \\
 0 &  \text{otherwise,} \\
  \end{cases}
\]
for $n=2m+1$ odd. We denote by $M_{m,m}$ the standard Levi subgroup of $\GL_{2m}$ associated with the partition $(m,m)$ of $2m$.
Let $w_{m,m}=\sigma_{2m}$ and then we set $H_{2m}=w_{m,m}M_{m,m}w^{-1}_{m,m}$. Let $w_{m+1,m}=w_{m+1,m+1}|_{{\rm GL}_{2m+1}}$ so that
\[
 w_{m+1,m}=\begin{pmatrix} 1 & 2 & \dotsm & m+1 & | & m+2 & m+3 & \dotsm &2m & 2m+1 \\ 
                                                1 & 3 & \dotsm & 2m+1 & | &  2 & 4 & \dotsm & 2m-2 &2m \\ 
                                                \end{pmatrix}.
\]
In the odd case, $ w_{m+1,m}\neq \sigma_{2m+1} $. We let $M_{m+1,m}$ denote the standard Levi subgroup of ${\rm GL}_{2m+1}$ associated with the partition $(m+1,m)$ of $2m+1$. We set $H_{2m+1}=w_{m+1,m}M_{m+1,m}w^{-1}_{m+1,m}$. The reason for introducing auxiliary elements $w_{m,m}$ and $w_{m+1,m}$ is that $J(g,g')=w_{m,m} {\rm diag}(g,g')w^{-1}_{m,m}$ for ${\rm diag}(g,g') \in M_{m,m}$, and
$J(g,g')=w_{m+1,m} {\rm diag}(g,g')w^{-1}_{m+1,m}$ for ${\rm diag}(g,g') \in M_{m+1,m}$.
We emphasize that $H_n$ is compatible with the intersection in a manner that $H_n \cap {\rm GL}_{n-1}=H_{n-1}$. Let $\pi$ be an irreducible cuspidal representation of $\GL_{n}(\fFiniteField)$. For  $W_{\pi} \in \whittaker(\pi,\fieldCharacter)$ and $\phi \in \Schwartz_0(\fFiniteField^{\lfloor {(n+1)}/2 \rfloor})$, we define the Bump-Friedberg sum as
\begin{equation}
\label{Def-BF-Even-Sum}
  Z(W_{\pi},\phi)
  :=\sum_{g \in \UnipotentSubgroup_m(\fFiniteField) \backslash \GL_m(\fFiniteField) } \sum_{g' \in \UnipotentSubgroup_m(\fFiniteField) \backslash \GL_m(\fFiniteField) } W_{\pi}(J(g,g')) \phi(e_mg'),
\end{equation}
in the even case $n=2m$ and
\begin{equation}
\label{Def-BF-Odd-Sum}
  Z(W_{\pi},\phi)
  :=\sum_{g \in \UnipotentSubgroup_{m+1}(\fFiniteField) \backslash \GL_{m+1}(\fFiniteField) } \sum_{g' \in \UnipotentSubgroup_m(\fFiniteField) \backslash \GL_m(\fFiniteField) } W_{\pi}(J(g,g')) \phi(e_{m+1}g),
\end{equation}
in the odd case $n=2m+1$. Similarly, we define the dual Bump-Friedberg sum as
\[
  \check{Z}(W_{\pi},\phi)
  :=\sum_{g \in \UnipotentSubgroup_m(\fFiniteField) \backslash \GL_m(\fFiniteField) } \sum_{g' \in \UnipotentSubgroup_m(\fFiniteField) \backslash \GL_m(\fFiniteField) } \check{W}_{\pi}(J(g,g')) \fourierTransform(\phi)(e_mg'),
\]
in the even case $n=2m$
\[
  \check{Z}(W_{\pi},\phi)
  :=\sum_{g \in \UnipotentSubgroup_{m+1}(\fFiniteField) \backslash \GL_{m+1}(\fFiniteField) } \sum_{g' \in \UnipotentSubgroup_{m}(\fFiniteField) \backslash \GL_{m}(\fFiniteField) } \check{W}_{\pi}(J(g,g')) \fourierTransform(\phi)(e_{m+1}g),
\]
in the even case $n=2m+1$.

\begin{lemma}
\label{dual-BF}
 Let $\pi$ be an irreducible cuspidal representation of $\GL_{2m}(\fFiniteField)$. Then we have
 \[
   \check{Z}(W_{\pi},\phi)=\sum_{g \in \UnipotentSubgroup_m(\fFiniteField) \backslash \GL_m(\fFiniteField) } \sum_{g' \in \UnipotentSubgroup_m(\fFiniteField) \backslash \GL_m(\fFiniteField) } W_{\pi}\left( \sigma_{2m} \begin{pmatrix} & g' \\ g & \end{pmatrix} \sigma^{-1}_{2m} \right) \fourierTransform(\phi)(e_1 {^tg'^{-1}})
 \]
 for $m=2n$ even and
 \[
   \check{Z}(W_{\pi},\phi)=\sum_{g \in \UnipotentSubgroup_{m+1}(\fFiniteField) \backslash \GL_{m+1}(\fFiniteField) } \sum_{g' \in \UnipotentSubgroup_{m}(\fFiniteField) \backslash \GL_{m}(\fFiniteField) }
   W_{\pi} (J(g,g'))   \fourierTransform(\phi)(e_1 {^tg^{-1}})
 \]
  for $m=2n+1$. 
\end{lemma}

\begin{proof}
We begin with inserting the identity \eqref{dual-Whittaker-def} for $\check{W}_{\pi}$. Subsequently, we make the change of variables $g \mapsto \weyllong_n {^tg^{-1}} \weyllong_n, g' \mapsto \weyllong_n {^t{g'}^{-1}} \weyllong_n$, and then $g \mapsto g\weyllong_n, g' \mapsto g'\weyllong_n$ to arrive at the result. 
\end{proof}

We aim to show that the space of $H_n(\fFiniteField)$-equivariant bilinear form
\[
 B : \whittaker(\pi,\fieldCharacter) \times \Schwartz_0(\fFiniteField^{\lfloor {(n+1)}/2 \rfloor}) \rightarrow  \mathbb{C}
\]
is at most one-dimensional. Since Bump-Friedberg sums $Z(W_{\pi},\phi)$ and dual Bump-Friedberg sums $\check{Z}(W_{\pi},\phi)$
defines non-zero elements of this space (see Proposition \ref{BF-one-side-thm}), there exists a proportionality constant called
the {\it Bump--Friedberg gamma factor} $  \BFpigammaFactor$ of an irreducible cuspidal representation $\pi$ of ${\GL}_n(\fFiniteField)$.
The Bump--Friedberg factor $  \BFpigammaFactor$ satisfies the following functional equation
\[
\BFpigammaFactor  Z(W_{\pi},\phi)= \check{Z}(W_{\pi},\phi)
\]
 for any $W_{\pi} \in \whittaker(\pi,\fieldCharacter)$ and $\phi \in \Schwartz_0(\fFiniteField^{\lfloor {(n+1)}/2 \rfloor})$.

\begin{theorem}
\label{BF-Multiplicity}
Let $\pi$ be an irreducible cuspidal representation of $\GL_{n}(\fFiniteField)$.  For every $W_{\pi} \in \whittaker(\pi,\fieldCharacter)$ and for any $\phi \in \Schwartz_0(\fFiniteField^{\lfloor {(n+1)}/2 \rfloor})$,
 there exists a complex number $\BFpigammaFactor$ such that
\begin{multline*}
 \BFpigammaFactor \sum_{g \in \UnipotentSubgroup_m(\fFiniteField) \backslash \GL_m(\fFiniteField) } \sum_{g' \in \UnipotentSubgroup_m(\fFiniteField) \backslash \GL_m(\fFiniteField) } W_{\pi} \left(\sigma_{2m}\begin{pmatrix} g & \\ & g'\end{pmatrix}\sigma^{-1}_{2m}\right) \phi(e_mg') \\
 =\sum_{g \in \UnipotentSubgroup_m(\fFiniteField) \backslash \GL_m(\fFiniteField) } \sum_{g' \in \UnipotentSubgroup_m(\fFiniteField) \backslash \GL_m(\fFiniteField) } W_{\pi}\left( \sigma_{2m} \begin{pmatrix} & g' \\ g & \end{pmatrix} \sigma^{-1}_{2m} \right) \fourierTransform(\phi)(e_1 {^tg'^{-1}}),
\end{multline*}
in the even case $n=2m$ and
\begin{multline*}
\BFpigammaFactor \sum_{g \in \UnipotentSubgroup_{m+1}(\fFiniteField) \backslash \GL_{m+1}(\fFiniteField) }  \sum_{g' \in \UnipotentSubgroup_m(\fFiniteField) \backslash \GL_m(\fFiniteField) }  W_{\pi} (J(g,g')) \phi(e_{m+1}g) \\
=\sum_{g \in \UnipotentSubgroup_{m+1}(\fFiniteField) \backslash \GL_{m+1}(\fFiniteField) }  \sum_{g' \in \UnipotentSubgroup_m(\fFiniteField) \backslash \GL_m(\fFiniteField) }  W_{\pi} (J(g,g'))   \fourierTransform(\phi)(e_1 {^tg^{-1}}),
\end{multline*}
in the odd case $n=2m+1$. 
\end{theorem}

\begin{proof} 
It can be checked from \Cref{dual-BF} that bilinear forms $B_1 : (W_{\pi},\phi) \mapsto Z(W_{\pi},\phi)$ and $B_2 : (W_{\pi},\phi) \mapsto \check{Z}(W_{\pi},\phi)$ belong to
the space $\Hom_{  H_n(\fFiniteField) }(\pi\restriction_{H_n(\fFiniteField)  } \otimes \Schwartz_0(\fFiniteField^{\lfloor {(n+1)}/2 \rfloor}) ,\TrivialRepresentation_{H_n(\fFiniteField)})$. 
Hence it suffices to show that such Bilinear forms $B_1$ and $B_2$ differ only by scalars, that is to say, 
\[
\dim \Hom_{  H_n(\fFiniteField) }(\pi\restriction_{H_n(\fFiniteField)  } \otimes \Schwartz_0(\fFiniteField^{\lfloor {(n+1)}/2 \rfloor}) ,\TrivialRepresentation_{H_n(\fFiniteField)}) \leq 1.
\]
We identify $\MirabolicSubgroup_n(\fFiniteField) \cap H_n(\fFiniteField) \backslash H_n(\fFiniteField) $ with $\fFiniteField^{\lfloor {(n+1)}/2 \rfloor} - \{ 0 \}$, and then use the 
Frobenius reciprocity to deduce that
\[
\begin{split}
  \Hom_{  H_n(\fFiniteField) }(\pi\restriction_{H_n(\fFiniteField)  } \otimes \Schwartz_0(\fFiniteField^{\lfloor {(n+1)}/2 \rfloor}) ,\TrivialRepresentation_{H_n(\fFiniteField)}) 
  &\cong \Hom_{  H_n(\fFiniteField) }(\pi\restriction_{H_n(\fFiniteField)  }  \otimes {\rm Ind}^{H_n(\fFiniteField)}_{\MirabolicSubgroup_n(\fFiniteField) \cap H_n(\fFiniteField)}(\TrivialRepresentation) ,\TrivialRepresentation_{H_n(\fFiniteField)}) \\
  & \cong \Hom_{  \MirabolicSubgroup_n(\fFiniteField) \cap H_n(\fFiniteField) }(\pi\restriction_{\MirabolicSubgroup_n(\fFiniteField) \cap H_n(\fFiniteField)  },\TrivialRepresentation_{\MirabolicSubgroup_n(\fFiniteField) \cap H_n(\fFiniteField)}).
\end{split}
\]
After suitably conjugating $\MirabolicSubgroup_n(\fFiniteField) \cap H_n(\fFiniteField)$ by Weyl elements, the dimension of the last space is at most one dimensional by \cite[Theorem 2.22]{YZ20} for even case $=2m$ and \cite[Theorem 2.25]{YZ20} for odd case $n=2m+1$. 
\end{proof}

In the course of the proof of proceeding theorem, we get the following multiplicity one result as a byproduct, which is used in the proof of \Cref{JR-sum-characterization}.

\begin{corollary} [Multiplicity one result]  
\label{mirabolic-BF-mulone}
Let $\pi$ be an irreducible cuspidal representation of $\GL_{n}(\fFiniteField)$. Then we have
\[
 \dim  \Hom_{  \MirabolicSubgroup_n(\fFiniteField) \cap H_n(\fFiniteField) }(\pi\restriction_{\MirabolicSubgroup_n(\fFiniteField) \cap H_n(\fFiniteField)  },\TrivialRepresentation_{\MirabolicSubgroup_n(\fFiniteField) \cap H_n(\fFiniteField}) \leq 1.
\]
\end{corollary}

\subsection{The Friedberg--Jacquet period}

Let $\rho$ be an irreducible supercuspidal representation of $\GL_{2m}(\fLocalField)$. As in \eqref{distinction}, the representation $\rho$ is called $H_{2m}(F)$-{\it distinguished} or {\it distinguished with respect to $H_{2m}(F)$} if
$\Hom_{H_{2m}(\fLocalField)}(\rho\restriction_{H_{2m}(\fLocalField)},\TrivialRepresentation_{H_{2m}(\fLocalField)}) \neq 0$.
We also say that $\rho$ admits a {\it Friedberg--Jacquet period} if $\rho$ is $H_{2m}(F)$-distinguished. Let $\ell$ and $\ell'$ the linear form on $\whittaker(\rho,\fieldCharacter_\fLocalField)$
defined by
\[
 \ell : W_{\rho} \mapsto Z_{(0)}(1/2,W_{\rho}):=\int_{\UnipotentSubgroup_{m}(\fLocalField) \backslash \MirabolicSubgroup_m(\fLocalField)} \int_{\UnipotentSubgroup_m(\fLocalField) \backslash \GL_m(\fLocalField)} W_{\rho}(J(g,p')) \,dgdp'
\]
and
\[
 \ell' : W_{\rho} \mapsto Z_{(0)}(1/2,\check{W}_{\rho}):=\int_{\UnipotentSubgroup_m(\fLocalField) \backslash \MirabolicSubgroup_m(\fLocalField)} \int_{\UnipotentSubgroup_m(\fLocalField) \backslash \GL_m(\fLocalField)} \check{W}_{\rho}(J(g,p')) \,dgdp'.
\]

\begin{lemma}
Let $\rho$ be an irreducible supercuspidal representations of $\GL_{2m}(\fLocalField)$ which is distinguished with respect to $H_{2m}(\fLocalField)$. Then there exists 
a non-zero constant $c(\rho) \in \cComplex^{\times}$, which is independent of $\fieldCharacter_\fLocalField$, such that $\ell'=c(\rho)\ell$. 
\end{lemma}

\begin{proof}
We know from \cite[Proposition 5.14]{Jo23} that $L(s,\pi, {\rm BF})$ is holomorphic at $s=1/2$ since $\rho$ is assumed to be distinguished 
with respect to $H_{2m}(\fLocalField)$. As a consequence, all the integrals $Z_{(0)}(s,W_{\rho})$ are holomorphic at $s=1/2$ from which it follows that
the linear forms $\ell$ and $\ell'$ are well-defined. Since $\rho$ is $H_{2m}(\fLocalField)$-distinguished, $\contragredient{\rho}$ is also $H_{2m}(\fLocalField)$-distinguished.
Taking 
\[
\Hom_{H_{2m}(\fLocalField)}(\rho\restriction_{H_{2m}(\fLocalField)},\TrivialRepresentation_{H_{2m}(\fLocalField)})= \Hom_{\MirabolicSubgroup_{2m}(\fLocalField) \cap H_{2m}(\fLocalField)}(\rho\restriction_{\MirabolicSubgroup_{2m}(\fLocalField) \cap H_{2m}(\fLocalField)},\TrivialRepresentation_{\MirabolicSubgroup_{2m}(\fLocalField) \cap H_{2m}(\fLocalField)})
\]
into account, $\ell$ is a $H_{2m}(\fLocalField)$-invariant functional on $\whittaker(\rho,\fieldCharacter_\fLocalField)$ and the integral in the linear form of $\ell'$
is a $H_{2m}(\fLocalField)$-invariant functional on $\whittaker(\contragredient{\rho},\fieldCharacter_\fLocalField^{-1})$. As $\contragredient{\rho} \cong \rho^{\iota}$ where $\rho^{\iota}(g)=\rho({^tg^{-1}})$, 
$\ell'$ gives a $H_{2m}(\fLocalField)$-invariant functional on $\rho$ as well.
Using the multiplicity one result of $H_{2m}(\fLocalField)$-invariant linear functionals accompanied by \cite[Remark 3]{Off11}, this yields that 
two linear forms $\ell$ and $\ell'$ differ by a non-zero scalar  $c(\rho)$ which depends only on the representation $\rho$.
\end{proof}

Let $\varepsilon(s,\rho,\fieldCharacter_\fLocalField)$ be the standard $\varepsilon$-factor defined by Godement and Jacquet \cite{GJ72}.
The local constant takes the form  $\varepsilon(s,\rho,\fieldCharacter_\fLocalField)=\varepsilon(0,\rho,\fieldCharacter_\fLocalField)q^{-f(\rho,\psi_F)s}$,
where $f(\rho,\psi_F)=-n+f(\rho)$, for a non-negative integer $f(\rho)$ regardless of the choice of $\psi_F$ \cite{BHK98}. We shall primarily be interested in the special value of $\varepsilon(s,\rho,\fieldCharacter_\fLocalField)$  at $s=1/2$. Although we narrow it down to $\{ \pm 1 \}$ for distinguished representations $\rho$,
it is challenging to determine the sign of the root number  $\varepsilon(1/2,\rho,\fieldCharacter_\fLocalField)$. This is what is known as {\it distinction problems} \cite{Off11}. It is our belief that $\varepsilon(1/2,\rho,\fieldCharacter_\fLocalField)=1$ for distinguished representations $\rho$, but we leave these out as it will be a digression from the main theorem of this paper.

\begin{lemma}
\label{central-epsilon}
Let $\rho$ be an irreducible supercuspidal representations of $\GL_{2m}(\fLocalField)$ which is distinguished with respect to $H_{2m}(\fLocalField)$.
Then we have $ \varepsilon(1/2,\rho,\fieldCharacter^{\flat}_\fLocalField)=\varepsilon(1/2,\rho,\fieldCharacter_\fLocalField) \in \{ \pm 1 \}$. In particular, if $\rho$ is a level zero supercuspidal representations of $\GL_{2m}(\fLocalField)$, then we get
\[
  \varepsilon(1/2,\rho,\fieldCharacter^{\flat}_\fLocalField)=\varepsilon(1/2,\rho,\fieldCharacter_\fLocalField)=\varepsilon(s,\rho,\fieldCharacter_\fLocalField) \in \{ \pm 1 \}.
\]
\end{lemma}

\begin{proof}
Appealing to \cite[Theorem 2.2-(iv)]{HL13} with an observation that the central character $\omega_{\rho}$ of the distinguished representation $\rho$ is trivial, the central value of epsilon factors $\varepsilon(1/2,\rho,\fieldCharacter_\fLocalField)$ does not depend on the choice of $\fieldCharacter_\fLocalField$.
With the choice of level one additive character $\fieldCharacter_\fLocalField$, we recall from \cite[\S 6.1]{BHK98} that the level zero supercuspidal representations are of 
conductor zero and then their corresponding gamma factors $\varepsilon(s,\rho,\fieldCharacter_\fLocalField)$ are complex numbers instead of rational functions in $q^s$.
In this way, we see that $\varepsilon(1/2,\rho,\fieldCharacter_\fLocalField)=\varepsilon(s,\rho,\fieldCharacter_\fLocalField)$.
\par
 Let us make a straight observation.
The $\varepsilon$-factor satisfies the identity
$\varepsilon(s,\rho,\fieldCharacter_\fLocalField) \varepsilon(1-s,\check{\rho},\fieldCharacter^{-1}_\fLocalField)=1$. 
Since $\rho$ is self-contragredient, that is to say $\rho \cong \check{\rho}$ \cite[Theorems 3.5 and 4.1]{Jo24}, it is clear from the fact $\varepsilon(s,\check{\rho},\fieldCharacter^{-1}_\fLocalField)=\varepsilon(s,\rho,\fieldCharacter_\fLocalField)$ that
\[
\varepsilon(1/2,\rho,\fieldCharacter_\fLocalField )^2=1.
\]
Thereupon we conclude that $\varepsilon(1/2,\rho,\fieldCharacter_\fLocalField ) \in \{ \pm 1\}$, as claimed.
\end{proof}

Let $W^{\rm ess}_{\rho}$ be the essential Whittaker function defined by Jacquet--Piatetski-Shapiro--Shalika and $\rho_{\rm ur}$ a certain 
unramified standard module attached to $\rho$.
 We refer the reader to \cite[\S 2]{Jo23} for precise definitions of these objects. By evaluating the essential Whittaker function $W^{\rm ess}_{\rho}$ in $\whittaker(\rho,\fieldCharacter^{\flat}_\fLocalField)$, we specify the constant $c(\rho)$.

\begin{proposition}
\label{BF-dual-formula}
Let $\rho$ be an irreducible supercuspidal representations of $\GL_{2m}(\fLocalField)$ which is distinguished with respect to $H_{2m}(\fLocalField)$.
Then we have $\ell'=\varepsilon(s,\rho,\fieldCharacter_\fLocalField)\ell$. In particular, if $\rho$ is a level zero supercuspidal representations of $\GL_{2m}(\fLocalField)$, then we get
\[
  Z(\check{W}_{\pi},\TrivialRepresentation_{\fFiniteField^m})=\varepsilon(s,\rho,\fieldCharacter_\fLocalField)Z(W_{\pi},\TrivialRepresentation_{\fFiniteField^m}).
\]
\end{proposition}

\begin{proof}
Since $c(\rho)$ does not depend on the choice of $\fieldCharacter_\fLocalField$, we can take $\fieldCharacter_\fLocalField$ to $\fieldCharacter^{\flat}_\fLocalField$.
Upon using \cite[Proposition 4.4]{AM17} in conjunction with \Cref{central-epsilon}, and then making the change of variables $g \mapsto g \begin{pmatrix} \varpi^{-f(\pi)}1_{n-1} & \\ & 1 \end{pmatrix}$, we are led to
\[
  \ell'(W^{\rm ess}_{\rho})=\ell(\check{W}^{\rm ess}_{\rho})=\varepsilon(1/2,\rho,\fieldCharacter^{\flat}_\fLocalField)^{2m-1}\ell\left(\contragredient{\rho}  \begin{pmatrix} \varpi^{f(\pi)}1_{n-1} & \\ & 1 \end{pmatrix} W^{\rm ess}_{\contragredient{\rho}} \right)
  =\varepsilon(s,\rho,\fieldCharacter_\fLocalField) \ell(W^{\rm ess}_{\contragredient{\rho}} ).
\]
With the help of \cite[Theorem 5.15]{Jo23}, we deduce from the self-contragredient of $\rho$ that
\[
 \ell(W^{\rm ess}_{\contragredient{\rho}} )=L(1/2,\contragredient{\rho}_{\rm ur})L(1,\contragredient{\rho}_{\rm ur},\wedge^2)=L(1/2,\rho_{\rm ur})L(1,\rho_{\rm ur},\wedge^2)=\ell(W^{\rm ess}_{\rho}).
\]
To sum up, we obtain $ \ell'(W^{\rm ess}_{\rho}) =\varepsilon(s,\rho,\fieldCharacter_\fLocalField)\ell(W^{\rm ess}_{\rho})$, from which we conclude that $c(\rho)=\varepsilon(s,\rho,\fieldCharacter_\fLocalField)$.
\par
We assume that $\rho$ is a level zero supercuspidal representation constructed from $\pi$ and choose $W_{\rho}$ to be $W^{\circ}_{\rho}$. Since the support of $W^{\circ}_{\rho}$
is contained in $\UnipotentSubgroup_{2m}(\fLocalField)\multiplicativegroup{\fLocalField}K_{2m}$, we arrive at (cf. \cite[Theorem 3.1]{NZ18})
\begin{multline*}
\ell(W^{\circ}_{\rho})= \int_{\UnipotentSubgroup_{m}(\fLocalField)\cap K_m \backslash K_m} \int_{\UnipotentSubgroup_m(\fLocalField)\cap K_m \backslash K_m} W^{\circ}_{\rho}(J(k,k')) \,dkdk' \\
 =\Volume(\UnipotentSubgroup_m(\integersRing)(1_m+\matrixRing_m(\mathfrak{p})))^2 \sum_{g \in \UnipotentSubgroup_m(\fFiniteField) \backslash \GL_m(\fFiniteField) } \sum_{g' \in \UnipotentSubgroup_m(\fFiniteField) \backslash \GL_m(\fFiniteField) } W_{\pi}(J(g,g'))\\ =\Volume(\UnipotentSubgroup_m(\integersRing)(1_m+\matrixRing_m(\mathfrak{p})))^2 Z(W_{\pi},\TrivialRepresentation_{{\fFiniteField}^m}),
\end{multline*}
and similarly we have $\ell'(W^{\circ}_{\rho})=\Volume(\UnipotentSubgroup_m(\integersRing)(1_m+\matrixRing_m(\mathfrak{p})))^2 Z_{\fieldCharacter}(\contragredient{W}_{\pi},\TrivialRepresentation_{{\fFiniteField}^m})$.
The common volume term is cancelled out, and we are left with $ Z(\contragredient{W}_{\pi},\TrivialRepresentation_{{\fFiniteField}^m})=\varepsilon(s,\rho,\fieldCharacter_\fLocalField)Z(W_{\pi},\TrivialRepresentation_{{\fFiniteField}^m})$, as required.
\end{proof}

A non-zero vector $v \in V_{\pi}$ is called a {\it Friedberg--Jacquet vector} if $\pi(h)v=v$ for every $h \in H_{2m}(\fFiniteField)$.
We now characterize the existence of Friedberg--Jacquet vectors in terms of the non-vanishing sum.

\begin{lemma}  
\label{JR-sum-characterization}
Let $\pi$ be an irreducible cuspidal representation of $\GL_{n}(\fFiniteField)$ with $n=2m$ even. Then $\pi$ admits a Friedberg--Jacquet vector if and only if there exists $W_{\pi}  \in \whittaker(\pi,\fieldCharacter)$ such that
\[
\sum_{g \in \UnipotentSubgroup_m(\fFiniteField) \backslash \GL_m(\fFiniteField) } \sum_{g' \in \UnipotentSubgroup_m(\fFiniteField) \backslash \GL_m(\fFiniteField) } W_{\pi}(J(g,g')) \neq 0.
\]
\end{lemma}

\begin{proof}
We assume that $\pi$ has a non-zero Friedberg--Jacquet vector.  We equip $\whittaker(\pi,\qfFieldCharacter)$ with
an inner product $(\cdot,\cdot)$ in which $\pi$ is unitary. We define $W_{\rm FJ} \in \whittaker(\pi,\fieldCharacter)$ by
\[
 W_{\rm FJ}(g)=\frac{1}{\sizeof{H_{n}(\fFiniteField)}} \sum_{p \in  \MirabolicSubgroup_n(\fFiniteField) \cap H_n(\fFiniteField) }
\grepBesselFunction{\pi}{\fieldCharacter}(gp)
\]
for $g \in H_{n}(\eFiniteField)$. Taking the advantage of the average, we see that $W_{\rm FJ}(gh)=W_{\rm FJ}(g)$
for all $h \in \MirabolicSubgroup_n(\fFiniteField) \cap H_n(\fFiniteField)$.
Using inclusion $\Hom_{H_n(\fFiniteField)}(\pi\restriction_{H_n(\fFiniteField)},\TrivialRepresentation)
\subseteq \Hom_{  \MirabolicSubgroup_n(\fFiniteField) \cap H_n(\fFiniteField) }(\pi\restriction_{\MirabolicSubgroup_n(\fFiniteField) \cap H_n(\fFiniteField)  },\TrivialRepresentation)$,
we deduce the equality  $\Hom_{H_n(\fFiniteField)}(\pi\restriction_{H_n(\fFiniteField)},\TrivialRepresentation)
=\Hom_{  \MirabolicSubgroup_n(\fFiniteField) \cap H_n(\fFiniteField) }(\pi\restriction_{\MirabolicSubgroup_n(\fFiniteField) \cap H_n(\fFiniteField)  },\TrivialRepresentation)$ from the one-dimensionality of both spaces, \Cref{mirabolic-BF-mulone}. In this way, $W_{\rm FJ}$ produces an element $T_{W_{\rm FJ}} \in \Hom_{H_n(\fFiniteField)}(\pi\restriction_{H_n(\fFiniteField)},\TrivialRepresentation)$
stated by $T_{W_{\rm FJ}}(W')=(W',W_{\rm FJ})$ for  $W' \in \whittaker(\pi,\fieldCharacter)$, from which it follows that $W_{\rm FJ}$ is a Friedberg--Jacquet vector.
Furthermore, the given summation is non-trivial, because \cite[Lemma 2.7]{Nie19} yields
\[
 \sum_{g \in \UnipotentSubgroup_m(\fFiniteField) \backslash \GL_m(\fFiniteField) } \sum_{g' \in \UnipotentSubgroup_m(\fFiniteField) \backslash \GL_m(\fFiniteField) } W_{\rm FJ}(J(g,g'))
 =\frac{1}{\sizeof{\UnipotentSubgroup_{n}(\fFiniteField)\cap H_n(\fFiniteField)} } \sum_{p \in  \MirabolicSubgroup_n(\fFiniteField) \cap H_n(\fFiniteField) }
\grepBesselFunction{\pi}{\fieldCharacter}(p)=1.
\]

\par
Conversely, we assume that there exists $W_{\pi}  \in \whittaker(\pi,\fieldCharacter)$ such that
\[
\sum_{g \in \UnipotentSubgroup_m(\fFiniteField) \backslash \GL_m(\fFiniteField) } \sum_{g' \in \UnipotentSubgroup_m(\fFiniteField) \backslash \GL_m(\fFiniteField) } W_{\pi}(J(g,g')) \neq 0.
\]
We define $W^{\sharp}_{\rm FJ}   \in \whittaker(\pi,\fieldCharacter)$ by
\[
W^{\sharp}_{\rm FJ}(h)=\frac{1}{\sizeof{\UnipotentSubgroup_{n}(\fFiniteField)\cap H_n(\fFiniteField)} }\sum_{g \in H_n(\fFiniteField)}  W_{\pi}(hg).
\]
Combining
\[
 W^{\sharp}_{\rm FJ}(1_{n})=\sum_{g \in \UnipotentSubgroup_m(\fFiniteField) \backslash \GL_m(\fFiniteField) } \sum_{g' \in \UnipotentSubgroup_m(\fFiniteField) \backslash \GL_m(\fFiniteField) } W_{\pi}(J(g,g')) \neq 0
\]
along with the quasi-invariance property that $W^{\sharp}_{\rm FJ}(hh')=W^{\sharp}_{\rm FJ}(h)$ for all $h' \in H_n(\fFiniteField)$, $ W^{\sharp}_{\rm FJ}$ is a non-zero 
Friedberg--Jacquet vector that we seek for.
\end{proof}

\subsection{Bump--Freidberg integrals and close field theory}

Let $\rho$ be level zero supercuspidal representations of $\GL_n(F)$
 constructed from irreducible cuspidal representations $\pi$ of  $\GL_n(\fFiniteField)$ 
 with its attached Whittaker models $\whittaker(\rho,\psi_F)$.  
For $W_{\rho} \in \whittaker(\rho,\fieldCharacter_\fLocalField)$ and $\Phi \in \Schwartz(\fLocalField^{\lfloor {(n+1)}/2 \rfloor})$, we define the {\it Bump-Friedberg integral} $Z(s_1,s_2,W_{\rho},\Phi)$ by
\[
\int_{\UnipotentSubgroup_m(\fLocalField) \backslash \GL_m(\fLocalField)} \int_{\UnipotentSubgroup_m(\fLocalField) \backslash \GL_m(\fLocalField)} W_{\rho}(J(g,g')) \Phi(e_mg')   \abs{\det g}^{s_1-1/2} \abs{\det g'}^{1/2+s_2-s_1} \, dg' dg
\]
for $n=2m$ even and
\[
\int_{\UnipotentSubgroup_{m+1}(\fLocalField) \backslash \GL_{m+1}(\fLocalField)} \int_{\UnipotentSubgroup_m(\fLocalField) \backslash \GL_m(\fLocalField)} W_{\rho}(J(g,g')) \Phi(e_{m+1}g) \abs{ \det g }^{s_1}   \abs{\det g'}^{s_2-s_1} \, dg' dg
\]
for $n=2m+1$ odd. For the sake of coherence with \cite{Mat15,Mat17}, we bring further notations. For $t \in \cComplex$ a complex number, we denote by $\delta_t$ the character defined by
\[
  \delta_t : J(g,g') \mapsto \abs{\frac{\det g}{\det g'}}^t.
\]
We denote by $\chi_n$ characters of $H_n(\fLocalField)$:
\[
 \chi_n  : J(g,g') \mapsto 
 \begin{cases}
 \TrivialRepresentation_{H_n(\fLocalField)} &\text{for $n=2m$;} \\
   \abs{\dfrac{\det g}{\det g'}} &\text{for $n=2m+1$.}
 \end{cases}
\]
In particular, we are interested in the case for $s_1=s+t+1/2$ and $s_2=2s$. With regard to $ \delta_r$, $\chi_n$, $s$, and $t$,
These Bump--Friedberg integrals depending on the parity of numbers $n$ can be put in a single integral as
\[
 Z(s,t,W_{\rho},\Phi)=\int_{ (\UnipotentSubgroup_n(\fLocalField) \cap H_n(\fLocalField)) \backslash H_n(\fLocalField)} W_{\rho}(h) \Phi(e_n h) \chi^{1/2}_n(h) \delta_t(h) \abs{\det h}^s  dh.
\]
The integral converges absolutely for ${\rm Re}(s)$ and ${\rm Re}(t)$ sufficiently large \cite[Proposition 3.1]{Mat17}, and it enjoys a meromorphic continuation to $\cComplex \times \cComplex$
as an element of $\cComplex(q^{-s},q^{-t})$. There exists a rational function  $\BFLocalRhogammaFactor$ in $\cComplex(q^{-s},q^{-t})$ such that for every $W_{\rho}$ in  $\whittaker(\rho,\fieldCharacter_\fLocalField)$ and $\Phi$ in $\Schwartz(\fLocalField^{\lfloor {(n+1)}/2 \rfloor})$, we have the following functional equation \cite[Corollary 3.2]{Mat17}:
\begin{equation}
\label{Local-BF-Func}  Z(1/2-s,-1/2-t,\contragredient{W}_{\rho},\gfourierTransform{\fieldCharacter_\fLocalField}(\Phi))=\BFLocalRhogammaFactor Z(s,t,W_{\rho},\Phi).
\end{equation}
For our purpose, it will often be convenient to write $Z(s,W_{\rho},\Phi)$ and $\Gamma(s,\rho,{\rm BF},\fieldCharacter_\fLocalField)$ in place of $Z(s,0,W_{\rho},\Phi)$ and $\Gamma(s,0,\rho,{\rm BF},\fieldCharacter_\fLocalField)$, respectively. The {\it local Bump--Friedberg $L$-function} $L(s,\rho,{\rm BF})$ is the generator of the $\mathbb{C}[q^{\pm s}]$-fractional ideal of Bump--Friedberg integrals $Z(s,W_{\rho},\Phi)$ with $W_{\rho} \in \mathcal{W}(\rho,\psi_F)$ and $\Phi \in \mathcal{S}(F^{\lfloor (n+1)/2 \rfloor})$ normalized to be of the form $P(q^{-s})$ for some $P(X) \in \mathbb{C}[X]$ satisfying $P(0)=1$.

\begin{proposition}[The modified functional equation]
\label{BF-modified-functional equation}
Let $\pi$ be an irreducible cuspidal representation of $\GL_{2m}(\fFiniteField)$. Then for every $W_{\pi} \in \whittaker(\pi,\fieldCharacter)$, $\phi \in \Schwartz(\fFiniteField^m)$, and $s \in \cComplex$,
there exists $\BFLocalRhogammaFactor$ such that
\begin{multline*}
 \contragredient{Z}(W_{\pi},\phi)+q^{-m(1-2s)}\omega^{-1}_{\rho}(\varpi)\fourierTransform(\phi)(0)L(m(1-2s),\omega^{-1}_{\rho})\varepsilon(s,\rho,\fieldCharacter_\fLocalField)Z(W_{\pi},\TrivialRepresentation_{\fFiniteField^m})\\
 = \BFLocalRhogammaFactor(Z(W_{\pi},\phi)+q^{-2ms}\omega_{\rho}(\varpi)\phi(0)L(2ms,\omega_{\rho})Z(W_{\pi},\TrivialRepresentation_{\fFiniteField^m})).
\end{multline*}
\end{proposition}

\begin{proof}
The computation for two sides are quite similar. For this reason, we give a detailed proof only in the dual integral  $Z(1/2-s,-1/2-t,\contragredient{W}^{\circ}_{\rho},\gfourierTransform{\fieldCharacter_\fLocalField}(\Phi_{\circ}))$.
The rational for our choice is that the dual side of modified functional equation is less written in the literature (cf. \cite[Theorem 3.11]{YZ20}, \Cref{modified-Flicker})
and certain additional difficulties arise. Since the support of $W^{\circ}_{\rho}$ lies in $\UnipotentSubgroup_n(\fLocalField)\multiplicativegroup{\fLocalField}K_n=\amalg_{l \in \zIntegers} \varpi^{l} \UnipotentSubgroup_n(\fLocalField)K_n$,
for $\realPart(s) \ll 0$, $\contragredient{Z}(1/2-s,\contragredient{W}^{\circ}_{\rho},\gfourierTransform{\fieldCharacter_\fLocalField}(\Phi))$ can be decomposed as
\begin{multline*}
Z(1/2-s,-1/2-t,\contragredient{W}^{\circ}_{\rho},\gfourierTransform{\fieldCharacter_\fLocalField}(\Phi_{\circ}))
 =\sum_{l \in \zIntegers} q^{-ml(1-2s)}\int_{\multiplicativegroup{\integersRing}}\omega^{-1}_{\rho}(x\varpi^l) \\
\times \int_{\UnipotentSubgroup_m(\fLocalField)\cap K_m \backslash K_m} \int_{\UnipotentSubgroup_m(\fLocalField)\cap K_m \backslash K_m} \contragredient{W}^{\circ}_{\rho}(J(k,k'))\gfourierTransform{\fieldCharacter_\fLocalField}(\Phi_{\circ})(e_mk'x\varpi^{l})\,dkdk'\multiplicativeMeasure x.
\end{multline*}
The Fourier transform $\gfourierTransform{\fieldCharacter_\fLocalField}(\Phi_{\circ})$ is a lift of $\fourierTransform(\phi)$, from which we deduce that $\gfourierTransform{\fieldCharacter_\fLocalField}(\Phi_{\circ})(e_mk'x\varpi^{l})=0$ for $l < 0$, while $\gfourierTransform{\fieldCharacter_\fLocalField}(\Phi_{\circ})(e_mk'x\varpi^{l})=\phi(0)$ for $l > 0$. Upon making the change of variables $k' \mapsto k'x^{-1}$ for $l=0$, the infinite sum can be reduced to
\begin{multline*}
 Z(1/2-s,-1/2-t,\contragredient{W}^{\circ}_{\rho},\gfourierTransform{\fieldCharacter_\fLocalField}(\Phi_{\circ}))\\
 =
 \int_{\UnipotentSubgroup_m(\fLocalField)\cap K_m \backslash K_m} \int_{\UnipotentSubgroup_m(\fLocalField)\cap K_m \backslash K_m}   \contragredient{W}^{\circ}_{\rho}(J(k,k')) \gfourierTransform{\fieldCharacter_\fLocalField}(\Phi_{\circ})(e_mk')\, dk dk'
+ \left( \sum_{l=1}^{\infty} q^{-ml(1-2s)}\omega^{-1}_{\rho}(\varpi^l) \right. \\
\cdot \fourierTransform(\phi)(0) 
\left. \int_{\multiplicativegroup{\integersRing}} \omega^{-1}_{\rho}\,(x) \multiplicativeMeasure x
 \int_{\UnipotentSubgroup_m(\fLocalField)\cap K_m \backslash K_m}  \int_{\UnipotentSubgroup_m(\fLocalField)\cap K_m \backslash K_m} \contragredient{W}^{\circ}_{\rho}(J(k,k')) \, dk dk' \right).
 \end{multline*}
 We rewrite the integration as sums akin to the proof of \cite[Theorem 3.1]{NZ18} by
\begin{multline}
\label{Tate-BF}
 Z(1/2-s,-1/2-t,\contragredient{W}^{\circ}_{\rho},\gfourierTransform{\fieldCharacter_\fLocalField}(\Phi_{\circ}))
 =\Volume(\UnipotentSubgroup_m(\integersRing)(1_m+\matrixRing_m(\mathfrak{p})))^2\\
 \times\left(\contragredient{Z}(W_{\pi},\phi)+ \sum_{l=1}^{\infty} q^{-ml(1-2s)}\omega^{-1}_{\rho}(\varpi^l) 
 \fourierTransform(\phi)(0) \int_{\multiplicativegroup{\integersRing}} \omega^{-1}_{\rho}\,(x) \multiplicativeMeasure x\cdot Z(\contragredient{W}_{\pi},\TrivialRepresentation_{\fFiniteField^m})
\right).
 \end{multline}
To deal with the second term, we assume that $\omega_{\rho}$ is unramified. Combining \cite[\S 5]{Mat14} with \cite[Theorem 2.30 and Proposition 3.23]{YZ20}, $\rho\nu^{s_0}$ is $S_{2m}(\fLocalField)$-distinguished for some $s_0 \in \mathbb{C}$, which is amount to saying that it is $H_{2m}(\fLocalField)$-distinguished.
It follows from \Cref{BF-dual-formula} that the second term is equal to
\begin{multline*}
q^{-m(1-2s)}\omega^{-1}_{\rho}(\varpi)\fourierTransform(\phi)(0)L(m(1-2s),\omega^{-1}_{\rho}) Z(\contragredient{W}_{\pi},\TrivialRepresentation_{{\fFiniteField}^m})\\
=q^{-m(1-2s)}\omega^{-1}_{\rho}(\varpi)\fourierTransform(\phi)(0)L(m(1-2s),\omega^{-1}_{\rho})\varepsilon(s,\rho,\fieldCharacter_\fLocalField)Z(W_{\pi},\TrivialRepresentation_{\fFiniteField^m}).
\end{multline*}
On the other hand, if $\omega_{\rho}$ is ramified,  $\pi$ does not admit a non-zero Friedberg--Jacquet vector in that $\omega_{\pi}$ is non-trivial.
This in turn implies that the second term vanishes, as
\[
\int_{\multiplicativegroup{\integersRing}} \omega^{-1}_{\rho}\,(x) \multiplicativeMeasure x=0=Z(W_{\pi},\TrivialRepresentation_{\fFiniteField^m}),
\]
 thanks to \Cref{JR-sum-characterization}. The analogous argument for $Z(s,t,W^{\circ}_{\rho},\Phi_{\circ})$ goes through, and it guides us to
\[
 Z(s,t,W^{\circ}_{\rho},\Phi_{\circ})
 =\Volume(\UnipotentSubgroup_m(\integersRing)(1_m+\matrixRing_m(\mathfrak{p})))^2
 (Z(W_{\pi},\phi)+q^{-2ms}\omega_{\rho}(\varpi)\phi(0)L(2ms,\omega_{\rho})Z(W_{\pi},\TrivialRepresentation_{\fFiniteField^m})).
\]
All that remains is to apply the functional equation \eqref{Local-BF-Func} and then to cancel out the common volume term.
\end{proof}

In contrast to the exterior square local factor $\Gamma(s,\rho,\wedge^2,\fieldCharacter_\fLocalField)$, the Bump and Friedberg local factor $\BFLocalRhogammaFactor$ possesses two parameters $s$ and $t$. For this reason, $\Gamma(s,t,\rho,{\rm BF},\fieldCharacter_\fLocalField)$ is not really defined as a proportionality, but rather the functional equation for the $\varepsilon$-factor, $\varepsilon(s,t,\rho,{\rm BF},\fieldCharacter_\fLocalField)$, need to be established beforehand.
Thankfully, we will see the following theorem, \Cref{levelzero-BF}, that $\BFLocalRhogammaFactor$ is independent of $t$ in the class of level zero supercuspidal representations $\rho$.
Therefore there is no harm to assign $t=0$ to define $Z(s,W_{\rho},\Phi)$ and $\Gamma(s,\rho,{\rm BF},\fieldCharacter_\fLocalField)$.

\begin{theorem}
\label{levelzero-BF} 
Let $\rho$ be a level zero supercuspidal representation of $\GL_{n}(\fLocalField)$.
 \begin{enumerate}[label=$(\mathrm{\arabic*})$]
 \item If $\pi$ does not admit a Friedberg--Jacquet vector, then we have
 \[
 \BFLocalRhogammaFactor = \BFpigammaFactor.
 \]
 \item If $n=2m$ and $\pi$ admits a Friedberg--Jacquet vector, then we have
 \begin{equation}
 \label{BF-levelzero-poles}
 \BFLocalRhogammaFactor=\varepsilon(s,\rho,\fieldCharacter_\fLocalField)q^{m\left(2s-\frac{1}{2}\right)}\omega^{-1}_{\rho}(\varpi) \frac{L(m(1-2s),\omega^{-1}_{\rho})}{L(2ms,\omega_{\rho})}.
 \end{equation}
\end{enumerate}
\end{theorem}

\begin{proof}
\par
We first look into the odd case $m=2n+1$. Just like \eqref{Tate-BF}, we decompose the domain of integrations $\UnipotentSubgroup_n(\fLocalField)\multiplicativegroup{\fLocalField}K_n$ as shells $\varpi^{l} \UnipotentSubgroup_n(\fLocalField)K_n$ to see that
\begin{multline}
\label{BF-one-side}
  Z(s,t,W^{\circ}_{\rho},\Phi_{\circ})=\Volume(\UnipotentSubgroup_m(\integersRing)(1_m+\matrixRing_m(\mathfrak{p})))\Volume(\UnipotentSubgroup_{m+1}(\integersRing)(1_{m+1}+\matrixRing_{m+1}(\mathfrak{p})))\\
   \times\left(Z(W_{\pi},\phi)+ \sum_{l=1}^{\infty} q^{-l((2m+1)s+t+1/2)}\omega_{\rho}(\varpi^l) 
\phi(0) \int_{\multiplicativegroup{\integersRing}} \omega_{\rho}\,(x) \multiplicativeMeasure x\cdot Z(W_{\pi},\TrivialRepresentation_{\fFiniteField^m})
\right),
\end{multline}
and
\begin{multline}
\label{BF-dual-side}
  Z(1/2-s,-1/2-t,\contragredient{W}^{\circ}_{\rho},\gfourierTransform{\fieldCharacter_\fLocalField}(\Phi_{\circ}))\\
  =\Volume(\UnipotentSubgroup_m(\integersRing)(1_m+\matrixRing_m(\mathfrak{p})))\Volume(\UnipotentSubgroup_{m+1}(\integersRing)(1_{m+1}+\matrixRing_{m+1}(\mathfrak{p}))) \left(\contragredient{Z}(W_{\pi},\phi) \right. \\
   \left.+ \sum_{l=1}^{\infty} q^{-l((2m+1)(1/2-s)-t)}\omega^{-1}_{\rho}(\varpi^l) 
 \fourierTransform(\phi)(0) \int_{\multiplicativegroup{\integersRing}} \omega^{-1}_{\rho}\,(x) \multiplicativeMeasure x\cdot Z(\contragredient{W}_{\pi},\TrivialRepresentation_{\fFiniteField^m})
\right).
\end{multline}
We take $W_{\pi}=\repBesselFunction{\pi}$ and $\phi=\delta_{e_{m+1}}$. In light of \eqref{BF-one-side} along with \Cref{key-distinction}, we are left with
\begin{multline*}
Z(s,t,\liftedBesselFunction{\pi}{\fieldCharacter_\fLocalField},\Phi_{\circ})=\Volume(\UnipotentSubgroup_m(\integersRing)(1_m+\matrixRing_m(\mathfrak{p})))\Volume(\UnipotentSubgroup_{m+1}(\integersRing)(1_{m+1}+\matrixRing_{m+1}(\mathfrak{p}))) Z(\repBesselFunction{\pi},\delta_{e_{m+1}})\\
=\Volume(\UnipotentSubgroup_m(\integersRing)(1_m+\matrixRing_m(\mathfrak{p})))\Volume(\UnipotentSubgroup_{m+1}(\integersRing)(1_{m+1}+\matrixRing_{m+1}(\mathfrak{p}))),
\end{multline*}
which is a non-zero constant. Appealing to \cite[Corollary 3.2]{Mat17} along with \cite[Theorem 3.6-(ii)]{Jo20} and \cite[Theorem 4.5]{Jo24}, the local factor $\BFLocalRhogammaFactor \in \multiplicativegroup{\cComplex[q^{\pm s},q^{\pm t}]}$ is a unit in $q^{-s}$ and $q^{-t}$, that is, a monomial of the form $\BFLocalRhogammaFactor=\alpha q^{-\beta s} q^{-\eta t}$, with $\alpha \in \cComplex$ and $\beta, \eta \in \mathbb{Z}$. This forces that all but one of summands in \eqref{BF-dual-side} must vanish. Among them, the only term which survives is
\[
 \contragredient{Z}(\repBesselFunction{\pi},\delta_{e_{m+1}})=\BFpigammaFactor Z(\repBesselFunction{\pi},\delta_{e_{m+1}})=\BFpigammaFactor.
\]
Combining all these calculations, we find that $ \BFLocalRhogammaFactor= \BFpigammaFactor$, as needed.

\par
We turn our attention to the case when $n=2m$ and $\pi$ does not have Friedberg-Jacquet vector. 
By taking the advantage of \Cref{JR-sum-characterization}, $Z(W_{\pi},\TrivialRepresentation_{\fFiniteField^m})=0$ for all $W_{\pi} \in \whittaker(\pi,\fieldCharacter)$. As before, \Cref{BF-modified-functional equation} simply turns into
\begin{equation}
\label{BF-reduction}
  \BFpigammaFactor Z(W_{\pi},\phi)=\contragredient{Z}(W_{\pi},\phi)=\BFLocalRhogammaFactor Z(W_{\pi},\phi).
\end{equation}
All that remains is to choose $W_{\pi}=\repBesselFunction{\pi}$ and $\phi=\delta_{e_m}$. In doing so, \Cref{key-distinction} guarantees that $Z(\repBesselFunction{\pi},\delta_{e_m})$ is precisely $1$, from which the equality $\BFpigammaFactor=\BFLocalRhogammaFactor$ follows.
\par
Suppose that $n=2m$ and $\pi$ admits a Friedberg--Jacquet vector.
Upon choosing $\phi=\TrivialRepresentation_{\fFiniteField^m}$, the relation $\fourierTransform(\TrivialRepresentation_{\fFiniteField^m})=q^{\frac{m}{2}}\delta_{0}$ implies that $Z(\contragredient{W}_{\pi},1_{\fFiniteField^m})=0$.
With aid of \Cref{JR-sum-characterization}, we take $W_{\pi} \in \whittaker(\pi,\fieldCharacter_\fLocalField)$ satisfying $Z(W_{\pi},\TrivialRepresentation_{\fFiniteField^m})=1$. In this way, we reduce \Cref{BF-modified-functional equation} to
\begin{multline*}
q^{-m(1-2s)+\frac{m}{2}}\omega^{-1}_{\rho}(\varpi)L(m(1-2s),\omega^{-1}_{\rho})\varepsilon(s,\rho,\fieldCharacter_\fLocalField)\\
 =\BFLocalRhogammaFactor(1+q^{-2ms}\omega_{\rho}(\varpi)L(2ms,\omega_{\rho}))=\BFLocalRhogammaFactor L(2ms,\omega_{\rho})
\end{multline*}
from which the required equality holds.
\end{proof}

We accomplish the following nice expression of Bump--Friedberg gamma factors $\BFpigammaFactor$ in terms of their Bessel functions.

\begin{proposition}
\label{BF-one-side-thm}
 Let $\pi$ be an irreducible cuspidal representation of $\GL_{n}(\fFiniteField)$. Then we have
\begin{equation}
\label{BF-one-side-thm-even}
 \BFpigammaFactor
 =q^{-\frac{m}{2}} \sum_{g \in \UnipotentSubgroup_m(\fFiniteField) \backslash \GL_m(\fFiniteField) } \sum_{g' \in \UnipotentSubgroup_m(\fFiniteField) \backslash \GL_m(\fFiniteField) } \repBesselFunction{\pi} \left( \sigma_{2m} \begin{pmatrix} & g' \\ g & \end{pmatrix} \sigma^{-1}_{2m} \right)    \fieldCharacter(e_1 {^tg'^{-1}}\;{^te_m})
\end{equation}
in the even case $n=2m$
  \[
 \BFpigammaFactor
 =q^{-\frac{m+1}{2}} \sum_{g \in \UnipotentSubgroup_{m+1}(\fFiniteField) \backslash \GL_{m+1}(\fFiniteField) } \sum_{g' \in \UnipotentSubgroup_{m}(\fFiniteField) \backslash \GL_{m}(\fFiniteField) } 
  \repBesselFunction{\pi}  (J(g,g'))  \fieldCharacter(e_1 {^tg^{-1}}\;{^te_{m+1}})
\]
in the odd case $n=2m+1$. In particular, we have $\BFDualpigammaFactor=\complexConjugate{ \BFpigammaFactor }$.
\end{proposition}

\begin{proof}
Just as in the proof of \Cref{flicker-one-side-thm}, we take $W_{\pi}=\grepBesselFunction{\pi}{\fieldCharacter}$ and $\phi$ to be an indicator function $\delta_{e_m}$ in the even case $n=2m$
and $\delta_{e_{m+1}}$ in the odd case $n=2m+1$, at which point \cite[Lemma 2.7]{Nie19} ensures that $Z(\grepBesselFunction{\pi}{\fieldCharacter},\delta_{e_m})=Z(\grepBesselFunction{\pi}{\fieldCharacter},\delta_{e_{m+1}})=1$. It remains to note that  $\fourierTransform(\delta_{e_m})(y)=q^{-\frac{m}{2}}\fieldCharacter(e_m {^ty})$ and 
$\fourierTransform(\delta_{e_{m+1}})(y)=q^{-\frac{m+1}{2}}\fieldCharacter(e_{m+1} {^ty})$.
\end{proof}

We precisely evaluate the sum \eqref{BF-one-side-thm-even}, when $\pi$ has the Friedberg--Jacquet vector.

\begin{proposition}
\label{BF-distinct-Gamma}
Let $\pi$ be an irreducible cuspidal representation of $\GL_{n}(\fFiniteField)$. Suppose that $n=2m$ and $\pi$ admits a Friedberg--Jacquet vector.
Then we have
\[
   \BFpigammaFactor = \BFDualpigammaFactor =-\varepsilon(s,\rho,\fieldCharacter_\fLocalField)q^{-\frac{m}{2}}.
\]
\end{proposition}

\begin{proof}
We insert the identity \eqref{BF-levelzero-poles} for $\BFLocalRhogammaFactor$. Next, we select $W_{\pi}=\repBesselFunction{\pi}$ and $\phi=\delta_{e_m}$ an indicator function on $e_m$, so that its Fourier transform
is given by $\fourierTransform(\delta_{e_m})(y)=q^{-\frac{m}{2}}\fieldCharacter(e_m {^ty})$. In this way, \Cref{BF-modified-functional equation} becomes
\begin{multline*}
 \BFpigammaFactor+q^{-m(1-2s)-\frac{m}{2}}\omega^{-1}_{\rho}(\varpi)L(m(1-2s),\omega^{-1}_{\rho})\varepsilon(s,\rho,\fieldCharacter_\fLocalField)Z(\repBesselFunction{\pi},\TrivialRepresentation_{\fFiniteField^m})\\
 =\varepsilon(s,\rho,\fieldCharacter_\fLocalField)q^{m\left(2s-\frac{1}{2}\right)}\omega^{-1}_{\rho}(\varpi) \frac{L(m(1-2s),\omega^{-1}_{\rho})}{L(2ms,\omega_{\rho})}.
 \end{multline*}
 We clear the denominator to express it as
\begin{multline*}
 \BFpigammaFactor-\omega^{-1}_{\rho}(\varpi)q^{-m(1-2s)}\BFpigammaFactor
 +q^{-m(1-2s)-\frac{m}{2}}\omega^{-1}_{\rho}(\varpi)\varepsilon(s,\rho,\fieldCharacter_\fLocalField)Z(\repBesselFunction{\pi},\TrivialRepresentation_{\fFiniteField^m})\\
 =\varepsilon(s,\rho,\fieldCharacter_\fLocalField)q^{m\left(2s-\frac{1}{2}\right)}\omega^{-1}_{\rho}(\varpi)-\varepsilon(s,\rho,\fieldCharacter_\fLocalField)q^{-\frac{m}{2}}.
 \end{multline*}
We compare the coefficients of constant terms and $q^{2ms}$ terms individually. In doing so, we arrive at a system of linear equations
 \[
  \begin{cases}
  \BFpigammaFactor=-\varepsilon(s,\rho,\fieldCharacter_\fLocalField)q^{-\frac{m}{2}} \\
  -\omega^{-1}_{\rho}(\varpi)q^{-m}\BFpigammaFactor+q^{-m-\frac{m}{2}}\omega^{-1}_{\rho}(\varpi)\varepsilon(s,\rho,\fieldCharacter_\fLocalField)Z(\repBesselFunction{\pi},\TrivialRepresentation_{\fFiniteField^m})
  =\varepsilon(s,\rho,\fieldCharacter_\fLocalField)q^{-\frac{m}{2}}\omega^{-1}_{\rho}(\varpi),
  \end{cases}
 \]
treating $\BFpigammaFactor$ and $Z(\repBesselFunction{\pi},\TrivialRepresentation_{\fFiniteField^m})$ as unknown variables, from which the equality $\BFpigammaFactor=-\varepsilon(s,\rho,\fieldCharacter_\fLocalField)q^{-\frac{m}{2}}$ shall follow. Having \Cref{BF-one-side-thm} in mind, all we have to do is to take the complex conjugate. We end up with 
 \[
 \BFDualpigammaFactor=\complexConjugate{ \BFpigammaFactor}=-\varepsilon(s,\rho,\fieldCharacter_\fLocalField)q^{-\frac{m}{2}}. \qedhere
\]
\end{proof}

We will shift our focus to the coincidence of arithmetic and analytic local factors, but beforehand we state functional equations for $\BFpigammaFactor$ over finite fields.

\begin{theorem}[Functional equation] 
\label{BF-Func}
Let $\pi$ be an irreducible cuspidal representation of $\GL_n(\fFiniteField)$.
 \begin{enumerate}[label=$(\mathrm{\arabic*})$]
\item\label{BF-Func1}
If $\pi$ does not admits a Friedberg--Jacquet vector, then we have
\[ 
\BFpigammaFactor \BFDualpigammaFactor = 1
\quad \text{and} \quad \abs{\BFpigammaFactor}=1.
  \]
\item\label{BF-Func2} If $n=2m$ and $\pi$ admits a Friedberg--Jacquet vector, then we have
 \[
 \BFpigammaFactor \BFDualpigammaFactor  =q^{-m} \quad \text{and} \quad
 \abs{\BFpigammaFactor}=q^{-\frac{m}{2}}. 
 \]
 \end{enumerate}
\end{theorem}

\begin{proof}
Upon invoking \Cref{dual-BF}, the functional equation in \ref{BF-Func1} can be seen from 
the double-duality
\[
 \contragredient{Z}( \contragredient{W}_{\pi},\ifourierTransform(\phi))=Z(W_{\pi},\phi).
\]
just as in \cite[Proposition 2.12]{YZ20}. In view of \Cref{central-epsilon}, we combine \Cref{BF-one-side-thm} and \Cref{BF-distinct-Gamma} to see the rest of results, and this ends the proof.
\end{proof}
In practice, \cite[Proposition 4.1]{Mat15} allows us to generalize the local function equation \eqref{Local-BF-Func} to
an irreducible generic representation $\rho$ and a spherical representation ${\rm Ind}_{B_n(F)}^{{\rm GL}_n(F)}(\mu_1 \otimes \dotsm \otimes \mu_n )$ at least for the Bump--Friedberg $\gamma$-factor with one variable $s$ $(t=0)$. We do not strive for maximal generality, so this hypothesis might be redundant, but which holds in all our applications. 

\begin{lemma}
\label{spherical-rep-BF}
Let $F$ be a local function field. Let $\rho$ be an irreducible subquotient of a spherical representation ${\rm Ind}_{B_n(F)}^{{\rm GL}_n(F)}(\mu_1 \otimes \dotsm \otimes \mu_n )$. Then we have
\begin{multline*}
\Gamma(s,\rho,{\rm BF},\fieldCharacter^{\flat}_\fLocalField)=\Gamma(s+1/2,\rho,\fieldCharacter^{\flat}_\fLocalField)\Gamma(2s,\rho,\wedge^2,\fieldCharacter^{\flat}_\fLocalField) \\
=\prod_{1 \leq i \leq n} \Gamma(s+1/2,\mu_i,\fieldCharacter^{\flat}_\fLocalField)\prod_{1\leq j < k \leq n} \Gamma(2s,\mu_j \times \mu_k,\fieldCharacter^{\flat}_\fLocalField).
\end{multline*}
\end{lemma}

\begin{proof}
The proof of Lemma \ref{spherical-rep-JS} literally goes through word by word except that we use the unramified computation of 
Bump and Friedberg \cite[Theorem 3]{BF89} in place of \cite[\S 2]{JS88} (cf. The reader may be referred to \cite[Proposition 1.2]{Ish18} for an alternative proof of unramified computations).
\end{proof}

As an intermediate step, we establish that Bump--Friedberg $\gamma$-factors agree with the counterpart Langlands--Shahidi gamma factors in positive characteristics. 
We extend the coincidence of two local factors to all characteristic, notably, zero in \Cref{BF-JS equal}.

\begin{theorem}
\label{main-BF-factor}
Let $\rho$ be a level zero supercuspidal representation of ${\rm GL}_n(F)$ over  a local function field $F$. Then we have
\[
  \Gamma(s,t,\rho,{\rm BF},\fieldCharacter_\fLocalField)=\Gamma(s+t+1/2,\rho,\fieldCharacter_\fLocalField)\Gamma_{\rm LS}(2s,\rho,\wedge^2,\fieldCharacter_\fLocalField)=\Gamma(s+t+1/2,\rho,\fieldCharacter_\fLocalField)\Gamma(2s,\rho,\wedge^2,\fieldCharacter_\fLocalField).
\]
\end{theorem}

\begin{proof} Putting together \Cref{levelzero-BF} and  \cite[\S 6.1]{BHK98}, $ \Gamma(s,t,\rho,{\rm BF},\fieldCharacter_\fLocalField)$ and $\Gamma(s+t+1/2,\rho,\fieldCharacter_\fLocalField)$ are independent of $t$,
so that we take $t$ to be $0$. With the help of \cite[\S 6.6-(vii)]{Lom16}, twists by unramified characters do not affect on the first equality. For this reason,
we may assume that $\pi$ is unitary without loss of generality. 
Applying Theorem \ref{Lomeli-Globalization} to the level zero supercuspidal representation, there are a global field $k$ with three places $v_0, v_1,$ and $v_{\infty}$
such that $k_{v_0} \cong F$, and an irreducible unitary cuspidal automorphic representation $\Pi$ of ${\rm GL}_n(\mathbb{A}_k)$ with the required properties in \Cref{Lomeli-Globalization}. 
We choose a non-trivial additive character $\Psi$ of $\mathbb{A}_k \slash k$, and assume, as we may, that $\Psi_{v_0}=\fieldCharacter_\fLocalField$.
The global functional equation for exterior square $L$-functions via the Langlands-Shahidi method can be read from \cite[\S4-(vi)]{HL11} as
\eqref{LS-global-exterior}, while that for standard $L$-functions due to Godement and Jacquet is extracted from \cite{GJ72} as
\begin{equation}
\label{partial-GJ}
  L^{S}(s,\Pi)=\Gamma(s,\Pi_{v_0},\Psi_{v_0}) \prod_{v \in S-\{ v_0\}} \Gamma(s,\Pi_v,\Psi_v)  L^{S}(1-s,\check{\Pi}).
\end{equation}
In the meantime, taking into account local functional equations \eqref{Local-BF-Func},
the global functional equation for Bump-Friedberg $L$-functions in \cite[Theorem 6]{BF89} (cf. \cite[(1.1)]{Ish18}) takes
the following explicit form:
\begin{multline}
\label{partial-BF}
  L^S(s+1/2,\Pi,\psi)L^S(2s,\Pi,\wedge^2)\\
  =\Gamma(s,\Pi_{v_0},{\rm BF},\Psi_{v_0}) \prod_{v \in S-\{ v_0\}} \Gamma(s,\Pi_v,{\rm BF},\Psi_v) L^S(1/2-s,\check{\Pi}) L^{S}(1-2s,\check{\Pi},\wedge^2).
\end{multline}
In accordance with Lemma \ref{spherical-rep-BF}, each places $v$ in $S-\{ v_0\}$ can be controlled in such a way that
\[
\Gamma(s,\Pi_v,{\rm BF},\Psi_v)=\Gamma(s+1/2,\Pi_v,\Psi_v)\Gamma_{\rm LS}(2s,\Pi_v,\wedge^2,\Psi_v).
\]
After plugging $s+1/2$ for $s$ in \eqref{partial-GJ}, the case for positive characteristics is at least done by dividing \eqref{partial-BF} by the product of \eqref{LS-global-exterior} and \eqref{partial-GJ}.
\end{proof}

The Bump--Friedberg gamma factor $ \Gamma(s,t,\rho,{\rm BF},\fieldCharacter_\fLocalField)$ is comparable with Kazhdan close field theory.
The proof is nearly identical to that of \Cref{Kaz-corres}, and so we shall be brief.

\begin{proposition} 
\label{close-field-BF}
For $(F,\rho,\psi)$ that is {\rm Kaz}-{\it associated to} $(F',\rho',\psi')$, we have
\[
    \Gamma(s,t,\rho,{\rm BF},\fieldCharacter_\fLocalField)=\Gamma(s,t,\rho',{\rm BF},\fieldCharacter_{\fLocalField'}).
    \]
\end{proposition}

\begin{proof}
As a consequence of \Cref{level-zero-RS}, we know that $\varepsilon(s,\rho,\fieldCharacter_\fLocalField)=\varepsilon(s,\rho',\fieldCharacter_{\fLocalField'})$. With ${\rm Vol}(\mathfrak{p}_F)={\rm Vol}(\mathfrak{p}_{F'})=q^{-1/2}$ and $\omega_{\rho}(\varpi_F)=\omega_{\rho'}(\varpi_{F'})$ in hand, \Cref{levelzero-BF} readily implies our assertion.
\end{proof}

We are now in a position to formulate the main factorization formula conjectured by Bump and Friedberg \cite[Conjecture 4]{BF89}.

\begin{theorem}
\label{BF-JS equal}
Let $\varphi$ be a $n$-dimensional tamely ramified representation of $W_\fLocalField$ corresponding to the level zero supercuspidal representation $\rho(\varphi)$ of $\GL_n(\fLocalField)$ via Macdonald correspondence. Then we have
\[
\BFlocalgammaFactor{s,t}{\rho(\varphi)}{\fieldCharacter_\fLocalField} =\varepsilon(s+t+1/2,\varphi,\fieldCharacter_\fLocalField)\Gamma(2s,\wedge^2( \varphi),\fieldCharacter_\fLocalField).
\]
\end{theorem}

\begin{proof}
Let $F$ be a non-archimedean local field of characteristic $0$ with its residue field $\mathfrak{o}_F / \mathfrak{p}_F $ isomorphic to $\fFiniteField_q$. Let $F'=\fFiniteField_q((t))$ so that $F$ and $F'$ are $1$-close.
At this point, we have given the detailed argument before in the proof of \Cref{equal-zero-exterior}. One may simply mimic the argument there, resting on \Cref{equal-zero-exterior}, \Cref{main-BF-factor}, \Cref{close-field-BF}, and a part of local Langlands correspondence \cite[$({\rm b})$ of Theorem 7.1]{Gan15}.
\end{proof}

%\begin{proof} The first equality is simply one of the main results in \cite[Theorems  4.3, 4.8]{Jo22}, relying on Deligne--Kazhdan close field theory. Alternatively,  as we directly compare \cite[Theorem 3.17]{YZ20} with \Cref{levelzero-BF}.
%\end{proof}

\subsection{The Bump--Friedberg epsilon factor and the Gauss sum}

The following elementary lemma illustrates how the standard $\varepsilon$-factors and $ \varepsilon_0$-factors are related by, but it does not seem to be recorded elsewhere. 
We take the occasion to provide a proof for completeness. 

\begin{lemma}
Let $\varphi$ be a $n$-dimensional tamely ramified representation of $W_\fLocalField$. Then we have
\label{standard-epsilon-eps0}
\[
  \varepsilon(s,\varphi,\fieldCharacter_\fLocalField)=\varepsilon_0(\varphi,\fieldCharacter_\fLocalField).
\]
\end{lemma}

\begin{proof}
As mentioned in the proof of \Cref{central-epsilon}, \cite[\S 6.1]{BHK98} asserts that $ \varepsilon(s,\varphi,\fieldCharacter_\fLocalField)=\varepsilon(s,\rho(\varphi),\fieldCharacter_\fLocalField)$
is a complex number. The identity \eqref{rel-epsilon-epsilon0} enforces that $V^{I_F}=\{ 0\}$. The result then follows from \cite[Corollary 2.7]{YZ21}.
\end{proof}

Our next task is to deduce the decomposition of Bump--Friedberg $\gamma$-factors, which we may think of as being the finite field analogue of \Cref{BF-JS equal}.

\begin{proposition} 
\label{BF-ext decomposition}
Let $\pi(\varphi)$ be an irreducible cuspidal representation of $\GL_{n}(\fFiniteField)$ associated to tamely ramified representation $\varphi$ of $W_F$ of degree $n$ via Macdonald correspondence. Then we have
\[
 \gamma(\pi(\varphi),{\rm BF},\psi)=\varepsilon_0(\varphi,\fieldCharacter_\fLocalField)\gamma(\pi(\varphi),\wedge^2,\fieldCharacter).
\]
\end{proposition}

\begin{proof}
We break it down into two cases. 
It is worth noting the equivalent statement that $\pi$ admits a Jacquet-Shalika vector if and only if 
it admits a Friedberg--Jacquet vector, which we defer to the next section (see \Cref{FJ-JS-Equiv}).
Suppose that $\pi$ does not admit a Friedberg--Jacquet vector. Owing to \Cref{levelzero-BF}, \Cref{BF-JS equal}, \Cref{standard-epsilon-eps0} together with \cite[Theorem 3.16]{YZ20}, 
we are guided to
\begin{multline*}
 \gamma(\pi(\varphi),{\rm BF},\psi)=\BFlocalgammaFactor{s}{\rho(\varphi)}{\fieldCharacter_\fLocalField}\\
 =\varepsilon(s+t+1/2,\varphi,\fieldCharacter_\fLocalField)\Gamma(2s,\wedge^2( \varphi),\fieldCharacter_\fLocalField)
 =\varepsilon_0(\varphi,\fieldCharacter_\fLocalField)\gamma(\pi(\varphi),\wedge^2,\fieldCharacter).
 \end{multline*}
 \par
It remains to deal with the case when $n=2m$ is even and $\pi$ admits a Friedberg--Jacquet vector. 
This case only requires a purely local approach avoiding globalization, as
the result is immediate from combining \Cref{JS-distinct-Gamma} and \Cref{BF-distinct-Gamma} with \Cref{standard-epsilon-eps0}.
\end{proof}

The following expression for $\varepsilon_0(\varphi,\fieldCharacter_\fLocalField)\varepsilon_0(\wedge^2 ( \varphi),\fieldCharacter_\fLocalField)$ in terms of Gauss sums is thought of as the Bump--Friedberg analogue of \Cref{Rankin-Selberg-Gauss} and \Cref{Asai-Gauss}.

\begin{theorem} 
Let $\pi$ be an irreducible cuspidal representation of $\GL_n(\fFiniteField)$. We let $\alpha \in  \multiplicativegroup{\widehat{\fFiniteField}_{q^{n}}}$ be a regular character 
corresponding to $\pi$ via Green's parametrization and $m=\lfloor \frac{n}{2} \rfloor$. Then we have
\[
\varepsilon_0(\varphi,\fieldCharacter_\fLocalField)\varepsilon_0(\wedge^2 ( \varphi),\fieldCharacter_\fLocalField)
=(-1)^{n+{n \choose 2}}q^{-\frac{1}{2}\left( n+{n \choose 2} \right)}\tau(\alpha,\psi_n)\tau(\alpha^{1+q^m},\psi_d)
\prod_{i=1}^{m-1}  \tau(\alpha^{1+q^{i}},\psi_{n}),
\]
where $d=n$ if $n$ is odd, and $d=m$ if $n$ is even.
\end{theorem}

\begin{proof}
When the second representation is the trivial one of $\GL_{1}(\fLocalField)$, $\TrivialRepresentation_{\multiplicativegroup{\fLocalField}}$, Rankin--Selberg $\gamma$-factor $\gamma(s,\pi(\varphi) \times \TrivialRepresentation_{\multiplicativegroup{\fLocalField}},\fieldCharacter_\fLocalField)$ degenerates into Godement--Jacquet $\varepsilon$-factor $\varepsilon(s,\pi(\varphi), \fieldCharacter_\fLocalField)$.
Now, \cite[Proposition 2.6]{YZ21} along with \cite[Theorems 4.3, 4.4]{YZ21} and \Cref{standard-epsilon-eps0} ensure that 
\[
\varepsilon_0(\varphi \otimes \TrivialRepresentation_{W_F},\fieldCharacter_\fLocalField)=\gamma(s,\pi(\varphi) \times \TrivialRepresentation_{\multiplicativegroup{\fLocalField}},\fieldCharacter_\fLocalField)=\varepsilon(s,\varphi, \fieldCharacter_\fLocalField)=\varepsilon_0(\varphi,\fieldCharacter_\fLocalField).
\]
Our claim is a direct consequence of \cite[Corollary 4.5]{YZ21} and \cite[Theorem 5.2]{YZ21}.
\end{proof}

The proof of \Cref{Artin-Asai-Flicker} should be compared with that of \Cref{Deligne-factorazation} below. Unlike the equality
$\Gamma(s,\rho(\varphi),{\rm As},\fieldCharacter_\fLocalField)
   =\omega^{n-1}_{\rho}(\delta) \lambda_{\eLocalField \slash \fLocalField}(\fieldCharacter_\fLocalField)^{-\frac{n(n-1)}{2}} \varepsilon(s,{\rm As}  ({\varphi}),\fieldCharacter_\fLocalField)$,
   which is independently settled in \cite{AKMSS21} and \cite{BP21} for irreducible supercuspidal representations $\rho$, the corresponding equality for exterior square  $\gamma$-factors 
   has been less developed. We use Deligne--Kazhdan close field theory to deduce the required identity for level zero supercuspidal representations $\rho$, which is enough for applications.

\begin{proposition} 
\label{Deligne-factorazation}
Let $\pi(\varphi)$ be an irreducible cuspidal representation of $\GL_{n}(\fFiniteField)$ associated to tamely ramified representation $\varphi$ of $W_F$ of degree $n$ via Macdonald correspondence.  Then we have
\[
\gamma(\pi(\varphi),\wedge^2,\fieldCharacter)=\varepsilon_0(\wedge^2 (\varphi),\fieldCharacter_\fLocalField)\quad \text{and} \quad \gamma(\pi(\varphi),{\rm BF},\psi)=\varepsilon_0(\varphi,\fieldCharacter_\fLocalField)\varepsilon_0(\wedge^2 ( \varphi),\fieldCharacter_\fLocalField).
\]
\end{proposition}

\begin{proof}
We separate it into two cases. Suppose that $\pi$ does not admits a Jacquet--Shalika vector. 
We use \Cref{equal-zero-exterior} in conjunction with \cite[Theorem 3.16]{YZ20} and \cite[Corollary 2.7]{YZ21} in order to see that
\[
\gamma(\pi(\varphi),\wedge^2,\fieldCharacter)=\Gamma(s,\rho(\varphi),\wedge^2,\fieldCharacter_\fLocalField)=\varepsilon(s,\wedge^2( \varphi),\fieldCharacter_\fLocalField)=
\varepsilon_0(\wedge^2 (\varphi),\fieldCharacter_\fLocalField).
\]
We now handle the remaining case when $n=2m$ is even and $\pi$ admits a Jacquet-Shalika vector. The central character $\omega_{\pi}=\alpha\restriction_{\multiplicativegroup{\fFiniteField}}$
becomes trivial so that $\alpha^{1+q^{m}}=\TrivialRepresentation$. Since $\left(\alpha^{1+q^{i}}\right)^{1+q^m}=\TrivialRepresentation$ and $\alpha^{1+q^i}$ is not trivial for $0 \leq i \leq m-1$, we invoke \cite[Proposition A.2]{Ze22} to get $ \tau(\alpha^{1+q^{i}},\psi_{n})=-q^{m}$. In light of \cite[Proposition 2.6]{YZ21}, we conclude that
\[
\varepsilon_0(\wedge^2 ( \varphi),\fieldCharacter_\fLocalField)=(-1)^{{2m \choose 2}}q^{-\frac{1}{2}{2m \choose 2}}\tau(\alpha^{1+q^m},\psi_{m})
\prod_{i=1}^{m-1}  \tau(\alpha^{1+q^{i}},\psi_{2m})=-q^{-\frac{1}{2}{2m \choose 2}}\cdot q^{m(m-1)} \tau(\TrivialRepresentation,\psi_{m}),
\]
which agrees with $\gamma(\pi(\varphi),\wedge^2,\fieldCharacter)=-q^{-\frac{m}{2}}$ in \Cref{JS-distinct-Gamma}, having used the fact that $\tau(\TrivialRepresentation,\psi_{m})=1$.
Then the second equality can be justified from \Cref{BF-ext decomposition}.  
\end{proof}

We are prepared to gain a product formula for $ \gamma(\pi,\wedge^2,\fieldCharacter)$ and $\gamma(\pi,{\rm BF},\psi)$ with regard to Gauss sums.

\begin{theorem}[Gauss sum]
\label{Exterior-Gamma-Gauss}
Let $\pi$ be an irreducible cuspidal representation of $\GL_n(\fFiniteField)$. We let $\alpha \in  \multiplicativegroup{\widehat{\fFiniteField}_{q^{n}}}$ be a regular character
corresponding to $\pi$ via Green's parametrization and $m=\lfloor \frac{n}{2} \rfloor$. Then we have
\[
 \gamma(\pi,\wedge^2,\fieldCharacter)=(-1)^{{n \choose 2}}q^{-\frac{1}{2}{n \choose 2}} \tau(\alpha,\psi_n)\tau(\alpha^{1+q^m},\psi_d)
\prod_{i=1}^{m-1}  \tau(\alpha^{1+q^{i}},\psi_{n})
\]
and 
\[
 \gamma(\pi,{\rm BF},\psi)=(-1)^{n+{n \choose 2}}q^{-\frac{1}{2}\left( n+{n \choose 2} \right)}\tau(\alpha,\psi_n)\tau(\alpha^{1+q^m},\psi_d)
\prod_{i=1}^{m-1}  \tau(\alpha^{1+q^{i}},\psi_{n}),
\]
where $d=n$ if $n$ is odd, and $d=m$ if $n$ is even.
\end{theorem}

\section{Period Vectors and Distinctions}
\label{period}

In this section, we study the period vectors and integrals for four pairs of groups $(G,L)$: Jacquet--Piatetski-Shapiro--Shalika period, Flicker--Rallis period, Friedberg--Jacquet period, and 
 Jacquet--Shalika period. Let $\sigma$ be a level zero supercuspidal representation of $G(E)$ coming from an irreducible cuspidal representation $\Pi$ of $G(\eFiniteField)$.
The group $G$ and its closed subgroups $L$ and $N$, as well as its representations $\Pi$ and $\sigma$, are given by the following table:

%\begin{table}[ht]
\begin{center}
\begin{tabular}{ |c|c|c|c|c|c|c|c|c|c| } 
 \hline 
 Period Vectors & $G(\eFiniteField)$ & $G(E)$ & $L$  & $U$ & $\Pi$ & $\sigma$ & $r$     \\ 
  \hline   \hline 
 \makecell{Jacquet\\--Piatetski-Shapiro\\ --Shalika} & $ \makecell{ {\rm GL}_n(\fFiniteField) \\  \times {\rm GL}_n(\fFiniteField)}$  & $ \makecell{ {\rm GL}_n(F) \\ \times {\rm GL}_n(F)}$ &${\rm GL}_n$ & $\UnipotentSubgroup_n $ & $\pi_1 \times \pi_2$  & $\rho_1 \times \rho_2$ & - \\ 
  \hline
 Flicker--Rallis & ${\rm GL}_{2m+1}(\eFiniteField)$ & ${\rm GL}_{2m+1}(E)$ &${\rm GL}_{2m+1}$& $N_{2m+1}$ & $\pi$  & $\rho$& As \\ 
 \hline
 Friedberg--Jacquet & ${\rm GL}_{2m}(\fFiniteField)$  & ${\rm GL}_{2m}(F)$ & $H_{2m}$  & $\UnipotentSubgroup_m \times \UnipotentSubgroup_m$ & $\pi$ & $\rho $ & BF \\ 
 \hline
  Jacquet--Shalika & ${\rm GL}_{2m}(\fFiniteField)$ &${\rm GL}_{2m}(F)$ &$S_{2m}$ & $\UnipotentSubgroup_m \times \nilpotentMatrices_m$ & $\pi$ & $\rho$ & $\wedge^2$  \\ 
 \hline
\end{tabular}
\end{center}
%\caption{Groups and Characters}
%\end{table}

Given a pair of representations $\sigma$ and $\Pi$, their corresponding central characters $\omega_{\Pi}$ and $\omega_{\sigma}$, and their associated Whittaker models $\whittaker(\Pi,\qfFieldCharacter)$ and $\whittaker(\sigma,\qFieldCharacter)$, in addition to the character $\Xi$ of $L$, are highlighted in the following table:
\begin{center}
\begin{tabular}{ |c|c|c|c|c|c|c|c|c|c|c|c| } 
 \hline 
  Period Vectors & $\omega_{\Pi}$ & $\omega_{\sigma}$  & $\Xi$  & $\nu^{\alpha s_0}$ &  $\whittaker(\Pi,\qfFieldCharacter)$  & $\whittaker(\sigma,\qFieldCharacter)$ &  $q^{-\beta}$  \\ 
  \hline   \hline 
 \makecell{Jacquet\\--Piatetski-Shapiro\\ --Shalika} 
 &  $\omega_{\pi_1}\omega_{\pi_2}$ & $\omega_{\rho_1}\omega_{\rho_2}$ &   $\TrivialRepresentation_L$ &$\nu^{s_0}$  & \makecell{$\whittaker(\pi_1,\fieldCharacter^{-1})$ \\ $\otimes \whittaker(\pi_2,\fieldCharacter)$} &\makecell{$\whittaker(\rho_1,\fieldCharacter_\fLocalField^{-1})$ \\ $\otimes \whittaker(\rho_2,\fieldCharacter_\fLocalField)$}& $q^{-n/2}$ \\ 
  \hline
 Flicker--Rallis & $\omega_{\pi}$ & $\omega_{\rho}$ &  $\TrivialRepresentation_L$&$\nu^{s_0}$   & $\whittaker(\pi,\qfFieldCharacter)$  & $\whittaker(\rho,\qFieldCharacter)$ &$q^{-n/2}$  \\ 
 \hline
 Friedberg--Jacquet & $\omega_{\pi}$ &$\omega_{\rho}$  & $\TrivialRepresentation_L$ & $\nu^{s_0}$   & $\whittaker(\pi,\fieldCharacter)$  &$\whittaker(\rho,\fieldCharacter_\fLocalField)$& $q^{-m/2}$     \\ 
 \hline
  Jacquet--Shalika & $\omega_{\pi}$ & $\omega_{\rho}$ & $\Theta$ & $\nu^{s_0/2}$  & $\whittaker(\pi,\fieldCharacter)$ &$\whittaker(\rho,\fieldCharacter_\fLocalField)$ &$q^{-m/2}$   \\ 
 \hline
\end{tabular}
\end{center}

For each four period vectors, we prove a relation between the period integrals and sums, and the $L$-factors $L(s,\sigma,r)$ and $\gamma$-factors $\gamma(\Pi,r,\fieldCharacter)$.
The local factors that show up in this section include the Rankin--Selberg factors $L(s,\rho_1 \times \rho_2)$ and $\gamma^{\star}(\pi_1\times \pi_2,\fieldCharacter)$, the Asai factors 
$L(s,\rho,{\rm As})$ and $\AsaiPiGammaFactor$, the Bump--Friedberg factors $L(s,\rho,{\rm BF})$ and $\BFpigammaFactor$, and the exterior square factors $L(s,\rho,\wedge^2)$ and $\gamma(\pi,\wedge^2,\fieldCharacter)$. 

\begin{theorem}[Finite period vectors and integrals] 
\label{Period-Vector-Integral}
When $\Pi=\pi_1 \times \pi_2$,  we let $\gamma(\Pi,r,\fieldCharacter)$ denote $\gamma^{\star}(\pi_1\times \pi_2,\fieldCharacter)$. With the above datum, the following statements are equivalent:
\begin{enumerate}[label=$(\mathrm{\arabic*})$]
\item\label{CharPeriod1} $\Pi$ admits a non-zero period vector,  i.e., there exists a non-zero vector $v \in V_{\Pi}$
such that $\Pi(g)v=v$ for all $g \in L$.
\item\label{CharPeriod2} $\omega_{\Pi}\restriction_{\multiplicativegroup{\fFiniteField}}= \TrivialRepresentation_{\multiplicativegroup{\fFiniteField}}$.
\item\label{CharPeriod3} There exists $W_{\Pi}  \in \whittaker(\Pi,\qfFieldCharacter)$ such that
\[
 \sum_{g \in U(\fFiniteField) \backslash L(\fFiniteField) } W_{\Pi}(g) \neq 0.
\]
\item\label{CharPeriod4}  $\abs{\gamma(\Pi,r,\fieldCharacter)}=q^{-\beta}$.
\item\label{CharPeriod5}  $\sigma$ admits a nontrivial twisted period, i.e.,  $\Hom_{L(\fLocalField)}(\sigma\nu^{\alpha s_0}\restriction_{L(\fLocalField)}, \Xi) \neq 0$ for some $s_0 \in \mathbb{C}$.
\item\label{CharPeriod6} $\omega_{\sigma}\restriction_{\multiplicativegroup{\fLocalField}}=\nu^{-\alpha s_0}\restriction_{\multiplicativegroup{\fLocalField}}$.
\item\label{CharPeriod7} There exists $W_{\sigma}  \in \whittaker(\sigma,\qFieldCharacter)$ such that the integral
\begin{equation}
\label{residual-integrals}
  \int_{\multiplicativegroup{\fLocalField}U(\fLocalField) \backslash L(\fLocalField)} W_{\sigma}(g) \nu^{\alpha s_0}(g) \, dg 
\end{equation}
is well-defined and non-vanishing for some $s_0 \in \cComplex$.
\item\label{CharPeriod8}  $L(s,\sigma,r)$ has a pole at $s=s_0$.
\end{enumerate}
\end{theorem}

In order to keep our exposition of the proof neat and concise, we separate it into two parts.
We deal with the equivalent assertions for finite field cases. 

\begin{proof}[Proof of the equivalence of \ref{CharPeriod1},\ref{CharPeriod3}, \ref{CharPeriod4}, and \ref{CharPeriod8}]
The equivalence of \ref{CharPeriod1} $\&$ \ref{CharPeriod3} has been explained in \Cref{FR-sum-characterization},  \Cref{JR-sum-characterization}, \cite[Lemma 4.2]{Ye19}, and \cite[Theorem 2.30]{YZ20}. 
The equivalence of \ref{CharPeriod1} $\&$ \ref{CharPeriod4} is a direct consequence of functional equations: \Cref{RS-Func}, \Cref{Asai-Func}, \Cref{Exterior-Func}, and \Cref{BF-Func}.
The equivalence of \ref{CharPeriod1} $\&$ \ref{CharPeriod8} is just a summary of \Cref{levelzero-RS}, \Cref{levelzero-Asai}, \Cref{levelzero-BF}, and \cite[Theorems 3.16 and 3.17]{YZ20}.
\end{proof}

We treat the parallel statements for non-archimedean local field cases. 

\begin{proof}[Proof of the equivalence of \ref{CharPeriod2},\ref{CharPeriod5},\ref{CharPeriod6}, \ref{CharPeriod7}, and  \ref{CharPeriod8}]
The equivalent statement about central characters \ref{CharPeriod2} $\&$ \ref{CharPeriod6} is clear from the fact that  $\omega_{\Pi}\restriction_{\multiplicativegroup{ \integersRing}}=
\omega_{\sigma} \circ \projection \restriction_{\multiplicativegroup{ \integersRing}}$. The equivalence of \ref{CharPeriod5} $\&$ \ref{CharPeriod8} is simply a restatement of \cite[Proposition 3.4]{Jo20}, \cite[Theorem 2.1]{Mat10}, and \cite[Proposition 4.6, Corollary 4.3]{Mat15}.
Since $L(s,\sigma,r)$ and $L(1-s,\contragredient{\sigma},r)$ does not share any common poles,
the equivalence of \ref{CharPeriod6} $\&$ \ref{CharPeriod8} can be seen from \Cref{levelzero-RS}, \Cref{levelzero-Asai}, \Cref{levelzero-BF}, and \cite[Theorems 3.16 and 3.17]{YZ20}.

\par
The proof of equivalent statements  \ref{CharPeriod7} $\&$ \ref{CharPeriod8} needs to be managed more carefully. The absolute convergence of the integral \eqref{residual-integrals} can be justified from \cite[Proposition 2.4]{Mat15}. We omit the detail, since its proof is standard \cite[Lemmata 4.1 and 4.2]{Kew11}. Taking care of the characterization of poles in terms of residual integrals  \eqref{residual-integrals}, the proof of \cite[Theorem 4.3]{Kew11} which is originated from \cite[\S 2.6.1]{CPS17}  carries over verbatim to the setting of the Flicker--Rallis and Friedberg--Jacquet periods.
The main point is to think of \eqref{residual-integrals} as a constant multiple of the leading coefficient in the Laurent expansion of Jacquet--Piatetski-Shapiro--Shalika integral $\Psi(s,W_{\rho_1},W_{\rho_2},\Phi)$, Flicker integral $I(s,W_{\rho},\Phi)$,  
Friedberg--Jacquet integral $Z(s,W_{\rho},\Phi)$, and Jacquet--Shalika integrals $J(s,W_{\rho},\Phi)$ at $s=s_0$, respectively.
\end{proof}

When $\Pi=\pi_1 \times \pi_2$, our result can be regarded as a special case of \cite[Theorem 1.3]{SZ23} for $d_{\pi}(\sigma)=1$.
A Friedberg--Jacquet vector and a Jacquet--Shalika vector are related in the following way:

\begin{corollary}[Equivalent Periods] 
\label{FJ-JS-Equiv}
Let $\rho$ be level zero supercuspidal representations of $\GL_{2m}(F)$
 constructed from irreducible cuspidal representations $\pi$ of  $\GL_{2m}(\fFiniteField)$.
 The following statements are equivalent:
\begin{enumerate}[label=$(\mathrm{\arabic*})$]
\item $\pi$ admits a Friedberg--Jacquet vector.
\item $\pi$ admits a Jacquet--Shalika vector.
\item $\rho\nu^{s_0}$ is $H_{2m}(\fLocalField)$-distinguished for some $s_0 \in \mathbb{C}$.
\item $\rho\nu^{s_0}$ is $(S_{2m}(F),\Theta)$-distinguished for some $s_0 \in \mathbb{C}$.
\end{enumerate}
\end{corollary}

\begin{proof}
The following equivalent statements can be read from the proof of \cite[Theorem 4.1]{Jo24}, which has its origin in the work of Matringe \cite[Proposition 2.1]{Mat17}:
\begin{enumerate}[label=$(\mathrm{\roman*})$]
\item $L(2s,\rho,\wedge^2)$ has a pole at $s=s_0$;
\item $L(s,\rho,{\rm BF})$ has a pole at $s=s_0$;
\item $\rho\nu^{s_0}$ is $(S_{2m}(F),\Theta)$-distinguished;
\item $\rho\nu^{s_0}$ is $H_{2m}(F)$-distinguished.
\end{enumerate}
All we need to do at this point is to look back on \Cref{Period-Vector-Integral} for Friedberg--Jacquet and Jacquet--Shalika periods.
\end{proof}

\begin{acknowledgements} The author is deeply indebted to Elad Zelingher for sending a proof of \Cref{Asai-Gauss} to us and kindly allowing us to reproduce it here. 
We would like to thank Rongqing Ye for drawing the attention to the equality of exterior square gamma factors in author's thesis, and Andrew Knightly and Gilbert Moss for many fruitful discussions. 
Thanks are also owed to David Schwein for elaborating the close field theory in Harish-Chandra learning seminar, where the author first learned the topic. 
Finally the author would like to thank to anonymous referees for a number of constructive comments, which improved the exposition of this paper.
This work was supported by the National Research Foundation of Korea (NRF) grant 
funded by the Korea government (No. RS-2023-00209992). 
\end{acknowledgements}

\begin{data availability}
This manuscript has no associated data.
\end{data availability}

\begin{conflicts of interest}
The author states that there is no conflict of interest.
\end{conflicts of interest}

 \bibliographystyle{amsplain}
 \bibliography{references}

\end{document}